\documentclass[a4paper,10pt]{article}

\usepackage{amsthm}  
\usepackage{amsmath} 
\usepackage{amssymb}
\usepackage{bm} 
\usepackage{relsize} 
\usepackage{algpseudocode} 
\usepackage{tikz} 
\usepackage{verbatim}
\usepackage{hyperref}
\usetikzlibrary{arrows}

\theoremstyle{plain}
\newtheorem{theorem}{Theorem}[section]
\newtheorem{lemma}[theorem]{Lemma}

\newtheorem{proposition}[theorem]{Proposition}
\theoremstyle{definition}
\newtheorem{definition}[theorem]{Definition}

\newtheorem{example}[theorem]{Example}

\newcommand{\bbalpha}{{\mathlarger{\mathlarger{\bm{\alpha}}}}}
\newcommand{\bblet}[1]{\mathlarger{\mathlarger{\bm{#1}}}}
\newcommand{\blet}[1]{\mathlarger{\bm{#1}}}

\newcommand{\blue}[1]{\textcolor{black}{#1}}

\def\fn{Figure} %

\begin{document}

\title{Uniform One-Dimensional Fragment over Ordered Structures}

\author{Jonne Iso-Tuisku\footnote{University of Tampere, Finland}, Antti Kuusisto\footnote{Tampere
University of Technology, Finland, and University of Bremen, Germany}}

\date{}

\maketitle

\begin{abstract}
\noindent
The uniform one-dimensional fragment $\mathrm{U}_1$ is a recently introduced extension of the two-variable fragment $\mathrm{FO}^2$. 
The logic $\mathrm{U}_1$ enables the use of relation symbols of all arities 
and thereby extends the scope of applications of $\mathrm{FO^2}$. 
In this article we show that the satisfiability 
and finite satisfiability problems of $\mathrm{U}_1$ over linearly ordered models are \textsc{NExpTime}-complete. 
The corresponding problems for $\mathrm{FO}^2$ are likewise \textsc{NExpTime}-complete, 
so the transition from $\mathrm{FO}^2$ to $\mathrm{U}_1$ in the ordered realm causes no increase in complexity. 
To contrast our results, we also establish that $\mathrm{U}_1$ with an unrestricted use of two built-in linear orders is undecidable.
\end{abstract}

\section{Introduction}

Two-variable logic $\mathrm{FO}^2$ and the
guarded fragment $\mathrm{GF}$ are currently probably the most
widely studied sublogics of first-order logic. 
Two-variable logic was shown \textsc{NExpTime}-complete in
\cite{DBLP:journals/bsl/GradelKV97} and the guarded fragment 
\textsc{2NExpTime}-complete in \cite{GradelR99}.
The guarded-fragment was extended in \cite{barany} to 
the guarded negation fragment $\mathrm{GNFO}$, and it was likewise
shown to be \textsc{2NExpTime}-complete in the same article. The recent paper
\cite{DBLP:conf/aiml/HellaK14} introduced the \emph{uniform one-dimensional
fragment} $\mathrm{U}_1$ as an extension of $\mathrm{FO}^2$,
and $\mathrm{U}_1$ was shown \textsc{NExpTime}-complete in
\cite{DBLP:conf/mfcs/KieronskiK14}. 
The uniform one-dimensional fragment extends $\mathrm{FO}^2$ in a
canonical way to contexts with relation symbols of all arities.
Indeed, $\mathrm{FO}^2$ does not cope well with higher-arity relations,
and this \emph{limits the scope of related research} in relation to 
potential  applications. In database theory contexts, for example,
this limitation can be crucial. The idea of $\mathrm{U}_1$ is
based on the notions of \emph{one-dimensionality} and \emph{uniformity}.
One-dimensionality means that quantification is limited to blocks of
existential (universal) quantifiers that leave at most one variable free. Uniformity
asserts that a Boolean combination of higher arity (i.e., at least binary) atoms
$R(x_1,...\, ,x_k)$ and $S(y_1,...\, ,y_n)$ is allowed only if 
$\{x_1,...\, ,x_k\} = \{y_1,...\, ,y_n\}$; see Section \ref{preliminariessection}
for the formal definition.
The article \cite{DBLP:conf/aiml/HellaK14} shows that if
either of these two restrictions---one-dimensionality or
uniformity---is lifted in a canonically minimal way, the resulting 
logic is undecidable.
While $\mathrm{U}_1$ has the same 
complexity as $\mathrm{FO}^2$, it was established in 
\cite{DBLP:conf/aiml/HellaK14,DBLP:conf/dlog/Kuusisto16}
that it is incomparable in expressivity with both $\mathrm{GNFO}$
and two-variable logic with counting $\mathrm{FOC}^2$.
Another interesting result on expressivity was
obtained in \cite{DBLP:conf/dlog/Kuusisto16},
which showed that if the uniformity condition is applied also to the equality symbols in
addition to other relation symbols, the resulting logic is
equi-expressive with $\mathrm{FO}^2$ over models 
with at most binary relations. While standard $\mathrm{U}_1$ is
more expressive than $\mathrm{FO}^2$ even over empty vocabularies,
this equi-expressivity result can be taken as an argument for
canonicity of $\mathrm{U}_1$ as an extension of $\mathrm{FO}^2$.
Recent results on $\mathrm{U}_1$ include, e.g., \cite{KuusistoL18},
which shows that the symmetric model counting problem for formulae of $\mathrm{U}_1$ is
computable in polynomial time. See also \cite{DBLP:conf/csl/KieronskiK15}, which gives an Ehrenfeucht-Fra\"{i}ss\'{e}
game characterization of $\mathrm{U}_1$.
Research on two variable logics has been very active in recent years,
see, e.g., \cite{james16,mikolaj11,Charatonik016a,
DBLP:conf/csl/Kieronski16,amaldevi,lidia14,zeume}.
The focus has been especially on questions related to
built-in relations and operators that increase the
expressivity of the base language.
%
%
%
%
%
%
%
%
%
%
The significance of $\mathrm{FO}^2$ from the point of view of
applications is at least partially explained by its close links to XML 
\blue{(in this connection, see e.g. \cite{DBLP:journals/jacm/BojanczykMSS09})}
as
well as the fact that typical systems of modal logic translate
into $\mathrm{FO}^2$. The  link with modal logic makes
investigations on two-variable logic relevant to various different
fields from verification and temporal logic to 
knowledge representation, description logics, and
even distributed computing. For more or less recent works
relating to links between logic and
distributed computing, see, e.g., \cite{distco15, gand2014, Kuusisto13, 
Reiter15, Kuu8}
In this article we obtain the first results on $\mathrm{U}_1$ 
over models with a built-in linear order \blue{(abbr. $\mathrm{U}_{1}(<)$)}. 
We prove three related complexity results: we show
that $\mathrm{U}_1(<)$ is \textsc{NExpTime}-complete
over linearly ordered models, well-ordered models and finite linearly ordered models.
\blue{In comparison with the complexity of $\mathrm{FOC}^2$ with a linear order,
which was shown to be \textsc{VAS}-complete in \cite{Charatonik016a},
$\mathrm{U}_1$ with a linear order is more efficient.}
Furthermore, we show undecidability of $\mathrm{U}_1$ with two
linear orders that can be used freely, i.e., the uniformity condition
does not apply to the linear orders. This result contrasts 
with the result of Zeume and Harwath \cite{zeume} showing that $\mathrm{FO}^2$ in the finite case
with two linear orders is decidable and in fact in \textsc{2NExpTime}.
%
%
Also, $\mathrm{U}_1$ with two freely usable built-in equivalence
relations has been shown \textsc{2NExpTime}-complete in
\cite{DBLP:conf/csl/KieronskiK15}. 
We note that $\mathrm{FO}^2$ with a single linear order was
shown $\textsc{NExpTime}$-complete by
Otto already in \cite{DBLP:journals/jsyml/Otto01}.
The proofs below use fresh methods 
together with notions from \cite{DBLP:conf/mfcs/KieronskiK14} 
and \cite{DBLP:journals/jsyml/Otto01}.
Our objectives in this article are two-fold. Firstly, we 
provide the first results on $\mathrm{U}_1$ over models
with a built-in linear order. Secondly, we wish to promote $\mathrm{U}_1$ as a
potential framework for \emph{expanding the scope} of the active research programme on
two-variable logics to the \emph{context of higher arity relations}. From the
point of view of applications, inter alia (and especially) database theory, the
restriction to binary relations can indeed be undesirable. 
The article 
\cite{DBLP:conf/dlog/Kuusisto16} contains a brief survey on $\mathrm{U}_1$
and argues how $\mathrm{U}_1$ can serve as a
framework for building (expressive) $n$-ary description logics.

\blue{The structure of the paper is the following.
The next section, Section \ref{preliminariessection}, introduces general concepts and notations used throughout this paper.
Section \ref{Analysing} also introduces concepts, but they are 
very specific, technical, and mostly used in Section \ref{reducing}.
It is recommended to use
Section \ref{Analysing} as a reference while reading Section \ref{reducing}
which presents the main results of this paper.
Section \ref{undecidableextensions}  presents the undecidability result, forming a contrast with the main results.
The results of this paper have appeared in the thesis \cite{JITthesis}}.

\section{Preliminaries} \label{preliminariessection}

\blue{First we introduce general concepts used in this paper.
Then we introduce the uniform one-dimensional fragment  $\mathrm{U}_1$ of first-order logic
and related concepts.}

We let $\mathbb{Z}_+$ denote the set of positive integers.
If $f$ is a function with a domain $S$,
we define $\mathit{img}(f) := \{\, f(s)\, |\, s\in S\,\}$.
An \emph{ordered set} is a structure $(A,<)$ where $A$ is a set
and $<$ a linear order on $A$.
We call a subset $I$ of $A$ an \textit{interval}
if for all $a,c \in I$ and all $b \in A$,
it holds that if $a < b < c$, then $b \in I$.
A \emph{permutation of a tuple} $(u_1,...\, , u_k)$ is a
tuple $(u_{f(1)},...\, , u_{f(k)})$ for some bijection $f:\{1,...\, , k\}
\rightarrow \{1,...\, ,k\}$. A \emph{trivial tuple} is a tuple $(u_1,...\, , u_k)$
such that $u_i = u_j$ for all $i,j\in\{1,...\, , k\}$.
We let $\mathrm{VAR}$ denote the 
set $\{v_1, v_2,\, ...\, \}$ of first-order variable symbols.
We mostly use metavariables $x,y,z,x_1,y_1,z_1,$ etc.,  to
denote the variables in $\mathrm{VAR}$. Note that
for example the metavariables $x$ and $y$ may denote the
same variable symbol $v_i$, while $v_i$ and $v_j$ for $i\not= j$
are always different symbols.
Let $R$ be a $k$-ary relation symbol.
An atomic formula $Rx_{1}...x_{k}$ is
called an $X$-atom if $X = \{x_{1},...\, , x_{k}\}$.
For example, if $x,y,z$ are distinct variables,
then $Syx$ and $Rxyxxy$ are $\{x,y\}$-atoms while $Px$ and $Txzy$
are not. $Txyz$ and $Syyxz$ are $\{x,y,z\}$-atoms.
For technical reasons, atoms $x=y$ with an 
equality symbol are \emph{not} $\{x,y\}$-atoms.
%
%
%
%
%
%
%
%
%
%
%
%
%
%
%
%
%
%

%
%
%
Let $\tau$ be a relational vocabulary.
A $k$-ary $\tau$-atom is an atomic $\tau$-formula
that mentions exactly $k$ variables:
for example, if $x,y,z$ are distinct variables and $R,T\in\tau$
relation symbols with arities $5$ and $3$, respectively, then
the atoms $Txxy$ and $x= y$ are binary $\tau$-atoms and $Rxxyzx$ and $Txyz$
ternary $\tau$-atoms. If $P,S\in\tau$ are relation symbols of
arities $1$ and $2$, respectively, then $Px$ and $x=x$ are unary $\tau$-atoms
and $Sxy$ a binary $\tau$-atom.
Let $\tau_m$ denote a countably infinite relational vocabulary
in which every relation symbol is of the arity $m$.
Let $\mathcal{V}$ be \textit{a complete relational vocabulary},
that is $\mathcal{V} = \bigcup_{m \in \mathbb{Z_+}}\tau_m$.
In this paper we consider models and logics with relation symbols only;
function and constant symbols will not be considered. (The identity
symbol is considered a logical constant and is therefore not a relation symbol.)
We denote models by $\mathfrak{A}$, $\mathfrak{B}$, et cetera. 
The domain of these models is then denoted by $A$ and $B$, respectively.
If $\tau$ is a vocabulary, then a $\tau$-model interprets all the
relation symbols in $\tau$ and no other relation symbols.
A $\tau$-formula is a formula 
whose relation symbols are contained in $\tau$.
If $\mathfrak{A}$ is a $\tau$-model
and $\mathfrak{B}$ a $\tau'$-model such that $\tau\subseteq\tau'$
and $\mathfrak{A} = \mathfrak{B}\upharpoonright\tau$,
then $\mathfrak{B}$ is an \emph{expansion} of $\mathfrak{A}$ and $\mathfrak{A}$ is
the \emph{$\tau$-reduct} of $\mathfrak{B}$.
The notion of a \emph{substructure} is
defined in the usual way, and if $\mathfrak{A}$ is a substructure of $\mathfrak{B}$ 
\blue{(written: $\mathfrak{A} \subseteq \mathfrak{B}$)},
then $\mathfrak{B}$ is an \emph{extension} of $\mathfrak{A}$.
Consider a vocabulary $\tau \subseteq \mathcal{V}$.
The set of $\tau$-formulae of the \emph{equality-free 
uniform one-dimensional fragment} $\mathrm{U}_1(\mathit{no}\hspace{-1mm}=)$ is 
the smallest set $\mathcal{F}$ such that the following conditions hold.
\begin{enumerate}
  \item Every unary $\tau$-atom is in $\mathcal{F}$.
  \item If $\varphi \in \mathcal{F}$, then $\neg \varphi \in \mathcal{F}$.  
  \item If $\varphi, \psi \in \mathcal{F}$, then $(\varphi \wedge \psi) \in \mathcal{F}$.
  \item
  Let $X':=\{x_0,...,x_k\} \subseteq \mathrm{VAR}$ and $X\subseteq X'$.
  Let $\varphi$ be a Boolean combination of $X$-atoms and 
  formulae in $\mathcal{F}$ whose free variables (if any) are in the set $X'$.
  Then the formulae $\exists x_1...\exists x_k\varphi$ and $\exists x_0...\exists x_k\varphi$
  are in $\mathcal{F}$.
\end{enumerate} 
For example,
$\exists x \exists y \exists z
(\neg Rxyzxy \wedge \neg Tyxz \wedge Px \wedge Qy)$ and
$\exists x \forall y \forall z (\\ \neg Sxy\rightarrow \exists u\exists v Tuvz)$
are formulae of $\mathrm{U}_1(\mathit{no}\hspace{-1mm}=)$.
If $\psi(y)$ is a formula of $\mathrm{U}_1(\mathit{no}\hspace{-1mm}=)$,
then $\exists y\exists z(Txyz\wedge Rzxyzz\wedge\psi(y))$ is as well.
However, the formula $\exists x \exists y\exists z(Sxy \vee Sxz)$ is
\emph{not} a formula of $\mathrm{U}_1(\mathit{no}\hspace{-1mm}=)$
because $\{x,y\}\not = \{x,z\}$. The formula is
said to violate the uniformity condition, i.e., the
syntactic restriction that the relational
atoms of higher arity bind the same set of variables.
The formula $\forall y (Py\wedge \exists xTxyz)$ is
not a formula of $\mathrm{U}_1(\mathit{no}\hspace{-1mm}=)$
because it violates one-dimensionality, as $\exists xTxyz$
has two free variables.
Perhaps the simplest formula of $\mathrm{U}_1(\mathit{no}\hspace{-1mm}=)$
that can be expressed in neither two-variable logic with
counting quantifiers $\mathrm{FOC}^2$
nor in the guarded negation fragment $\mathrm{GNFO}$ is
the formula $\exists x\exists y\exists z\neg Txyz$.
The set of formulae of the \emph{fully
uniform one-dimensional fragment} $\mathrm{FU}_1$ is
obtained from the set of formulae of $\mathrm{U}_1(\mathit{no}\hspace{-1mm}=)$
by allowing the substitution of any binary relation 
symbols in a formula of $\mathrm{U}_1(\mathit{no}\hspace{-1mm}=)$
by the equality symbol $=$. If restricted to vocabularies
with at most binary symbols, $\mathrm{FU}_1$ is 
exactly as expressive as $\mathrm{FO}^2$ \cite{DBLP:conf/dlog/Kuusisto16}. 
The set of $\tau$-formulae of the
\emph{uniform one-dimensional fragment} $\mathrm{U}_1$ is
the smallest set $\mathcal{F}$ obtained by adding to the four above
clauses that define $\mathrm{U}_1(\mathit{no}\hspace{-1mm}=)$
the following additional clause:
\begin{enumerate}\setcounter{enumi}{4}
\item Every equality atom $x=y$ is in $\mathcal{F}$.
\end{enumerate}
For example $\exists y\exists z( Txyz \wedge Qy\wedge x\not= y)$
as well as the formula $\exists x \exists y\exists z( x\not= y\wedge
y\not= z\wedge z\not = x)$ are $\mathrm{U}_1$-formulae.
The latter formula is an example of a (counting) condition that is
well known to be undefinable in $\mathrm{FO}^2$.
A more interesting example of a condition not
expressible in $\mathrm{FO}^2$ (cf. \cite{DBLP:conf/dlog/Kuusisto16}) is
defined by the $\mathrm{U}_1$- formula $\exists x\forall y
\forall z(Syz\rightarrow (x=y\vee x=z))$,
which expresses that some element is part of every tuple of $S$.
For more examples and background intuitions, see the
survey \cite{DBLP:conf/dlog/Kuusisto16}.
Let $\bar{x}$ be a tuple of variables.
Let $\exists \bar{x} \varphi$ be a $\mathrm{U}_{1}$-formula
which is formed by applying the rule 4 of the syntax above.
Recall the set $X$ used in the formulation.
If $\varphi$ does not contain any relational atom (other than equality)
with at least two distinct variables, we define $L_{\varphi}:=\emptyset$,
and otherwise we define $L_{\varphi}:= X$.
We call the set $L_{\varphi}$ the set of
\textit{live variables} of $\varphi$.
For example, in $\exists y\exists z\exists u(Txyz
\wedge Rxxyyz\wedge x= u\wedge Q(u))$ the
set of live variables is $\{x,y,z\}$.
A quantifier-free subformula of a $\mathrm{U}_1$-formula is
called a $\mathrm{U}_1$-matrix.
%
%
%
%
%
%
Let $\psi(x_1, ...\, ,x_k)$ be a $\mathrm{U}_{1}$-matrix with exactly the distinct
variables $x_1,...\, , x_k$.
Let $\mathfrak{A}$ be a model with domain $A$, 
and let $a_1,...\, , \\ a_k \in A$ be (not necessarily distinct) elements.
Let $T$ be the smallest subset of $\{a_1,...,a_k\}$
such that for every $x_i \in L_{\psi}$, we have $a_i \in T$,
\blue{i.e. $T = \{a_i \mid x_i \in L_{\psi}\}$.}
We denote $T$ by $live\big(\psi(x_1,...,x_k)[a_1,...,a_k]\big)$.
For example, if 
$\psi(v_1,v_2,v_3,v_4) := (Rv_2v_3v_2\wedge Pv_4\wedge v_1=v_2)$,
then $\\ live\big(\psi(v_1,v_2,v_3,v_4)[a,b,c,b]\big) = \{b,c\}.$
We shall shorten the notation 
$$live\big(\psi(x_1,...,x_k)[a_1,\ldots,a_k]\big)$$
to $live\big(\psi[a_1,...,a_k]\big)$ when there is no
possibility of confusion.
%
%
%
%
%
%
%
%
%
%

%
%
%
%
A $\mathrm{U}_1$-formula $\varphi$ is in \textit{generalized Scott normal form}, if
\begin{align*}
\varphi = &\bigwedge \limits_{1\leq i \leq m_{\exists}} \forall x \exists y_1
\ldots \exists y_{k_i}
\varphi_{i}^{\exists}(x,y_1,\ldots,y_{k_i})
\; \wedge \\
&\bigwedge \limits_{1\leq i \leq m_{\forall}} \forall x_1 \ldots \forall x_{l_i}
\varphi_{i}^{\forall}(x_1,\ldots,x_{l_i}),
\end{align*}
where the formulae $\varphi_{i}^{\exists}$ and $\varphi_{i}^{\forall}$
are \emph{$\mathrm{U}_1$-matrices}.
Henceforth by a normal form we always mean generalized Scott normal form.
The formulae $$\forall x \exists y_1... \exists y_{k_i} 
\varphi_{i}^{\exists}(x,y_1,...,y_{k_i})$$ are 
called \emph{existential conjuncts} and the
formulae $$\forall x_1...\forall x_{l_i}
\varphi_{i}^{\forall}(x_1,...,x_{l_i})$$
\emph{universal conjuncts} of $\varphi$.
The quantifier-free part of an existential (universal) conjunct is called an
\emph{existential (universal) matrix}.
We often do not properly differentiate between
existential conjuncts and existential matrices when
there is no risk of confusion. The same holds for universal
matrices and universal conjuncts.
The \emph{width} of $\varphi$ is 
the maximum number of the set $\{ k_{i}+1 \}_{1\leq i
\leq m_{\exists} }
\cup \{ l_{i} \}_{1\leq i \leq m_{\forall} }$.
\blue{We shall usually denote the width of $\varphi$ by $n$ and we assume w.l.o.g. that $n \geq 2$.}
%
%
%
%
%
%
%
%
%
\begin{proposition}[\cite{DBLP:conf/mfcs/KieronskiK14}]\label{PreliminariesNormalForm}
Every $\mathrm{U}_{1}$-formula $\varphi$ can be translated in polynomial time
to a $\mathrm{U}_{1}$-formula $\varphi'$ in
generalized Scott normal form that is equisatisfiable
with $\varphi$ in the following sense. If $\mathfrak{A}\models\varphi$,
then $\mathfrak{A}^*\models\varphi'$ for some expansion $\mathfrak{A}^*$
of $\mathfrak{A}$, and vice versa, if $\mathfrak{B}\models\varphi'$, then
$\mathfrak{B}'\models\varphi$ for some reduct $\mathfrak{B}'$ of $\mathfrak{B}$.
The vocabulary of $\varphi'$ expands the vocabulary of $\varphi$ with
fresh unary relation symbols only.
\end{proposition}
Let $\mathfrak{A}$ be a model satisfying a normal
form sentence $\varphi$ of $\mathrm{U}_1$.
Let $$a,a_1,...\, ,a_{k_{i}} \in A,$$
and let $\forall x \exists y_1... \exists y_{k_i}
\varphi_{i}^{\exists}(x,y_1,...,y_{k_i})$ be an existential conjunct of $\varphi$
such that $\mathfrak{A} \models
\varphi_{i}^{\exists}(a,a_1,...,a_{k_{i}})$.
Then we define $\mathfrak{A}_{a,\varphi_{i}^{\exists}} :=
\mathfrak{A}\upharpoonright\nolinebreak\{a,a_1,\ldots,a_{k_{i}} \}$
and we call $\mathfrak{A}_{a,\varphi_{i}^{\exists}}$ a \textit{witness structure}
for the pair $(a,\varphi_{i}^{\exists})$.
The elements of the witness structure
are called \textit{witnesses}.
In addition, we define
$\bar{\mathfrak{A}}_{a,\varphi_{i}^{\exists}} :=
\mathfrak{A}_{a,\varphi_{i}^{\exists}} \upharpoonright 
live(\varphi_{i}^{\exists}[a,a_1,\ldots,a_{i_{k}}])$
and we call it the $\textit{live part}$ 
of $\mathfrak{A}_{a,\varphi_{i}^{\exists}}$.
If the live part 
$\bar{\mathfrak{A}}_{a,\varphi_{i}^{\exists}} $ 
does not contain $a$, then it is called \textit{free}.   
The remaining part 
$\mathfrak{A}_{a,\varphi_{i}^{\exists}} \upharpoonright
(A_{a,\varphi_{i}^{\exists}} \setminus \bar{A}_{a,\varphi_{i}^{\exists}})$ 
of $\mathfrak{A}_{a,\varphi_{i}^{\exists}}$
is called the \textit{dead part} of the witness structure.
In other words, the witness structure consists of the two parts:
the live part and the dead part.

\blue{The next two subsections \ref{structureclassessection} and \ref{typesandtables} introduce concepts
which abstract a treatment of relational structures in a convenient way. 
}

\subsection{Structure classes}\label{structureclassessection}
Fix a binary relation $<$. 
Throughout the article, we
let $\mathcal{O}$ denote the class of all structures $\mathfrak{A}$
such that $\mathfrak{A}$ is a $\tau$-structure for some $\tau\subseteq\mathcal{V}$
with $<\, \in\tau$, and the symbol $<$ is
interpreted as a \emph{linear order} over $A$. 
\blue{
Intuitively, $<$ can be regarded as an interface 
``implemented'' by the structures in $\mathcal{O}$.\footnote{Cf. object-oriented programming theory}
This abstraction enables us to study the structures in $\mathcal{O}$ without 
interfering with other relational information which the structures may have.}
(Note that thus the vocabulary is not required to be the same for
all models in $\mathcal{O}$.)
The class $\mathcal{WO}$ is defined similarly,
but this time $<$ is interpreted as a well-ordering of $A$, i.e., a
linear order over $A$ such that each nonempty subset of $A$ has a
least element w.r.t. $<$.
Similarly, $\mathcal{O}_{fin}$ is the subclass of $\mathcal{O}$
where every model is finite.

Consider a class $\mathcal{K}\in\{\mathcal{O},\mathcal{WO},\mathcal{O}_{fin}\}$.
%
%
%
%
The \emph{satisfiability problem of\, $\mathrm{U}_1$ over $\mathcal{K}$}, 
\blue{denoted by $sat_{\mathcal{K}}(\mathrm{U}_1)$},
asks, given a formula of $\mathrm{U}_1$,
whether $\varphi$ has a model in $\mathcal{K}$.\footnote{
By the complexity of a logic $\mathcal{L}$ we mean the complexity of the satisfiability problem of the logic $\mathcal{L}$.
}
The set of relation symbols in the input formula $\varphi$ is not limited in any way.
\blue{The standard (general case) satisfiability problem of $\mathrm{U}_{1}$ is denoted by 
$sat(\mathrm{U}_1)$. }

If $R_1$ and $R_2$ are binary relation
symbols, we let $\mathrm{U}_1[R_1,R_2]$ be the
extension of $\mathrm{U}_1$ such that
$\varphi$ is a formula of $\mathrm{U}_1[R_1,R_2]$ iff it can be obtained
from some formula of $\mathrm{U}_1$ by replacing any number of
equality symbols with $R_1$ or $R_2$; 
for example $\forall x\forall y\forall z ((R_1xy\wedge R_1yz)\rightarrow R_1xz)$
is obtained from the $\mathrm{U}_1$-formula 
$\forall x\forall y\forall z ((x=y\wedge y=z)\rightarrow x=z)$ this way.
Such extensions of $\mathrm{U}_{1}$
are said to allow non-uniform use of $R_1$ and $R_2$ in formulae.
%
%
%
%
At the end of this paper we investigate $\mathrm{U}_1[<_1, <_2]$
over structures where $<_1$ and $<_2$
both denote linear orders.
\subsection{Types and tables}\label{typesandtables}
Let $\tau$ be a finite relational vocabulary.
A 1-type (over $\tau$) is a maximally consistent set of unary $\tau$-atoms and negated
unary $\tau$-atoms in the single variable $v_1$.
We denote 1-types by $\alpha$ and the
set of all 1-types over $\tau$ by $\bbalpha_{\tau}$.
If there is no risk of confusion, we may write $\bbalpha$ instead of $\bbalpha_{\tau}$.
The size of $\bbalpha_{\tau}$ is clearly bounded by $2^{|\tau|}$.
We often identify a $1$-type $\alpha$ with the
conjunction of its elements, thereby considering $\alpha(x)$ as
simply a formula in the single variable $x$. (Note that here 
we used $x$ instead of the official variable $v_1$ with
which the $1$-type $\alpha$ was defined.)
Let $\mathfrak{A}$ be a $\tau$-model and $\alpha$ a $1$-type over $\tau$.
The type $\alpha$ is said to be \emph{realized} in $\mathfrak{A}$
if there is some $a \in A$ such that $\mathfrak{A} \models \alpha(a)$.
We say that the point $a$ realizes the $1$-type $\alpha$ in $\mathfrak{A}$
and write $tp_{\mathfrak{A}}(a) = \alpha$. Note that 
every element of $\mathfrak{A}$ realizes exactly one $1$-type over $\tau$.
We let $\bbalpha_{\mathfrak{A}}$ denote the
set of all $1$-types over $\tau$ that
are realized in $\mathfrak{A}$.
It is worth noting that $1$-types do not only involve
unary relations:  for example an atom $Rxxx$ can be part of a $1$-type.

Let $k\geq 2$ be an integer.
A \emph{$k$-table} over $\tau$
is a maximally consistent set of $\{v_1,...\, ,v_k\}$-atoms
and negated $\{v_1,...\, ,v_k\}$-atoms over $\tau$.
\blue{Recall that a $\{v_1,...\, ,v_k\}$-atom must contain exactly all the variables in $\{v_1,...\, ,v_k\}$.
Recall also that $v_1 = v_2$ is not $\{v_1,v_2\}$-atom, and thus 
2-tables do not contain identity atoms or negated identity 
atoms.\footnote{Alternatively, we could define 2-tables such that they always contain $\neg v_1=v_2$
and $\neg v_2=v_1$.}  }
\begin{example}\label{rubbishexample}
Using the metavariables $x,y$ instead of $v_1,v_2$,
the set
$$\{Rxxy,Rxyx,\neg Ryxx, Ryyx,\neg Ryxy, Rxyy,
x<y,\neg \, y<x\}$$
is a $2$-table over $\{R,<,P\}$, where $R$ is a ternary,
$<$ binary and $P$ a unary symbol.
\blue{
The set $\{ \neg Rxxx, Pxx, \neg(x < x), x = x \}$ is a 1-type over the same vocabulary.
}
\end{example}
We denote $k$-tables by $\beta$.
%
%
%
%
%
Similarly to what we did with $1$-types, a $k$-table $\beta$
can be identified 
with the conjunction of its elements, denoted by $\beta(x_1,...\, , x_k)$.
If $a_1,...\, , a_k \in A$ are \emph{distinct} elements
such that $\mathfrak{A} \models \beta(a_1,\ldots,a_k)$,
we say that $(a_1,...\, , a_k)$ \emph{realizes} the table $\beta$
and write $tb_{\mathfrak{A}}(a_1,...\, ,a_k) = \beta$.
Every tuple of $k$ distinct elements in the $\tau$-structure $\mathfrak{A}$
realizes exactly one $k$-table $\beta$ over $\tau$.

\blue{
Note that 1-types and $k$-tables encapsulate relational information,
enabling us to construct relational structures in the following way.
Let $m$ be the maximum arity of symbols in $\tau$. 
To fully define a $\tau$-structure $\mathfrak{A}$ over a given domain $A$, 
we first specify a 1-type for each element in $A$.
Then for each subset $B \subseteq A$ for which $2 \leq |B| \leq m$,
we choose an enumeration $(b_1,\ldots,b_{|B|})$ of the elements of $B$ and specify
a $|B|$-table $tb_{\mathfrak{A}}(b_1,\ldots,b_{|B|})$. }

\blue{Lastly we define some notions derived from the concept of 1-type.}
Let $\alpha$ be a $1$-type. We define the formulae 
%
%
%
%
$min_{\alpha}(x) := \alpha(x) \wedge \forall y \big( (\alpha(y) \wedge x \neq y) \rightarrow x < y \big)$ and
$max_{\alpha}(x) := \alpha(x) \wedge \forall y \big( (\alpha(y) \wedge x \neq y) \rightarrow y < x \big)$
%
%
%
%
for later use. An element $a\in A$ is called a
\emph{minimal} (resp., \emph{maximal})
\emph{realization} of $\alpha$ in $\mathfrak{A}$ iff
$\mathfrak{A}\models\mathit{min}_{\alpha}(a)$ (resp.,
$\mathfrak{A}\models\mathit{max}_{\alpha}(a)$).
This definition holds even if $\mathfrak{A}$ interprets $<$ as a 
binary relation that is \emph{not} a linear order; at a certain very clearly
marked stage of the investigations below, the symbol $<$ is 
used over a model $\mathfrak{B}$ where it is not necessarily interpreted as an order
but is instead simply a binary relation.
%
%
%
%
%
%

%
%
%
Let $\varphi$ be a normal form sentence of $\mathrm{U}_1$ over $\tau$
and let $\mathfrak{A}$ be a $\tau$-model.
Let $n$ be the width of $\varphi$.
A $1$-type $\alpha$ over $\tau$ is called \textit{royal} (in $\mathfrak{A}$
and w.r.t. $\varphi$)
if there are at most $n-1$ elements in $A$ realizing $\alpha$.
Elements in $A$ that realize a royal $1$-type are called \textit{kings} (w.r.t. $\varphi$).
Other elements in $A$ are \textit{pawns} (w.r.t. $\varphi$).
If $K_{\mathfrak{A}}$ denotes the set of kings in $\mathfrak{A}$,
then $K_{\mathfrak{A}}$ is bounded by $(n-1)|\bbalpha| = (n-1)2^{|\tau|}$,
where $\bbalpha$ is the set of all 1-types over $\tau$.
\blue{It is important to distinguish kings from pawns in structures.
In $\mathrm{U}_1$ we can express, for example, that there are exactly
three elements satisfying a unary relation $P$.
In symbols: 
\begin{align*}
 \psi :=
  &\exists x_1...\exists x_3(\, \bigwedge \limits_{i \neq j} (x_i \neq x_j) \wedge
  \bigwedge \limits_{i} Px_i \,) \, \wedge \\ 
  &\forall x_1...\forall x_{4}  
  \bigl(\,(\,\bigwedge \limits_{i} Px_i\,) \rightarrow \bigvee \limits_{j\neq k} x_j = x_k \,\bigr).
\end{align*}
Let $\mathfrak{A}$ be a model of $\psi$.  
Cloning or copying kings, and thus extending the structure $\mathfrak{A}$ by new kings, 
would directly contravene the sentence $\psi$.
However, cloning pawns, even in ordered structures, would do no harm. 
This fact is established in the next section.
Note that $\psi$ is in normal form and its width is 4.
}

Now recall the notion of a
witness structure $\mathfrak{A}_{a,\varphi_i^{\exists}}$ in a
model $\mathfrak{A}$ for a pair $(a,\varphi_i^{\exists})$,
where $a\in A$ is an element and $\varphi_i^{\exists}$ an
existential conjunct of a normal form formula. 
Let $\alpha$ be a $1$-type. By a
\emph{witness structure of $(\alpha,\varphi_i^{\exists})$}
we mean a witness structure $\mathfrak{A}_{a',\varphi_i^{\exists}}$
for some pair $(a',\varphi_i^{\exists})$ such that $a'\in A$ realizes $\alpha$.

\section{Analysing ordered structures}\label{Analysing}
\blue{We begin this section by showing the fact that, if we have a structure (particularly an ordered one) and
it has pawns, then these pawns can be cloned as many times as one wishes.
This is the property that will be seen to be very convenient not only in this section, but
also when we are proving the main result in Section \ref{reducing}.}

Let $\varphi$ be a normal form sentence of $\mathrm{U}_1$
and $\tau$ the set of relation symbols in $\varphi$.
Assume that the symbol $<$ occurs in $\varphi$.
Let $r$ be the highest
arity occurring in the symbols in $\tau$, and
let $n$ be the width of $\varphi$.
Denote $min\{r,n\}$ by $m$.
Let $\mathfrak{A}\in\mathcal{O}$ be a $\tau$-model that satisfies $\varphi$.
Let $P\subseteq A$ be the set of all pawns (w.r.t. $\varphi$) of $\mathfrak{A}$.
Thus, for every $p \in P$,
there are at least $n$ elements in $A$ realizing
the $1$-type of $p$.
Let $\mathbf{c}\geq 3$ be an integer.
The \textit{$\mathbf{c}$-cloning extension of $\mathfrak{A}$
with respect to $\varphi$} is a linearly ordered \textit{extension} $\mathfrak{A}'$ of
$\mathfrak{A}$ defined by the following process.

%
\textbf{1. \textit{Defining an ordered domain for $\mathfrak{A}'$}:}
For each $p \in P$, let $Cl(p)$ be\,  a\, 
set $\{p_{0}\}\cup\{p_2,...\, ,p_{\mathbf{c}-1}\}$ of fresh elements.
The domain of $\mathfrak{A}'$ is the set $A' = A \cup \bigcup_{p \in P} Cl(p)$.
For each $p\in P$, the elements $\{p_2,...\, ,p_{\mathbf{c}-1}\}$ are
placed immediately after $p$ while the element $p_0$ is inserted 
immediately before $p$, so $\{p_0\}\cup \{p\}\cup\{p_2,...\, ,p_{\mathbf{c}-1}\}$
becomes an interval with $\mathbf{c}$\nolinebreak\ elements
such that $p_0 < p< p_2 < ... < p_{\mathbf{c}-1}$.
%
%
%
The reason why we place the element $p_0$ before $p$ and the other elements after it will
become clear later on. 
\blue{Note that the ordering of the domain $A'$ is completed in the next stage.}

\textbf{2. \textit{Cloning stage}:}
For every $p \in P$, every $p' \in Cl(p)$,
and every subset $S \subseteq A\setminus \{p\}$ such that
$1\leq |S| \leq m-1$,
we define $tp_{\mathfrak{A'}}(p') := tp_{\mathfrak{A}}(p)$ and
$tb_{\mathfrak{A'}}(p',\bar{s}) := tb_{\mathfrak{A}}(p,\bar{s})$,
where $\bar{s}$ is an $|S|$-tuple that
enumerates the elements of $S$.
\textbf{3. \textit{Completion stage}:}
For each $p \in P$,
let $I_{p}$ denote the interval $\{p_0\}\cup\{p\} \cup \{p_2,...\, ,p_{\mathbf{c}-1}\}$.
We call the intervals $I_{p}$ \textit{clone intervals}
and define $\mathbf{I} := \bigcup_{p \in P} I_{p}$.
Now define $P_2$ to be the 
set of all pairs $(\alpha_1,\alpha_2)$ of $1$-types
such that we have $\mathfrak{A}'\models \alpha_1(u)\wedge\alpha_2(u') \wedge u < u'$
for some elements $u,u'\in A'$. (Note that $\alpha_1$ and $\alpha_2$ are allowed to be the same type.)
Then define a function $t_2:P_2\rightarrow A^2$ that maps every pair $(\alpha_1,\alpha_2)$ in $P_2$
to some pair $(w,w')\in A^2$ such that $tp_{\mathfrak{A}}(w) = \alpha_1$,
$tp_{\mathfrak{A}}(w') = \alpha_2$ and $w<^{\mathfrak{A}}w'$.
We then do the following.
%
%
%


Assume $u,u'\in \mathbf{I}$
such that $u <^{\mathfrak{A}'} u'$.
Let $\alpha_1$ and $\alpha_2$ denote the $1$-types of $u$
and $u'$, respectively, and assume no table has been 
defined over $(u,u')$ or $(u',u)$ in the \textit{cloning stage}.
%
%
%
Then we define
$tb_{\mathfrak{A}'}( u, u' ) := tb_{\mathfrak{A}}( t_2(\alpha_1,\alpha_2))$.
%
%


%
Now recall $m=\mathit{min}\{n,r\}$. Assume $k\in\{3,...\, , m\}$, and let $P_k$ be the 
set of tuples $(\alpha_1,...\, , \alpha_{k})$ of $1$-types (repetitions of types allowed)
such that $\mathfrak{A}'\models \alpha_1(u_1)\wedge...\wedge\, \alpha_k(u_k)$
for some elements $u_1,...\, ,u_k\in A'$ such
that $u_1<^{\mathfrak{A}'} u_2 <^{\mathfrak{A}'}  ...<^{\mathfrak{A}'}u_k$.
Define a function $t_k:P_k\rightarrow A^k$ that maps every
tuple $(\alpha_1,...\, ,\alpha_k)$ in $P_k$
to some tuple $(w_1,...\, , w_k)\in A^k$ of \textit{distinct}
elements such that $tp_{\mathfrak{A}}(w_j) = \alpha_j$ for
each $j\in\{1,...\, , k\}$. Note that the order of the
elements $w_1,...\, ,w_k$ in $\mathfrak{A}$
does not matter, and note also that it is indeed always possible to find $k$ suitable
elements because each pawn in $\mathfrak{A}$ has at
least $n\geq m\geq k$ occurrences in $\mathfrak{A}$. 
Now consider every tuple $(u_1,...\, , u_k)\in {A'}^k$ of elements
such that $u_1<^{\mathfrak{A}'} u_2 <^{\mathfrak{A}'}  ...<^{\mathfrak{A}'}u_k$
and such that we have not defined any table in the \textit{cloning stage}
over $(u_1,...\, , u_k)$ or over any permutation of $(u_1,...\, , u_k)$,
and define $tb_{\mathfrak{A}'}(u_1,...\, , u_k) := tb_{
\mathfrak{A}}(t_k(\alpha_1,...\, ,\alpha_k))$,
where $\alpha_j$ denotes the type of $u_j$ for each $j$.
Do this procedure for each $k\in\{3,...\, ,m\}$.
Finally, \emph{over
tuples with more than $m$ distinct elements}, we
define arbitrarily the interpretations (in $\mathfrak{A}'$) of
relation symbols of arities greater than $m$.
This completes the definition of $\mathfrak{A}'$.
%
%
%
%
\begin{lemma}\label{DecidableLemma1}
Let $\mathfrak{A}\in\mathcal{O}$ be a model
and $\mathfrak{A}'$ its $\mathbf{c}$-cloning
extension w.r.t. $\varphi$.
Now, if $\mathfrak{A}\models\varphi$,
then $\mathfrak{A}'\models\varphi$.
\begin{proof}
It is easy to show that the existential 
conjuncts are dealt with in the \emph{cloning
stage} of the construction of $\mathfrak{A}'$, so we only need to
argue that for all universal conjuncts $\chi$ of $\varphi$, if $\mathfrak{A}\models\chi$,
then $\mathfrak{A}'\models\chi$.
To see that $\mathfrak{A}'$ satisfies the universal conjuncts,
consider such a conjunct $\forall x_1...\forall x_k \psi(x_1,...\, ,x_k)$,
where $\psi(x_1,...\, , x_k)$ is quantifier free,
and let $(a_1,...\, ,a_k)$ be a tuple of elements from $A'$, with possible repetitions.
We must show that $\mathfrak{A}'\models\psi(a_1,...\, ,a_k)$.
Let $\{u_1,...\, ,u_{k'}\} := 
\mathit{live}(\psi(x_1,...\, ,x_k)[a_1,...\, ,a_k])$ and
call $V := \{a_1,...\, ,a_k\}$.
The table $tb_{\mathfrak{A}'}( u_1,...\, ,u_{k'})$ has
been defined either in the \emph{cloning stage} or the
\emph{completion stage} to be $tb_{\mathfrak{A}}( b_1,...\, ,b_{k'})$
for some distinct elements $b_1,...\, ,b_{k'}\in A$.
Furthermore, since $\mathfrak{A}'$ and $\mathfrak{A}$
have exactly the same number of realizations of each royal $1$-type 
and since both models have at least $n\geq k$ realizations of
each pawn, it is easy to define an injection $f$ from $V$ into $A$
that preserves $1$-types and such that $f(u_i) = b_i$ for each $i\in\{1,...\, , k'\}$.
Therefore $\mathfrak{A}'\models\psi(a_1,...\, ,a_k)$
iff $\mathfrak{A}\models\psi(f(a_1),...\, ,f(a_k))$.
Since $\mathfrak{A}\models\varphi$, we
have $\mathfrak{A}\models\psi(f(a_1),...\, ,f(a_k))$
and therefore $\mathfrak{A}'\models\psi(a_1,...\, ,a_k)$.
\end{proof}
\end{lemma}
We now fix a \blue{normal form} sentence $\varphi$ of $\mathrm{U}_1$
with the set $\tau$ (with $<\, \in\tau$) of relation symbols
occurring in it. We also fix a $\tau$-model $\mathfrak{A}\in\mathcal{O}$.
We assume $\mathfrak{A}\models\varphi$ and
fix a $3$-cloning
extension $\mathfrak{A}'$ of $\mathfrak{A}$ w.r.t. $\varphi$.
We let $n$ be the width of $\varphi$ and $m_{\exists}$ the number of
existential conjuncts in $\varphi$.
The models $\mathfrak{A}$ and $\mathfrak{A}'$ as well as the
sentence $\varphi$ will remain fixed in the next 
two subsections (\ref{courtsection} and \ref{intervalpartitionsection}).
%
%
%
In the two subsections we will study these
two models and the sentence $\varphi$
and isolate some constructions and concepts that will be used later on.
\subsection{Identification of a court}\label{courtsection}
Let $K$ denote the set of kings of $\mathfrak{A}'$  (w.r.t. $\varphi$). Thus $K$ is
also the set of kings of $\mathfrak{A}\subseteq \mathfrak{A'}$.
We next identify a finite substructure $\mathfrak{C}$ of $\mathfrak{A}$
called a \textit{court of\, $\mathfrak{A}$ with respect to $\varphi$}.
We note that a court of $\mathfrak{A}$ w.r.t. $\varphi$ can in
general be chosen in several ways.
Before defining $\mathfrak{C}$,
we construct a certain set $D\subseteq A$.
Consider a pair $(\alpha,\varphi_i^{\exists})$, where $\alpha$ is a $1$-type
(over $\tau$) and $\varphi_i^{\exists}$ an existential conjunct of $\varphi$.
If there exists a \textit{free}
witness structure in $\mathfrak{A}$ for $\varphi_i^{\exists}$
and some element $a\in A$ realizing $1$-type $\alpha$,
then pick exactly one such free witness structure $\mathfrak{A}_{a,\varphi_i^{\exists}}$
and define $D(\alpha,\varphi_i^{\exists}) :=
\bar{A}_{a,\varphi_i^{\exists}}$, i.e., $D(\alpha,\varphi_i^{\exists})$
is the \emph{live part} of $\mathfrak{A}_{a,\varphi_i^{\exists}}$.
Otherwise define $D(\alpha,\varphi_i^{\exists}) = \emptyset$.
Define $D$ to be the union of the sets $D(\alpha,\varphi_i^{\exists})$ for
each $1$-type $\alpha$ (over $\tau$) and
each existential conjunct $\varphi_i^{\exists}$ of $\varphi$.
The size of $D$ is bounded by $m_{\exists}|\bbalpha|n$.
Now, for each $a \in (K \cup D) \subseteq A$ and each $\varphi_{i}^{\exists}$,
let $\mathfrak{C}_{a,\varphi_{i}^{\exists}}$ be
some witness structure for the pair $(a,\varphi_{i}^{\exists})$ in $\mathfrak{A}$.
Define the domain $C$ of $\mathfrak{C}$ as follows:
%
%
\begin{center}
$C\ :=\ \bigcup\limits_{a \in K\cup D,\ 1 \leq i \leq m_\exists} 
C_{a,\varphi_{i}^{\exists}}$ 
\end{center}
%
%

Note that $K$ and $D$ are both subsets of $C$.
We define $\mathfrak{C}$ to be the substructure of $\mathfrak{A}$
induced by $C$, i.e., $\mathfrak{C} := \mathfrak{A}\upharpoonright C$.
Thus $\mathfrak{C}$ is also a substructure of $\mathfrak{A}'$.
An upper bound for
the size of $C$ is obtained as follows, where $\bbalpha$
denotes $\bbalpha_{\tau}$\hspace{1mm}.
\begin{flalign*}
|C| &\leq |D \cup K|nm_\exists
\leq (nm_{\exists}|\blet{\alpha}| + n|\blet{\alpha}|)nm_\exists  \\
&\leq (|\varphi|^2|\blet{\alpha}| + |\varphi||\blet{\alpha}|)|\varphi|^2
%
%
\leq  2|\varphi|^4|\blet{\alpha}|.
\end{flalign*}
We call $\mathfrak{C}$ the court of $\mathfrak{A}$ (w.r.t. $\varphi$). 
Note that we could
have chosen the court $\mathfrak{C}$ in many ways from $\mathfrak{A}$.
Here we choose a single court $\mathfrak{C}$ for $\mathfrak{A} \subseteq \mathfrak{A'}$ 
and fix it for Subsection \ref{intervalpartitionsection}.
%
%
%
%
\blue{
Note also that 
if $p \in C\subseteq A \subseteq A'$ is a pawn, 
then it forms the interval $I_p = \{p_0\}\cup\{p\} \cup \{p_2\}$ in the 3-cloning extension $\mathfrak{A'}$,
where $p_0$ and $p_2$ are of the same 1-type as $p$, and they are in $A'$ but not in $A$.
Therefore, choosing elements from $A$ guarantees that
no pawn in $C$ of a 1-type $\alpha$ satisfies either $min_{\alpha}(x)$ or $max_{\alpha}(x)$
in the 3-cloning extension $\mathfrak{A'}$.
}

\subsection{Partitioning cloning extensions
%
%
into intervals}\label{intervalpartitionsection}
In this subsection we partition
the $3$-cloning extension $\mathfrak{A}'$ of
the ordered structure $\mathfrak{A}$ 
into a finite number of non-overlapping intervals.
Roughly speaking, the elements of
the court $\mathfrak{C}$ of $\mathfrak{A}$ will all create a singleton interval and
the remaining interval bounds will indicate the least upper bounds and
greatest lower bounds of occurrences of $1$-types in $\mathfrak{A}'$.
We next define the partition formally; we call the resulting family of intervals $I_s\subseteq A'$
the \emph{canonical partition of\, $\mathfrak{A}'$ with
respect to $\mathfrak{C}$}.
We begin with some auxiliary definitions. Recall that $\bbalpha_{\mathfrak{A}}$
denotes the set of $1$-types realized in $\mathfrak{A}$,
and thus $\bbalpha_{\mathfrak{A}} = \bbalpha_\mathfrak{A'}$.
For each non-royal $1$-type $\alpha$ in $\bbalpha_{\mathfrak{A}}$, define the sets
\begin{flalign*}
&A'_{\alpha} = \{a \in A'\mid tp_{\mathfrak{A}'}(a) = \alpha \},\;
\mathcal{D}_{\alpha}^{-} = \bigcup_{a \in A'_{\alpha}} \{b \in A' \mid a \leq b \}, \\ 
&\text{ and } 
\mathcal{D}_{\alpha}^{+} = \bigcup_{a \in A'_{\alpha}}\{b \in A' \mid b \leq a \}. 
\end{flalign*}
In an ordered set $(L,<)$, an \textit{interval bound} is
defined to be a nonempty set $S\subsetneq L$ that is
downwards closed ($u'<u\in S\Rightarrow u'\in S$).
%
%
%
%
A finite number of interval bounds define a partition of an ordered set
into a finite number of intervals in a natural way.
We define the following finite collection of interval bounds for $\mathfrak{A}'$.
\begin{itemize}
\item Every $c \in C$ defines two interval bounds,
$\{ u\in A'\ |\ u < c\}$ and $\{u\in A'\, |\, u\leq c\}$.
Thereby each $c\in C$
forms a singleton interval $\{c\}$.
\item Each non-royal 1-type $\alpha$ creates two
interval bounds: the sets $A'\setminus\mathcal{D}_{\alpha}^{-}$
and $\mathcal{D}_{\alpha}^{+}$.
\end{itemize}
This creates a finite family of intervals $(I_s)_{1\leq s \leq N}$ that
partitions $A'$. Here $N$ is the finite total number of intervals in
the family. The intervals $I_s$ in the family are enumerated 
in the natural way, so if $s<s'$ for some $s,s'\in\{1,...\, , N\}$,
then $u<u'$ for all $u\in I_s$ and $u'\in I_{s'}$.
%
%
%
%
%
%

%
%
%
We obtain an upper bound for $N$ as follows.
Observe that the number of interval bounds is
bounded from above by $2(|C|+|\bbalpha|)$,
where $\bbalpha$ denotes the set $\bbalpha_{\tau}$ of all $1$-types over $\tau$.
Thus the number of intervals is 
definitely bounded from above by $2(|C|+|\bbalpha|) + 1$.
Since we know from the previous section
that $|C|\ =\ 2|\varphi|^4|\bbalpha|$, we obtain that
\begin{equation*}
N \leq 2(2|\varphi|^4|\blet{\alpha}|+|\blet{\alpha}|) + 1
= (4|\varphi|^4+2)|\blet{\alpha}| + 1\leq 6|\varphi|^4|\blet{\alpha}|.
\end{equation*}
%
%
%
%

\blue{The next two subsections (\ref{admissibilitytuplessection}
 and \ref{pseudo-orderingaxiomssect}) provide two quite specific constructions
used later in the formulation of Lemma \ref{DecidabilityLemma2}.
Intuitively, the first construction can be regarded as 
some sort of data structure derived from the analysis done in the previous subsections and
the second is in fact simply a set of axioms derived from this data structure.
It is recommended to use especially Subsections \ref{admissibilitytuplessection}
 and \ref{pseudo-orderingaxiomssect} as a reference while reading Section \ref{reducing}. }
\subsection{Defining admissibility tuples}\label{admissibilitytuplessection}
%
%
%
%
%
%
Let $\chi$ be a normal form sentence of $\mathrm{U}_1$
with the set $\sigma$ of relation symbols. Assume $<\, \in\sigma$.
We now define the notion of an \textit{admissibility tuple for $\chi$}.
At this stage we only give a formal definition of admissibility tuples.
The point is to capture enough information of ordered models of $\chi$
to the admissibility tuples for $\chi$ so that satisfiability of $\mathrm{U}_1$
over ordered structures can be
reduced to satisfiability of $\mathrm{U}_1$ over general 
structures in Section \ref{reducing}. In particular, our
objective is to facilitate Lemma \ref{DecidabilityLemma2}.
Once we have given the formal definition of an admissibility tuple, we
provide an example (see Lemma \ref{technlemma}) how a concrete linearly
ordered model of a $\mathrm{U}_1$-sentence can be
canonically associated with an admissibility tuple for that sentence,
thereby providing background intuition related to admissibility tuples.
Indeed, the reader may find it helpful to refer to that part while
internalising the formal definitions.
%
%
%
%
%
%
%
%
%
%
%
%
%
%

%
%
%
Consider a tuple
$$\Gamma_{\chi} := ( \mathfrak{C}^*, 
(\blet{\alpha}_{\sigma,s})_{1\leq s \leq N^{*}}, \blet{\alpha}_\sigma^{K},
\blet{\alpha}_{\sigma}^{\bot}, \blet{\alpha}_{\sigma}^{\top}, 
\delta, F)$$ 
such that the following conditions hold.
\begin{itemize}
\item
$\mathfrak{C}^*$ is a linearly ordered $\sigma$-structure,
and the size of the domain $C^*$ of $\mathfrak{C}^*$ is
bounded by $2|\chi|^4|\bbalpha_{\sigma}|$.
Compare this to the bound $2|\varphi|^4|\bbalpha_{\tau}|$ 
for the size of $\mathfrak{C}$ from Section \ref{courtsection}.
We call $\mathfrak{C}^*$ the \emph{court structure of $\Gamma_{\chi}$}.
\item
$N^* \in \mathbb{Z}_{+}$ is an
integer such that $|C^*| \leq N^* \leq 6|\chi|^4|\bbalpha_{\sigma}|$,
%
%
%
and $(\bbalpha_{\sigma,s})_{1\leq s \leq N^*}$
is a family of sets $\bbalpha_{\sigma,s}\subseteq\bbalpha_{\sigma}$ of $1$-types
such that we have $\bbalpha_{\sigma,s} \subseteq
\{\alpha \in \bbalpha_{\sigma} \mid
\neg(v_1 < v_1) \in \alpha \}$ for each $s\in\{1,...\, ,N^*\}$;
recall here that $v_1$ is the variable with which we
formally speaking specify $1$-types, and recall also that
in addition to ordered models, we will ultimately also
consider model classes where $<$ is
simply a binary symbol not necessarily interpreted as an order.
Compare the bound $6|\chi|^4|\bbalpha_{\sigma}|$ to the
bound $6|\varphi|^4|\bbalpha_{\tau}|$ for $N$ from
Section \ref{intervalpartitionsection}.
We call $N^*$ the \emph{index of $\, \Gamma_{\chi}$.}
%
%
%
%
%
%
\item  $\bbalpha_{\sigma}^{K}\subseteq\bbalpha_{\sigma}$
and also $\bbalpha_{\sigma}^{\bot}\subseteq
\bbalpha_{\sigma}$ and $\bbalpha_{\sigma}^{\top}\subseteq\bbalpha_{\sigma}$
%
%
%
%
%
%
%
\item
$\delta$ is an injective mapping from $C^*$ to $\{1,\ldots,N^*\}$. 
%
%
%
\item
$F$ is a subset of the set $\bbalpha_{\sigma}\times
\Phi^{\exists}$, where $\Phi^{\exists}$ is
the set of all existential conjuncts of $\chi$.
%
%
%
%
%
%
\end{itemize}
Note that we could have chosen the tuple $\Gamma_{\chi}$
above in multiple ways.
%
%
We denote the set of all tuples $\Gamma_{\chi}$ that
satisfy the above conditions by
$\hat{\Gamma}_{\chi}$.
The tuples in $\hat{\Gamma}_{\chi}$ are
called \textit{admissibility tuples} for $\chi$.
%
%
%
%
%

%
\begin{lemma}\label{th!MostImportantProofInYourLife}
The length of the (binary) description of\, $\Gamma_{\chi}$ is bounded exponentially in $|\chi|$.
\begin{proof}
Let $\chi$ be a sentence and $\sigma$ the set of relation 
symbols in $\chi$ with $<\in\sigma$.
Consider an arbitrary tuple
$$\Gamma_{\chi} = ( \mathfrak{C}^*, 
(\blet{\alpha}_{\sigma,s})_{1\leq s \leq N^{*}}, \blet{\alpha}_\sigma^{K},
\blet{\alpha}_{\sigma}^{\bot}, \blet{\alpha}_{\sigma}^{\top}, 
\delta, F)$$
such that $\Gamma_{\chi}\in\hat{\Gamma}_{\chi}$.
To prove Lemma \ref{th!MostImportantProofInYourLife}, 
we show that each of the seven elements in $\Gamma_{\chi}$
has a binary description whose length is exponentially bounded in $|\chi|$.
This clearly suffices to prove the lemma.
For describing the model $\mathfrak{C^*}$, we use the straightforward convention from
Chapter 6 of \cite{DBLP:books/sp/Libkin04} according to which
the unique description of $\mathfrak{C^*}$ with some
ordering of $\sigma$ is of the length
$|C^*|+1 +\sum_{i=1}^{|\sigma|}\, |C^*|^{ar(R_i)}$
where $ar(R_i)$ is the arity of $R_i \in \sigma$.
Since we have $|C^*|\leq 2|\chi|^4|\bbalpha_{\sigma}|$
by definition of the tuples in $\hat{\Gamma}_{\chi}$,
and since we clearly have $|\bbalpha_{\sigma}|\leq 2^{|\chi|}$, we 
observe that $|C^*|$ is exponentially bounded in $|\chi|$.
Since $\mathit{ar}(R_i)\leq |\chi|$, each term $|C^*|^{\mathit{ar}(R_i)}$ is
likewise exponentially bounded in $|\chi|$.
Furthermore, as $|\sigma|\leq |\chi|$, 
we conclude that the length of the description of $\mathfrak{C}^*$ is
exponentially bounded by $|\chi|$.
As $|\bbalpha_{\sigma}| \leq 2^{|\chi|}$,
and as each 1-type $\alpha \in \bbalpha_{\sigma}$ can clearly be encoded by a string
whose length is polynomial in $|\chi|$,
we can describe $\bbalpha_{\sigma}$ with a
description that is exponentially bounded in $|\chi|$, and
as $\bbalpha_{\sigma}^K$, $\bbalpha_{\sigma}^\bot $, and $\bbalpha_{\sigma}^\top$
are subsets of $\bbalpha_{\sigma}$, their descriptions are also
exponentially bounded in $|\chi|$.
Moreover, the same upper bound bounds each member $\bbalpha_{\sigma,s}$
of the family $(\bbalpha_{\sigma,s})_{1\leq s \leq N^*}$.
Therefore, as we
have $N^* \leq 6|\chi|^4|\bbalpha_{\sigma}|$ by
the definition of tuples in $\hat{\Gamma}_{\chi}$, we
observe that $N^* \leq 6|\chi|^4\cdot 2^{|\chi|}$,
and therefore the length of the
description of $(\bbalpha_{\sigma,s})_{1\leq s \leq N^*}$ is also
exponentially bounded in $|\chi|$.
Due to the bounds for $|C^*|$ and $N^*$ identified above,
the function $\delta: C^* \rightarrow \{1,\ldots, N^*\}$
can clearly be encoded by a description bounded exponentially in $|\chi|$.
Let $m_{\exists}$ denote the number of 
existential conjuncts in $\chi$.
Thus we have $|F| \leq m_{\exists}|\bbalpha_{\sigma}|\leq |\chi|\cdot 2^{|\chi|}$, so the
description of $F$ can clearly be bounded exponentially in $|\chi|$.
\end{proof}
\end{lemma}
For each $s\in\{1,...\, , N^*\}$,
let $\bbalpha_{\sigma,s}^{-}$ and $\bbalpha_{\sigma,s}^{+}$ be
the subsets of $\bbalpha_{\sigma,s}$ defined as follows:
$\bbalpha_{\sigma,s}^{-} 
:= \bbalpha_{\sigma,s} \setminus \bigcup_{i<s}\bbalpha_{\sigma,i}$ and
$\bbalpha_{\sigma,s}^{+} 
:= \bbalpha_{\sigma,s} \setminus \bigcup_{i>s}\bbalpha_{\sigma,i}$.
The following definition provides an important 
classification of admissibility tuples.
\begin{definition}\label{admidefn}
Consider the set $\hat{\Gamma}_{\chi}$ of
admissibility tuples for $\chi$.
We define the following six conditions,
called \emph{admissibility conditions} for $\chi$, in
order to classify the set $\hat{\Gamma}_{\chi}$
into different sets of admissibility tuples.
\begin{enumerate}
\renewcommand{\theenumi}{\roman{enumi}}
\item \label{AdmissibilitySubsets}
The sets $\bbalpha_{\sigma}^{K}$, $\bbalpha_{\sigma}^{\top}$
and $\bbalpha_{\sigma}^{\bot}$ are subsets of\,
$\bigcup_{1\leq s \leq N^*} \bbalpha_{\sigma,s}$.
\item \label{AdmissibilityCourtiers} If $\bbalpha_{\sigma,s}
\cap \bbalpha_{\sigma}^{K} \neq \emptyset$,
then $s = \delta(c)$ for some $c\in C^*$.
Also, for every $c \in C^*$, it holds that
$\bbalpha_{\sigma,\delta(c)} = \{ tp_{\mathfrak{C}^*}(c) \}$, and
furthermore, $tp_{\mathfrak{C}^*}(c) \in \bbalpha_{\sigma}^{K}$ or\
$\bbalpha_{\sigma,\delta(c)}^{-} = \emptyset = \bbalpha_{\sigma,\delta(c)}^{+}$.
%
%
%
%
%
%
\item \label{AdmissibilityMinusalphaone}
\ $|\bbalpha_{\sigma,s}^{-}| \leq 1$ for all $s\in\{1,...\, ,N^*\}$\vspace{1mm}
\item \label{AdmissibilityBot}
\ $\bbalpha_{\sigma}^{\bot} = \bigcup_{1\leq s \leq N^*} \bbalpha_{\sigma,s}$
\item \label{AdmissibilityPlusalphaone}
\ $|\bbalpha_{\sigma,s}^{+}| \leq 1$ for all $s\in\{1,...\, ,N^*\}$\vspace{1mm}
\item \label{AdmissibilityTop}
\ $\bbalpha_{\sigma}^{\top} = \bigcup_{1\leq s \leq N^*} \bbalpha_{\sigma,s}$
\end{enumerate}
\end{definition}
\begin{definition}\label{admidefn2}
An admissibility tuple $\Gamma_{\chi}$ is
admissible for $\mathcal{O}$
if the conditions \ref{AdmissibilitySubsets}
and \ref{AdmissibilityCourtiers} in Definition \ref{admidefn} are satisfied.
It is admissible for $\mathcal{WO}$
if the four conditions \ref{AdmissibilitySubsets}-\ref{AdmissibilityBot} in Definition
\ref{admidefn} are satisfied.
Finally, it is admissible for $\mathcal{O}_{fin}$ if all
the six conditions \ref{AdmissibilitySubsets}-\ref{AdmissibilityTop} in Definition
\ref{admidefn} are satisfied.
We call admissibility for $\mathcal{O}$ the lowest degree of
admissibility and admissibility for $\mathcal{O}_{fin}$ the highest.
\end{definition}
Let $\varphi$ be a normal form $\mathrm{U}_1$-sentence containing $<$
and $\mathfrak{A}\models\varphi$ a model. Let $\mathfrak{C}$ be a
court of $\mathfrak{A}$ w.r.t. $\varphi$
and $\mathfrak{A}'$ a $3$-cloning
extension of $\mathfrak{A}$ w.r.t. $\varphi$.
Let $(I_s)_{1\leq s \leq N}$ be the canonical
partition of $\mathfrak{A}'$ w.r.t. $\mathfrak{C}$.
We will next specify a tuple
%
%
%
%
%
%
$$\Gamma_{\varphi}^{\mathfrak{C},\mathfrak{A},\mathfrak{A}'} :=
( \mathfrak{C}, (\blet{\alpha}_{\mathfrak{A}',s})_{1\leq s\leq N},
\blet{\alpha}_{\mathfrak{A'}}^{K},
\blet{\alpha}_{\mathfrak{A'}}^{\bot}, \blet{\alpha}_{\mathfrak{A}'}^{\top}, \delta,F)$$
%
%
%
%
%
%
which we call a \emph{canonical admissibility
tuple of\, $\mathfrak{A}'$
w.r.t $(\mathfrak{C},\mathfrak{A},\varphi)$} (cf.
Lemma \ref{technlemma} below).
%
%
%
%
%
%
%
%
%
%
%

%
We now specify the elements of the
tuple $\Gamma_{\varphi}^{\mathfrak{C},\mathfrak{A},\mathfrak{A}'}$ above;
note that $\mathfrak{C}$ has already been specified to be a court of $\mathfrak{A}$.
Recall that $(I_s)_{1\leq s \leq N}$ is the canonical
partition of $\mathfrak{A}'$ w.r.t. $\mathfrak{C}$
and define the family $(\bbalpha_{\mathfrak{A}',s})_{1\leq s \leq N}$ such that
$\bblet{\alpha}_{\mathfrak{A}',s} :=
\{ tp_{\mathfrak{A'}}(a) \mid a \in I_{s} \}$ for all $s\in\{1,...\, , N\}$.
%
%
%
%
%
%
%
%
%
%
%
%
%
%
%
%
%
%
%
%
Let $\bbalpha_{\mathfrak{A}'}^{K} \subseteq \bbalpha_{\mathfrak{A}'}$
be the set of the royal 1-types realized in $\mathfrak{A}'$, and 
define $\bbalpha_{\mathfrak{A}'}^{\bot}
\subseteq \bbalpha_{\mathfrak{A}'}$ (respectively, $\bbalpha_{\mathfrak{A}'}^{\top}
\subseteq \bbalpha_{\mathfrak{A}'}$) to be the set of 1-types that
have a minimal (resp., maximal) realization in $\mathfrak{A}'$.
Note that if $\mathfrak{A}'$ is in $\mathcal{WO}$, 
we have $\bbalpha_{\mathfrak{A}'}^{\bot} = \bbalpha_{\mathfrak{A}'}$,
and if $\mathfrak{A}'$ is also in $\mathcal{O}_{fin}$, then
$\bbalpha_{\mathfrak{A}'}^{\bot} = \bbalpha_{
\mathfrak{A}'}^{\top} = \bbalpha_{\mathfrak{A}'}$.
For every $c$ in the domain $C$ of $\mathfrak{C}$, we
define $\delta(c):=j\in\{1,...\, ,N\}$ such that $I_j=\{c\}$.
We let $F$ be the set of those pairs $(\alpha,\varphi_i^{\exists})$
that have a witness structure in $\mathfrak{A}'$ whose live part is free.
%
%
%
%
%
%
%
%
%
%
%
%
%
%
%
%
%
%

%
\begin{lemma}\label{technlemma}
Let $\mathfrak{A}\in\mathcal{K}\in\{\mathcal{O},\mathcal{WO},\mathcal{O}_{fin}\}$ 
and suppose $\Gamma_{\varphi}^{\mathfrak{C},\mathfrak{A},
\mathfrak{A}'}$ is a
canonical admissibility tuple for $\mathfrak{A}'$ w.r.t $(\mathfrak{C},
\mathfrak{A},\varphi)$.
Then $\Gamma_{\varphi}^{\mathfrak{C},\mathfrak{A},
\mathfrak{A}'}\in\hat{\Gamma}_{\varphi}$ 
and $\Gamma_{\varphi}^{\mathfrak{C},\mathfrak{A},\mathfrak{A}'}$ is
admissible for $\mathcal{K}$.
\begin{proof}
Note that by definition, since $\Gamma_{\varphi}^{\mathfrak{C},
\mathfrak{A},\mathfrak{A}'}$ is canonical admissibility tuple 
for $\mathfrak{A}'$ w.r.t. $(\mathfrak{C},\mathfrak{A},\varphi)$, the
structure $\mathfrak{C}$ is a
court of $\mathfrak{A}$ w.r.t. $\varphi$ and we have $\mathfrak{A}\models\varphi$,
and furthermore, the set of relation symbols in $\varphi$ (to be
denoted by $\tau$) contains $<$.
We let $N$ denote the index of $\Gamma_{\varphi}^{\mathfrak{C},
\mathfrak{A},\mathfrak{A}'}$. We note that $N$ is the
number of intervals in the canonical
partition of $\mathfrak{A}'$ w.r.t. $\mathfrak{C}$.
By the discussion in Section \ref{courtsection}, $\mathfrak{C}$
is an ordered structure whose size is bounded
by $2|\varphi|^4|\bbalpha|$ where $\bbalpha$ is
the set of all $1$-types over $\tau$.
By Section \ref{intervalpartitionsection},
we have $|C| \leq N \leq 6|\varphi|^4|\bbalpha|$.
Thus the admissibility condition \ref{AdmissibilityCourtiers}
from Definition \ref{admidefn} is the
only non-trivial remaining condition to show in order to
conclude that $\Gamma_{\varphi}^{\mathfrak{C},\mathfrak{A},\mathfrak{A}'}$ is an
admissibility tuple in $\hat{\Gamma}_{\varphi}$
admissible for each $\mathcal{K}\in\{\mathcal{O},\mathcal{WO},
\mathcal{O}_{fin}\}$ such that $\mathfrak{A}\in\mathcal{K}$.
We next argue that this condition indeed holds.
First assume that $\bbalpha_{\mathfrak{A}',s}
\cap\bbalpha_{\mathfrak{A}'}^K\not=\emptyset$.
Thus $\alpha\in\bbalpha_{\mathfrak{A}',s}$ for
some royal $1$-type $\alpha$ realized in $\mathfrak{A}'$.
Therefore the interval $I_s\subseteq A'$
contains a king $c$ of $\mathfrak{A}'$ that realizes $\alpha$.
Since kings of $\mathfrak{A}'$ are in singleton intervals of
the family $(I_t)_{1\leq t \leq N}$, we have $I_s = \{c\}$.
Furthermore, since kings of $\mathfrak{A}'$ are all in $\mathfrak{C}$, we
have $c$ in the domain of $\delta$, and thus,  by the
definition of $\delta$, we have $I_{\delta(c)} = \{c\}$.
Thus we have $I_s = \{c\} = I_{\delta(c)}$, whence $s = \delta(c)$.
%
%
%
%
%
%
%
Thus the first part of admissibility
condition \ref{AdmissibilityCourtiers} is satisfied.
%

%
%
%
%
%
%
%
%
To prove the second condition,
assume $c\in C$. Therefore the set $\{c\}$ was
appointed, as described in Section \ref{intervalpartitionsection}, to be a 
singleton interval $I_{\delta(c)}$ in the family $(I_s)_{1\leq s\leq N}$.
Thus $\bbalpha_{\mathfrak{A}',\delta(c)}
= \{tp_{\mathfrak{C}}(c)\}$.
To show that, furthermore, we
have $tp_{\mathfrak{C}}(c)
\in\bbalpha_{\mathfrak{A}'}^K$ or $\bbalpha_{\mathfrak{A}',\delta(c)}^{-}
= \emptyset = \bbalpha_{\mathfrak{A}',\delta(c)}^{+}$, we
consider two cases, the case where $c$ is a king
and the case where it is a pawn.
If $c$ is a king, then $tp_{\mathfrak{C}}(c)\in\bbalpha_{\mathfrak{A}'}^K$ by
the definition of $\bbalpha_{\mathfrak{A}'}^K$.
On the other hand, if $c\in C$ is a pawn, we argue as follows.
Now, as $C \subseteq  A \subseteq A'$,
we know that $c$ has two elements $u,u'\in A'$
of the same $1$-type (as $c$ itself) immediately
before and after $c$ that 
were introduced when 
constructing the $3$-cloning extension $\mathfrak{A}'$ of $\mathfrak{A}$
(see the beginning of Section \ref{Analysing}).
Therefore every $1$-type has
neither its first nor last realization in $\mathfrak{A}'$ in the
interval $I_{\delta(c)} = \{c\}$,
and hence $\bbalpha_{\mathfrak{A}',\delta(c)}^{-} =
\emptyset = \bbalpha_{\mathfrak{A}',\delta(c)}^{+}$, as required.
\end{proof}
\end{lemma}

%
%
%
%
%
%
%
%
%
\subsection{Pseudo-ordering axioms}\label{pseudo-orderingaxiomssect}
Let $\chi$ be a normal form sentence 
of $\mathrm{U}_{1}$ with the set $\sigma$ of relation symbols.
We assume that the symbol $<$ occurs in $\chi$.
Let $r$ be the highest arity occurring in the
symbols in $\sigma$, and let $n$ be the width of $\chi$.
Let $m_{\exists}$ be the number of existential conjuncts in $\chi$. Assume\
%
%
%
$$\Gamma_{\chi} = (\mathfrak{C}, (\blet{\alpha}_{\sigma,s})_{1\leq s \leq N}
, \blet{\alpha}_{\sigma}^{K},
\blet{\alpha}_{\sigma}^{\bot}, \blet{\alpha}_{\sigma}^{\top},
\delta, F)$$
%
%
%
is some admissibility tuple in $\hat{\Gamma}_{\chi}$.
In this subsection we construct a certain large sentence 
$\mathit{Ax}(\Gamma_{\chi})$ that axiomatizes structures
with properties given by $\Gamma_{\chi}$. The ultimate use of the
sentence $\mathit{Ax}(\Gamma_{\chi})$ will be revealed by the statement of
Lemma \ref{DecidabilityLemma2}, which is one of
our main technical results.
Note that in that lemma, satisfiability of $\mathit{Ax}(\Gamma_{\chi})$ is
considered in relation to classes of models where the symbol $<$ is not
necessarily interpreted as a linear order.
Let $K$, $D$, $P_{\bot}$, $P_{\top}$, and $U_s$ for each $s \in \{1,...\, , N\}$
be fresh unary relation symbols, where $N$ is the size of the family
$(\bbalpha_{\sigma,s})_{1\leq s \leq N}$ in $\Gamma_{\chi}$.
Intuitively, the relation symbols $K$ and $D$
correspond to a set of kings and a set of domains of
free witness structures, respectively, as we shall see.
The symbols $U_s$, for $1 \leq s \leq N$, correspond to
intervals, but this intuition is not precise as we 
shall interpret the predicates $U_s$ over models
where $<$ is not assumed to be a linear order.
%
%
%
%
%
The predicates $P_{\bot}$ and $P_{\top}$
will be axiomatized to contain the minimal and the maximal
realization of each 1-type belonging to $\bbalpha_{\sigma}^{\bot}$ and 
$\bbalpha_{\sigma}^{\top}$, respectively.
%
%

%
Let $\sigma'$ be the vocabulary $\sigma \cup \{K, D, P_{\bot}, P_{\top} \} \cup \{U_{s}\mid 1\leq s \leq N\}$. 
We define the \textit{pseudo-ordering axioms for $\Gamma_{\chi}$} (over $\sigma'$) as
follows. For most axioms we also give an informal description of its
meaning (when interpreted together with the other pseudo-ordering axioms).
Each of the 16 axioms is a $\mathrm{U}_1$-sentence in normal form.
\begin{enumerate}
  \item \label{POAchi} $\chi$
%
%
  \item \label{POApartition} The predicates $U_s$ partition the universe:
 $\\ \bigwedge\limits_{s} \exists x\, U_{s}x\ \wedge\ \forall x \big(\, \bigvee\limits_{s} (\, U_{s}x \wedge 
  \bigwedge\limits_{t \neq s} \neg U_{t}x)\, \big)$
  \item \label{POArealizations} For all $s\in\{1,...\, , N\}$, the elements in $U_{s}$ realize exactly the 1-types (over $\sigma$) in $\bbalpha_{\sigma,s}$:
  $\\\bigwedge \limits_{1\leq s \leq N} \forall x ( \, U_{s}x
   \hspace{0.1cm} \leftrightarrow  \hspace{0.1cm} \bigvee_{\alpha \in \bbalpha_{\, \sigma,s}} \alpha(x) \,)$\\
  Note indeed that the $1$-types $\alpha$ in $\bbalpha_{\sigma,s}$ are with respect to the vocabulary $\sigma$,
  and thus are definitely not $1$-types with respect to the extended vocabulary $\sigma'$.
  \item  \label{POAsingleton} Each predicate $U_{\delta(c)}$, where $c$ is an element in
   the domain $C$ of $\mathfrak{C}$, is a singleton set
  containing an element that realizes $\alpha = tp_{\mathfrak{C}}(c):\\$
  $\bigwedge \limits_{c \in C,\, \alpha = tp_{\mathfrak{C}}(c)} \bigl(\, \exists y (\, U_{\delta(c)}y \wedge \alpha(y)\, )
\hspace{0.1cm} \wedge \hspace{0.1cm} \forall x\forall y (\, (\,U_{\delta(c)}x \wedge U_{\delta(c)}y\,) \\
  {\hspace{1.8cm} \rightarrow x = y} \,) \, \bigr)$
  \item \label{POApawns} Each $\alpha \in (\bigcup \bbalpha_{\sigma,s} \setminus \bbalpha_{\sigma}^{K})$ is realized at least $n$  times (recall that $n$ is the width of $\chi$):
  $\\ \bigwedge \limits_{\alpha \in (\bigcup \bbalpha_{\sigma,s} \setminus \bbalpha_{\sigma}^{K})} 
  \exists x_1...\exists x_n(\, \bigwedge \limits_{i \neq j} (x_i \neq x_j) \wedge
  \bigwedge \limits_{i} \alpha(x_i) \,)$
  \item \label{POAkings} Each $\alpha \in \bbalpha_{\sigma}^{K}$ is realized at least once but at most $n-1$ times: 
  $\\ \bigwedge \limits_{\alpha \in \bbalpha_{\sigma}^{K}} \exists y\, \alpha(y) \wedge \forall x_1...\forall x_{n}  
  \bigl(\,(\,\bigwedge \limits_{i} \alpha(x_i)\,) \rightarrow \bigvee \limits_{j\neq k} x_j = x_k \,\bigr)$
\item \label{POAkingscourt}
$K$ is the set of realizations of types in $\bbalpha_{\sigma}^K$: \\
       $\forall x \bigl(\, (\bigvee\limits_{\alpha\, \in\,
              \bbalpha_{\sigma}^K} \alpha(x)\, ) \leftrightarrow Kx \, \bigr)$
  \item \label{POArunningfreeyeah}
In order to define the next axiom, we begin
with an auxiliary definition. For each 
existential matrix $\chi_{i}^{\exists}(x,y_1,...\, , y_{k_i})$ in $\chi$,
let the set $\{z_1,...\, , z_{l_i}\}\subseteq \{x, y_1,...\, , y_{k_i}\}$
be the set of live variables of $\chi_{i}^{\exists}(x,y_1,...\, , y_{k_i})$.
We then define the following axiom which asserts that
the set $F$ is the set of
all pairs $(\alpha,\chi_i^{\exists})$ that have a witness structure
whose live part is free, and furthermore, the set $D$ contains,
for each pair $(\alpha,\chi_i^{\exists})\in F$, the live part of at
least one free witness structure for $(\alpha,\chi_i^{\exists})$.\\ 
$\bigwedge\limits_{(\alpha,\chi_i^{\exists})\in F}
\exists x\exists y_{1}...\exists y_{k_i} \bigl(\,
\alpha(x)\wedge\chi_{i}^{\exists}(x,y_1,...,y_{k_i}\, ) \wedge \\
\bigwedge\limits_{1\leq j \leq l_i}(\, z_j\neq x\wedge D z_j\,) \, \bigr)
\hspace{0.1cm}$ \\
$\wedge\bigwedge\limits_{(\alpha,\chi_i^{\exists})\not\in F}
\forall x\forall y_{1}...\forall y_{k_i}  \bigl( \,
\neg\, (\,\alpha(x)\wedge\chi_{i}^{\exists}(x,y_1,...,y_{k_i}\,) \\
\wedge\bigwedge\limits_{1\leq j \leq l_i} z_j\neq x\,) \, \bigr)$
%
%
%
%
%
%
%
%
%
%
%
%
%
%
%
%
%
  \item \label{POAfreecourt} Axioms \ref{POAkings} and \ref{POAkingscourt} define the set $K$, 
  and $D$ is described by the previous axiom. The next axiom says that every element $c \in (K \cup D)$
  is in $\bigcup_{c \in C}U_{\delta(c)}$:
  $\forall x (\, (\, Kx \vee Dx \,) \rightarrow \bigvee\limits_{c\in C} U_{\delta(c)}x \,)$
  \item \label{POAkingswitnesses} There is a witness structure for every $c \in (K\cup D)$ such that 
  each element of the witness structure is in $\bigcup_{c \in C}U_{\delta(c)}$:
  $\\\bigwedge \limits_{1\leq i \leq m_{\exists}}  \forall x \exists y_1...\exists y_{k_i} 
  \big(  \hspace{0.1cm} (Kx \vee Dx) \hspace{0.1cm}\rightarrow \\ \hspace{0.1cm}
  (\,(\,\bigwedge \limits_{1\leq j \leq k_{i} } \bigvee \limits_{c\in C} U_{\delta(c)} y_j\, )
  \wedge \chi_{i}^{\exists}(x,y_1,...\, ,y_{k_i})\,) \hspace{0.1cm} \bigr)$
\item \label{POAcourt}
The next axiom ensures that there exists an isomorphic copy of $\mathfrak{C}$ in the
model considered. Let $m = min\{n,r\}$, where $r$ is the maximum arity of
relation symbols that occur in $\chi$. For each $k\in\{1,...\, , m\}$, let $\mathcal{C}_k$
denote the set of all subsets of size $k$
of the domain $C$ of $\mathfrak{C}$.
Let $\bar{\mathcal{C}_k}$ denote the set obtained from $\mathcal{C}_k$
by replacing each set $C_k\in\mathcal{C}_k$ by exactly one tuple $(c_1,...\, , c_k)$
that enumerates the elements of $C_k$ in some
arbitrarily chosen order. (Thus $|\mathcal{C}_k| = |\bar{\mathcal{C}_k}|$.)
For each tuple $(c_1,...\, ,c_k)\in\bar{\mathcal{C}_k}$, let $\beta_{[(c_1,...,c_k)]}$ denote the 
table $tb_{\mathfrak{C}}(c_1,...\, , c_k)$.
%
%
%
%
We define the required axiom as follows:
$\bigwedge\limits_{1\leq k\leq m}\bigwedge\limits_{(c_1,...,c_k)\, \in\, \bar{\mathcal{C}}_k}
\forall x_1...\forall x_{k} \big( \\ \hspace{0.1cm} (\,\bigwedge \limits_{c_j\, \in\, (c_1,...,c_k)} 
  U_{\delta(c_j)}x_j \,) 
\hspace{0.1cm}\rightarrow \hspace{0.1cm} \beta_{[(c_1,...,c_k)]}(x_1,...\, , x_{k}) \hspace{0.1cm} \big)$ \\
Note that strictly speaking the axiom ignores sets of size
greater than\ \nolinebreak $m$.
  \item \label{POAtournament}
   The relation symbol $<$ is interpreted to be a tournament:
\\ $\forall x \forall y (\, x < y \vee y < x \vee x=y \,)
  \wedge \forall x \forall y \, \neg(\,x < y \wedge y < x \, )$
%
%
%
%
%
%
%
\item \label{POAmin}
Together with the previous axiom, the \emph{first three} big conjunctions of the next axiom
imply that for all $\alpha\in\bbalpha_{\sigma}^{\bot}$ there exists  a
point in $P_{\bot}$ that realizes $\alpha$, and furthermore,
$P_{\bot}$ is true at a point $u$ iff
there exists a $1$-type $\alpha$ such
that $\alpha\in\bbalpha_{\sigma}^{\bot}$
and $u$ is the unique minimal realization of
that $1$-type. The \emph{last} big conjunction of the axiom 
implies that if $\alpha\in\bbalpha_{\sigma,s}^-\cap\, \bbalpha_{\sigma}^{\bot}$
for some $s\in\{1,...\, , N\}$, then there exists a point $u'$ 
which is the minimal realization of $\alpha$ and satisfies $U_s$:\\
$\bigwedge\limits_{\alpha\, \in\, \bbalpha_{\sigma}
\setminus\bbalpha_{\sigma}^{\bot}}\forall x\, \neg\, (\, \alpha(x)\wedge P_{\bot}(x)\, ) 
\\ \wedge \bigwedge\limits_{\alpha\, \in\,
\bbalpha_{\sigma}^{\bot}}\bigl( \hspace{0.1cm} \exists x(\, \alpha(x)\wedge P_{\bot}x\, ) \, \bigr) \\
\wedge\bigwedge\limits_{\alpha\, \in\, \bbalpha_{\sigma}^{\bot}}
{\forall x \forall y \bigl( \, ( \,P_{\bot}x \wedge \alpha(x) 
  \wedge \alpha(y) \wedge y\neq x \, ) \rightarrow x < y \, \bigr)} \\
\wedge\bigwedge\limits_{\alpha\, \in\, 
\bbalpha_{\sigma,s}^{-}\cap\, \bbalpha_{\sigma}^{\bot}}
\exists x (\, P_{\bot}x\wedge \alpha(x)\wedge U_s x\, )$

\blue{The above axiom could be expressed simply as:
For all $\alpha \in \bbalpha_{\sigma,s}^{-} \cap \bbalpha_{\sigma}^{\bot}$:
$\exists x ( min_{\alpha}(x) \wedge U_{s}x )$.
However, this is not in normal form, and translating it into
normal form makes it look somewhat cumbersome.}
\item \label{POAmax}
The next axiom is analogous to the previous one:
$\\ \bigwedge\limits_{\alpha\, \in\, \bbalpha_{\sigma}
\setminus\bbalpha_{\sigma}^{\top}}\forall x\, \neg\, (\, \alpha(x)\wedge P_{\top}(x)\, )\\
\wedge\bigwedge\limits_{\alpha\, \in\,
\bbalpha_{\sigma}^{\top}}\bigl(\hspace{0.1cm} \exists x(\, \alpha(x)\wedge P_{\top}x \, ) \, \bigr) \\
\wedge\bigwedge\limits_{\alpha\, \in\, \bbalpha_{\sigma}^{\top}}
\forall x \forall y \bigl( \, (\, P_{\top}x \wedge \alpha(x) 
  \wedge \alpha(y) \wedge y\neq x \, ) \rightarrow x > y \, \bigr) \\
\wedge \bigwedge\limits_{\alpha\, \in\, 
\bbalpha_{\sigma,s}^{+}\cap\, \bbalpha_{\sigma}^{\top}}
\exists x (\, P_{\top}x\wedge \alpha(x)\wedge U_s x)\, $

\blue{Or equivalently but not in normal form:
For all $\alpha \in \bbalpha_{\sigma,s}^{+} \cap \bbalpha_{\sigma}^{\bot}$:
$\exists x ( max_{\alpha}(x) \wedge U_{s}x )$.}
%
%
%
%
%
%
%
%
%
\end{enumerate}
The last two axioms below are
technical assertions about the predicates $U_s$, the relation $<$
and $1$-types. The significance of these axioms becomes
clarified in the related proofs.
\begin{enumerate}\setcounter{enumi}{14}
  \item \label{POAquasiorder}$\bigwedge \limits_{s,t\in\{1,...,N\},\, s<t} \forall x \forall y\big(\, (\,U_{s}x
   \wedge U_{t}y\,) \hspace{0.1cm} \rightarrow \hspace{0.1cm} x < y \, \big)$
  \item \label{POAadmissibility4}
$\bigwedge \limits_{s\, \in\, \{1,...,N\} \setminus img(\delta) }
  \bigwedge \limits_{\alpha\, \in\, \bbalpha_{\sigma, s}^{+}}\bigwedge\limits_{\alpha'\, \in\, \bbalpha_{\sigma, s}}
 \exists x \exists y \bigl( \\ \, \alpha(x) \wedge \alpha'(y) \wedge U_{s}x \wedge U_{s}y \wedge y
   < x \, \bigr)$
%
%
%
%
%
%
\end{enumerate}
%
%
%
%
%
%
We denote the conjunction of the above 16 pseudo-ordering axioms over $\sigma'$ for
the admissibility tuple $\Gamma_{\chi}$ by $\mathit{Ax}(\Gamma_{\chi})$.
We note that $\mathit{Ax}(\Gamma_{\chi})$ is a normal form
sentence of $\mathrm{U}_1$ over the vocabulary $\sigma'$ which expands
the vocabulary $\sigma$ of $\chi$.
The formulae $\mathit{Ax}(\Gamma_{\chi})$ play a central role in
the reduction of ordered satisfiability to
standard satisfiability based on Lemma \ref{DecidabilityLemma2}. 
%
%
%
%
%
%
%
%
\begin{lemma}\label{thefact}
Let $\varphi$ be a normal form $\mathrm{U}_1$-sentence
with the set $\tau$ of relation symbols, $<\in\tau$.
Let $\mathfrak{A}\models\varphi$ be a $\tau$-model.
Let $\mathfrak{C}$ be a court of\, $\mathfrak{A}$ w.r.t. $\varphi$
and $\mathfrak{A}'$ a $3$-cloning extension of\, $\mathfrak{A}$ w.r.t. $\varphi$.
Let $\Gamma_{\varphi}^{\mathfrak{C},\mathfrak{A},\mathfrak{A}'}$ be a
canonical admissibility tuple for $\mathfrak{A}'$ w.r.t. $(\mathfrak{C},
\mathfrak{A},\varphi)$
and $N$ the index of\, $\Gamma_{\varphi}^{\mathfrak{C},\mathfrak{A},\mathfrak{A}'}$.
Then the $\tau$-model $\mathfrak{A}'$ has an expansion $\mathfrak{A}''$ to
the vocabulary $\tau\cup\{K,D,P_{\bot},P_{\bot}\}\cup\{U_s\, |\, 1\leq s \leq N\}$
such that $\mathfrak{A}''\models
\mathit{Ax}(\Gamma_{\varphi}^{\mathfrak{C},\mathfrak{A},\mathfrak{A}'})$.

\begin{proof}
Recall the $\tau$-sentence $\varphi$ and
the $\tau$-model $\mathfrak{A}\models\varphi$ fixed in Section \ref{Analysing}.
Recall also the $3$-cloning
extension $\mathfrak{A}'$ of $\mathfrak{A}$ w.r.t. $\varphi$
and the court $\mathfrak{C}$ of $\mathfrak{A}$ w.r.t. $\varphi$
fixed in that Section.
Let $\Gamma_{\varphi}^{\mathfrak{C},
\mathfrak{A},\mathfrak{A}'}$ be 
the canonical admissibility tuple of $\mathfrak{A}'$ w.r.t.
$(\mathfrak{C},\mathfrak{A},\varphi)$.
Note that by Lemma \ref{technlemma}, we
have $\Gamma_{\varphi}^{\mathfrak{C},
\mathfrak{A},\mathfrak{A}'}\in\hat{\Gamma}_{\varphi}$, 
and furthermore, $\Gamma_{\varphi}^{\mathfrak{C},\mathfrak{A},\mathfrak{A}'}$ is
admissible for each $\mathcal{K}\in\{\mathcal{O},\mathcal{WO},\mathcal{O}_{fin}\}$
such that $\mathfrak{A}\in\mathcal{K}$.
%
%
We will show that $\mathfrak{A}'$ has an 
expansion $\mathfrak{A}''$ such that $\mathfrak{A}''\models
\mathit{Ax}(\Gamma_{\varphi}^{\mathfrak{C},\mathfrak{A},\mathfrak{A}'})$.
As $\varphi$, $\mathfrak{A}$, $\mathfrak{A}'$ and $\mathfrak{C}$
%
%
%
were fixed arbitrarily, this proves the
current lemma ( Lemma \ref{thefact}).
Let $N$ be the index of $\Gamma_{\varphi}^{\mathfrak{C},\mathfrak{A},\mathfrak{A}'}$, in
other words, $N$ is the
size of the family $(I_s)_{1\leq s\leq N}$ of intervals
fixed in Section \ref{intervalpartitionsection}.
Thus we now must prove that $\mathfrak{A}'$
has an expansion $\mathfrak{A}''$ to 
the vocabulary $\tau\cup\{K,D,P_{\bot},P_{\bot}\}\cup\{U_s\, |\, 1\leq s \leq N\}$
such that $\mathfrak{A}''\models \mathit{Ax}(
\Gamma_{\varphi}^{\mathfrak{C},\mathfrak{A},\mathfrak{A}'})$.
We let $\mathfrak{A}''$ be the expansion of $\mathfrak{A}'$
obtained by interpreting the extra predicates $\{K,D,P_{\bot}, P_{\top} \} \cup
\{\, U_{s}\mid 1\leq s \leq N\}$ as follows.
%
%
\begin{enumerate}
\item
$K^{\mathfrak{A}''}$ and $D^{\mathfrak{A}''}$ are
defined as $K$ and $D$ in the Section \ref{courtsection}, respectively.
Thus $K^{\mathfrak{A}''}\subseteq A$ is the set of
kings in $\mathfrak{A}'$ (and $\mathfrak{A}$)
and $D^{\mathfrak{A}''}\subseteq A$ is a set that contains,
for every pair $(\alpha,\varphi_i^{\exists})$ that has a free witness
structure in $\mathfrak{A}$, the free part of at least
one such witness structure (cf. Section \ref{courtsection}).
\item
$P_{\bot}^{\mathfrak{A}''}$ is defined to satisfy the 
pseudo-ordering axiom \ref{POAmin}; we let $P_{\bot}^{\mathfrak{A}''}$ be
true at a point $u$ iff there is some $1$-type $\alpha$ such that $u$ is
the minimal realization of $\alpha$. $P_{\top}^{\mathfrak{A}''}$ is
defined analogously to satisfy axiom \ref{POAmax}.
\item
Each predicate $U_s^{\mathfrak{A}''}$ is
defined to be the interval $I_s\subseteq A'$
identified in Section \ref{intervalpartitionsection}.
\end{enumerate}
Next we show that $\mathfrak{A}''
\models\mathit{Ax}(\Gamma_{\varphi}^{\mathfrak{C},\mathfrak{A},\mathfrak{A}'})$.
As it is easy to see that $\mathfrak{A}''$ satisfies axioms 1-7 and 9-16,
it suffices to show that $\mathfrak{A}''$ satisfies axiom \ref{POArunningfreeyeah}.
Recalling the definition of $D^{\mathfrak{A}''}$,
this can clearly be done by proving the following claim.
(Recall (cf. Section \ref{admissibilitytuplessection}) that 
$F$ is the set of those pairs $(\alpha,\varphi_i^{\exists})$
that have a free witness structure in $\mathfrak{A}'$.)

$\bm{Claim}$: $\mathfrak{A}$ has a free witness structure for a pair $(\alpha,\varphi_{i}^{\exists})$
iff $(\alpha,\varphi_{i}^{\exists}) \in F$.

As $\mathfrak{A}'$ is a $3$-cloning extension of $\mathfrak{A}$, 
it is clear that $\mathfrak{A}'$ 
has a free witness structure for a pair $(\alpha,\varphi_{i}^{\exists})$ if $\mathfrak{A}$ has.
Suppose now that for some $a \in A'$, $\mathfrak{A}'$ has a free witness structure 
$\mathfrak{A}'_{a,\varphi_{i}^{\exists} }$ for some $(\alpha,\varphi_{i}^{\exists}) \in F$
and $\mathfrak{A}$
does not have a free witness structure for this pair.
Let $\mathfrak{A}'_{a,\varphi_{i}^{\exists} } \models \varphi_{i}^{\exists}(a,a_1,...\, ,a_{k_i})$ for some points
$a_1,...\, ,a_{k_i} \in A'$, which are not necessarily distinct.
Let $u_1,...\, ,u_{l} \in (A'_{a,\varphi_{i}^{\exists}} \setminus \{a\})$ be the distinct points 
forming the live part of $\mathfrak{A}'_{a,\varphi_{i}^{\exists} }$.
Thus some points $a_1,...\, ,a_{k'} \in (A'_{a,\varphi_{i}^{\exists} }
\setminus \{u_1,...\, ,u_{l}\})$
together with  $a$ form the dead part of $\mathfrak{A}'_{a,\varphi_{i}^{\exists} }$.

The table $tb_{\mathfrak{A}'}( u_1,...\, ,u_{l})$ has
been defined either in the \emph{cloning stage} or the
\emph{completion stage} to be $tb_{\mathfrak{A}}( b_1,...\, ,b_{l})$
for some distinct elements $b_1,...\, ,b_{l}\in A$.
Furthermore, since $\mathfrak{A}'$ and $\mathfrak{A}$
have exactly the same number of realizations of each royal $1$-type 
and since both models have at least $n\geq k_{i}+1$ realizations of
each pawn, it is easy to define an
injection $f$ from $A'_{a,\varphi_{i}^{\exists}}$ into $A$
that preserves $1$-types and such that $f(u_i) = b_i$ for each $i\in\{1,...\, , l\}$.
Therefore $\mathfrak{A}'\models \varphi_{i}^{\exists}( a,a_1,...\, ,a_{k_i} )$
iff $\mathfrak{A}\models \varphi_{i}^{\exists}( f(a),f(a_1),...\, ,f(a_{k_i}) )$,
whence we have $\mathfrak{A}\models
\varphi_{i}^{\exists}( f(a),f(a_1),...\, ,f(a_{k_i}) )$.
Therefore, as $f$ is injective, we see that $\mathfrak{A}$ has a free
witness structure for $(\alpha,\varphi_i^{\exists})$.
This contradicts the assumption that $\mathfrak{A}$
does not have a free witness structure for
the pair $(\alpha,\varphi_{i}^{\exists})$.
%
%
%
%
%
%
%
%
%
\end{proof}
\end{lemma}

\section{Reducing ordered satisfiability to standard satisfiability}\label{reducing}
In this section we establish decidability of the
satisfiability problems of $\mathrm{U}_1$ over $\mathcal{O}$, $\mathcal{WO}$
and $\mathcal{O}_{fin}$. The next lemma (Lemma \ref{DecidabilityLemma2}) is
the main technical result needed for the decision procedure.
%
%
%
%
%
Note that satisfiability in the case (b) of the lemma is with respect to
general rather than ordered models. In the lemma we
assume w.l.o.g. that $\varphi$ contains $<$.
%
%
%
%
%
\begin{lemma} \label{DecidabilityLemma2}
Let $\varphi$ be a $\mathrm{U}_1$-sentence containing the symbol $<$.
Let $\mathcal{K} \in \{\mathcal{O}, \mathcal{WO}, \mathcal{O}_{fin} \}$.
The following conditions are equivalent:
\begin{itemize}
\item[(a)]$\varphi \in sat_{\mathcal{K}}(\mathrm{U}_{1})$.
\item[(b)]$\mathit{Ax}(\Gamma_{\varphi})\in sat(\mathrm{U}_{1})$ for
some admissibility tuple
$\Gamma_{\varphi}\in\hat{\Gamma}_{\varphi}$ that is admissible for $\mathcal{K}$.
\end{itemize}
\end{lemma}
\begin{proof}
In order to prove the implication from (a) to (b),
suppose that $\varphi \in sat_{\mathcal{K}}(\mathrm{U}_{1})$.
Thus there is a structure $\mathfrak{A} \in \mathcal{K}$
such that $\mathfrak{A}\models\varphi$. As $\mathfrak{A}
\models\varphi$, there exists a
court $\mathfrak{C}$ of $\mathfrak{A}$ w.r.t. $\varphi$.
Now let $\mathfrak{A}'$ be a $3$-cloning
extension of $\mathfrak{A}$ w.r.t. $\varphi$,
and let $\Gamma_{\varphi}^{\mathfrak{C},\mathfrak{A}'\mathfrak{A}'}$ be the 
canonical admissibility tuple of $\mathfrak{A}'$ w.r.t. $(\mathfrak{C},
\mathfrak{A},\varphi)$. By Lemma \ref{technlemma}, the
canonical tuple is in $\hat{\Gamma}_{\varphi}$ and admissible for $\mathcal{K}$.
By Lemma \ref{thefact}, $\mathfrak{A}'$ has an expansion $\mathfrak{A}''$
such that $\mathfrak{A}''\models\mathit{Ax}(
\Gamma_{\varphi}^{\mathfrak{C},
\mathfrak{A},\mathfrak{A}'})$.
The proof for the direction from (b) to (a) is given in Appendix \ref{proofofreduction}.
%
%
\end{proof}
The following gives a brief description of the
decision process which is also
outlined in Figure \ref{pseudocodefig}.
A complete and rigorous treatment of related details is given in 
Appendix \ref{complexitysection} which is 
devoted to the proof of Theorem \ref{nexptimethmmaintext}.
%
%
%
%
%
%
%
%
%
%
%
%
\begin{enumerate}
\item
An input to the problem is a sentence $\psi'$ of $\mathrm{U}_1$,
which is immediately converted into a normal form
sentence $\psi$ of $\mathrm{U}_1$ (cf. Proposition \ref{PreliminariesNormalForm}).
\item
Based on $\psi$, an admissibility tuple $\Gamma_{\psi}\in\hat{\Gamma}_{\psi}$ is
guessed non-deterministically.
The size of the tuple is exponential in $|\psi|$ (cf. Lemma
\ref{th!MostImportantProofInYourLife}). It is then
checked whether the tuple is admissible for the 
class $\mathcal{K}\in\{\mathcal{O},\mathcal{WO},\mathcal{O}_{fin}\}$
whose decision problem we are considering.
\item
Based on $\Gamma_{\psi}$, the sentence $\mathit{Ax}(\Gamma_{\psi})$ is
produced. The length of $\mathit{Ax}(\Gamma_{\psi})$ is
exponential in $|\psi|$ (cf. Lemma \ref{sentencelimitlemma} in 
Appendix \ref{complexitysection}).
\item
Then a model $\mathfrak{B}$, whose description is
exponential in $|\psi|$ (cf. Lemma
\ref{ComplexityLemma1} in Appendix \ref{complexitysection}), is
guessed. It is then checked whether $\mathfrak{B}
\models\mathit{Ax}(\Gamma_{\psi})$, which can be done in
exponential time in $|\psi|$ (cf. the end of
Appendix \ref{complexitysection}).
\end{enumerate}
\begin{figure}[!ht]
\begin{algorithmic}[1]
\Procedure{Satisfiability} \empty $(\psi')$ over $\mathcal{K}$.
\Comment{\emph{The $\mathrm{U_1}$-sentence $\psi'$ is an input to the algorithm.
Here $\mathcal{K}\in\{\mathcal{O},\mathcal{WO},\mathcal{O}_{fin}\}$, so
we are outlining three procedures in parallel.}}
\State Construct a normal form sentence $\psi$ of $\mathrm{U_1}$ from $\psi'$.
Let $\tau$ be the vocabulary consisting of all the relation symbols
occurring in $\psi$.
\label{normalform}
\Comment{\emph{By Proposition \ref{PreliminariesNormalForm}, it holds that $\psi$ is satisfiable iff $\psi'$ is satisfiable.}}
\State Guess $\Gamma_{\psi} \in \hat{\Gamma}_{\psi}$ and check
that $\Gamma_{\psi}$ is an admissibility tuple admissible for $\mathcal{K}$.
\label{guess1}
\State Let $\tau':=\tau \cup \{U_s\mid s\in\{1,...\, , N\} \} \cup \{K,D,P_{\bot}, P_{\top}\}$.
Formulate the pseudo-ordering axioms for $\Gamma_{\psi}$ over $\tau'$
and let $\mathit{Ax}(\Gamma_{\psi})$ be the conjunction of these axioms.
\label{psigamma} 
\Comment{\emph{Note that $\mathit{Ax}(\Gamma_{\psi})$ is in normal form.}}
\State Guess a potential model $\mathfrak{B}$ of $\mathit{Ax}(\Gamma_\psi)$
whose size is exponentially bounded in $|\psi|$.\label{guess2}\;
\Comment{\emph{In the next lines it is
checked whether $\mathfrak{B} \models \mathit{Ax}(\Gamma_\psi)$.
Note that by Lemma \ref{DecidabilityLemma2}, if $\mathfrak{B}
\models \mathit{Ax}(\Gamma_\psi)$,
then $\psi\in\mathit{sat}_{\mathcal{K}}(\psi)$.}}
%
%
%
\ForAll{$b \in B$}
  \ForAll{existential conjuncts $\chi:=$ \newline ${\hspace{55pt}\forall x\exists y_1...\exists y_i\beta(x,y_1,...,y_i)}$ of $\mathit{Ax}(\Gamma_{\psi})$ }
    \State Guess elements $b_1',...\, ,b_l'$ in $B$ to form a witness \newline \text{ }
     {$\hspace{34pt} \text{ structure } \mathfrak{B}_{b,\chi} \text{ and }$}
    \State check whether
$\mathfrak{B} \models \beta(b,b_1',...\, , b_l')$.
\label{Eop}
  \EndFor
\EndFor
\ForAll{ universal conjuncts $\forall x_1...\forall x_{l'} 
\beta'(x_1,...\, x_{l'})$\newline \text{ } {$\hspace{39pt}  \text{of } \mathit{Ax}(\Gamma_{\psi})$} }
  \ForAll{ tuples $(b_1,...\, , b_{l'})$ of elements of $B$,}
    \State Check whether $\mathfrak{B} \models \beta'(b_1,\ldots,b_{l'})$.
  \EndFor
\EndFor
\label{Aop}
\EndProcedure
\end{algorithmic}
\caption{Solving satisfiability of $\mathrm{U_1}$
over $\mathcal{K}\in\{\mathcal{O},\mathcal{WO},\mathcal{O}_{fin}\}.$
The symbol $\triangleright$ indicates comment.}
\label{pseudocodefig}
\end{figure}

\begin{theorem}\label{nexptimethmmaintext}
Let $\mathcal{K}\in\{\mathcal{O},\mathcal{WO},\mathcal{O}_{fin}\}$.
The satisfiability problem for $\mathrm{U}_1$ over $\mathcal{K}$
is \textsc{NExpTime}-complete.
\end{theorem}
\begin{proof}
The lower bound (for each of the three decision problems)
follows immediately from \cite{DBLP:journals/jsyml/Otto01}. The remaining 
part of the proof is given in Appendix \ref{complexitysection}.
\end{proof}

\section{Undecidable extensions}\label{undecidableextensions}
The satisfiability problem for $\mathrm{FO^2}$ over structures 
with three linear orders is
undecidable \cite{DBLP:conf/csl/Kieronski11}. 
%
On the other hand, while the finite satisfiability problem for $\mathrm{FO^2}$ 
over structures with two linear orders
is decidable and in \textsc{2NExpTime} \cite{zeume}, 
the general satisfiability problem 
for $\mathrm{FO^2}$ with two linear orders
(and otherwise
unrestricted vocabulary) is open. 
These results raise the question whether 
the satisfiability problem for the extension $\mathrm{U}_{1}[<_{1},<_{2}]$ 
of $\mathrm{U}_{1}$ (see Section \ref{structureclassessection}) 
over structures with two linear orders is decidable. 
We use tiling arguments to answer this question in the negative;
see Appendix \ref{appedixtiling} for a proof of the following theorem.
\begin{theorem} \label{U1twosucc}
The satisfiability problem for $\mathrm{U}_{1}[<_{1},<_{2}]$ over structures
with two built-in linear orders is undecidable.
\end{theorem}
We note that $\mathrm{U}_{1}[\sim_{1},\sim_{2}]$,
where $\sim_{1}$ and $\sim_{2}$ denote built-in 
equivalence relations, is decidable and complete for $\textsc{2NExpTime}$
\cite{DBLP:conf/csl/KieronskiK15}.

\section{Conclusion}
We have shown that $\mathrm{U}_1$ is \textsc{NExpTime}-complete
over ordered, well-ordered and finite ordered structures.
To contrast these results, we have established that $\mathrm{U}_1[<_1,<_2]$ is
undecidable. 
The results here are the first results concerning $\mathrm{U}_1$
with built-in linear orders. 
Several open problems remain, e.g., investigating $\mathrm{U}_1$
with combinations of equivalence relations and linear orders. 
Such results would contribute in a natural way to the active research program
concerning $\mathrm{FO}^2$ with built-in relations and push the field
towards investigating frameworks with relation symbols of arbitrary arity.


\bibliography{lics18}

\appendix




 


%
%
%
\section{Proof of Lemma \ref{DecidabilityLemma2}} \label{proofofreduction}

Please note that the proof of Lemma \ref{DecidabilityLemma2} below spans all of
the current section,
ending at the end of Subsection \ref{completionproceduresubsection}.

\begin{proof}
The implication from (a) to (b) was 
proved in the main text, so we prove the implication from (b) to (a) here.
We deal with the three cases
$\mathcal{K} \in \{\mathcal{O}, \mathcal{WO}, \mathcal{O}_{fin}\}$
in parallel.
We let $\tau$ denote the set of relation symbols in $\varphi$.
To prove the implication from (b) to (a),
assume that $\mathfrak{B}\models\mathit{Ax}(\Gamma_{\varphi})$
for some $\tau'$-model $\mathfrak{B}$ and
some admissibility tuple

\smallskip

$\Gamma_{\varphi} = (\mathfrak{C}, 
(\blet{\alpha}_{\tau,s})_{1\leq s \leq N}, \blet{\alpha}_{\tau}^{K}, 
\blet{\alpha}_{\tau}^{\bot}, \blet{\alpha}_{\tau}^{\top}, \delta, F)
\in \hat{\Gamma}_{\varphi}$

\smallskip

\noindent
that is admissible for
the class $\mathcal{K}$.
Here $\tau' = \tau\cup\{K,D,P_{\bot},P_{\top}\}\cup\{U_s|1\leq N\leq U_s\}$.
%
%
%
%
Note that while $\mathfrak{B}$ interprets the symbol $<$,
it is not assumed to be an ordered model.
Based on $\mathfrak{B}$ and $\Gamma_{\varphi}$,
we will construct an ordered $\tau$-model $\mathfrak{A}\in\mathcal{K}$
such that $\mathfrak{A} \models \varphi$.
The construction of $\mathfrak{A}$
consists of the following (informally described) four steps; each step is
described in full detail in its own subsection below.\vspace{1mm}
\noindent
\textbf{\textit{1})} We first construct the domain $A$ of $\mathfrak{A}$ and
define a linear order $<$ over it. We also label the elements of $A$
with $1$-types in $\bbalpha_{\tau}$.
After this stage the relations of $\mathfrak{A}$ (other than $<$)
contain no tuples other than trivial
tuples, i.e., tuples $(u,...\, ,u)$ with $u$ repeated.
\vspace{0.6mm}
\noindent
\textbf{\textit{2})} We then copy a
certain substructure $\mathfrak{C}$ of $\mathfrak{B}$
into $\mathfrak{A}$; the structure $\mathfrak{C}$ is the set of 
points in $B$ that satisfy some predicate $U_s$ with $s\in\mathit{img}(\delta)$.
This step introduces fresh non-trivial tuples
into the relations of $\mathfrak{A}$.
%
%
%
\vspace{0.6mm}
\noindent
\textbf{\textit{3})} We then define a witness structure
for each element $a \in A$ and each existential 
conjunct $\varphi_{i}^{\exists}$ of $\varphi$.
As the above step, this step introduces non-trivial tuples into
the relations of $\mathfrak{A}$.\vspace{0.6mm}

\noindent
\textbf{\textit{4})} Finally, we complete the construction of $\mathfrak{A}$
by making sure that $\mathfrak{A}$ also satisfies 
all universal conjuncts $\varphi_{i}^{\forall}$ of $\varphi$.
Also this step involves introducing non-trivial tuples.
%
%
%
%

%
%

\subsection{Constructing an ordered and
labelled domain for \texorpdfstring{$\mathfrak{A}$}{modA} }\label{ordersection}
%

%
%
%
Before defining an ordered domain $(A,<)$ for $\mathfrak{A}$,
we construct an ordered set $(I_s,<)$ for each $s\in\{1,...\, , N\}$
based on the set $\bbalpha_{\tau,s}
\in (\bbalpha_{\tau,s})_{1\leq s \leq N}$ of $\Gamma_{\varphi}$.
Once we have the ordered sets defined, the 
ordered domain $(A,<)$ is defined to be the ordered sum
$(A,<) = \Sigma_{1 \leq s \leq N} (I_s,<)$, i.e., the
ordered sets $(I_s,<)$ are simply concatenated so that
the elements of $I_t$ are before the elements of $I_{t'}$ iff $t < t'$.
Thus the ordered sets $(I_s,<)$ become intervals in $(A,<)$.
However, we will not only construct an ordered domain $(A,<)$ in
the current subsection (Subsection \ref{ordersection}),
we will also label the elements of $A$ by $1$-types over $\tau$.
Thus, by the end of the current subsection, the structure $\mathfrak{A}$
will be a linearly ordered structure with the $1$-types over $\tau$ defined.
Each interval $I_s$ will be labelled such that
exactly all the $1$-types in the set $\bbalpha_{\tau,s}$
given in $\Gamma_{\varphi}$
are satisfied by the elements of $I_s$.
Let $s \in \{1,...\, , N\}$.
We now make use of 
the admissibility tuple $\Gamma_{\varphi}$ as follows.
If $\bbalpha_{\tau,s} \cap \bbalpha_{\tau}^{K} \neq \emptyset$,
then by the admissibility condition \ref{AdmissibilityCourtiers} from Definition \ref{admidefn},
we have $s = \delta(c)$ for some $c \in C$ where $C$ is the
domain of the structure $\mathfrak{C}$ from $\Gamma_{\varphi}$.
%
%
Furthermore, we infer, using the admissibility
condition \ref{AdmissibilityCourtiers}, that $\bbalpha_{\tau,\delta(c)}$ must in fact be a
singleton $\{\alpha_s\}$ such $\alpha_s = tp_{\mathfrak{C}}(c)$.
%
%
%
%
We define $I_s$ to be a singleton set, and we
label the unique element $u$ in $I_s$ by the
type $\alpha_s$ by 
defining $\mathit{lab}(u) = \alpha_s$ where $\mathit{lab}$ denotes a labelling
function $\mathit{lab}:A\rightarrow\bbalpha_{\tau}$
whose definition will become fully fixed once we have 
dealt with all the intervals $(I_s,<)$.
Having discussed the case
where $\bbalpha_{\tau,s}\cap\bbalpha_{\tau}^K\not=\emptyset$, we
assume that $\bbalpha_{\tau,s} \cap \bbalpha_{\tau}^{K} = \emptyset$.
We divide the analysis of this case into three subcases (see below)
depending on the
degree of admissibility of $\Gamma_{\varphi}$ (cf. Definition \ref{admidefn2}).
%
%
Before dealing with the cases,
we define some auxiliary ordered sets that will function as building blocks
when we construct the intervals $(I_s,<)$.
Fix $n$ to be the width of $\varphi$ and $m_{\exists}$ the
number of existential conjuncts in $\varphi$.
By a \textit{$3(m_{\exists} + n)$-block} we mean a finite ordered set
that consists of $3(m_{\exists} + n)$ elements.
A $3(m_{\exists} + n)$-block divides into 
into three disjoint sets that we call the \emph{$E$-part, $F$-part and $G$-part}.
Each of the parts contains $m_{\exists} + n$ consecutive elements in the block
such that the sets $E$, $F$ and $G$ 
appear in the given order.
We will define the remaining intervals $(I_s,<)$ below
using $3(m_{\exists} + n)$-blocks. For each $3(m_{\exists} + n)$-block $(U,<)$ we use,
the elements in $U$ will be labelled with a single $1$-type, i.e.,
we will define $\mathit{lab}(u) = \mathit{lab}(u')$ for all $u,u'\in U$.
Therefore we in fact (somewhat informally) talk about
about $3(m_{\exists} + n)$-blocks $(U,<)$ \textit{of $1$-type $\alpha$}.
This means that while $(U,<)$ is strictly speaking only an (unlabelled) ordered set
with $3(m_{\exists} + n)$ elements, we will ultimately set $\mathit{lab}(u) = \alpha$
for all $u\in U$.
Let $(J,<)$ be a finite, ordered set consisting of
several $3(m_{\exists} + n)$-blocks such that there is one $3(m_{\exists} + n)$-block
for each $1$-type $\alpha \in \bbalpha_{\tau,s}$ and no
other blocks; the order in
which the blocks $(U,<)$ for different $1$-types appear in $(J,<)$ is
chosen arbitrarily. 
Similarly, let $(J^{-},<)$
contain a $3(m_{\exists} + n)$-block
for each $\alpha \in \bbalpha_{\tau,s}^{-}$ in
some order and no other blocks.
Let $(J^{+},<)$ contain a block for each $\alpha \in \bbalpha_{\tau,s}^{+}$ in
some order and no other blocks. Note that $J^-$ and $J^+$ may be empty.
%
%
%
%
%
%
%
%
%
We define the ordered interval $(I_s,<)$ as follows:
\begin{enumerate}
\item
Assume $\Gamma_{\varphi}$ is admissible for $\mathcal{O}$ but not for $\mathcal{WO}$.
%
%
%
We define $(I_s,<)$ to be the
ordered set consisting of three parts $(I_s,<)_1$, $(I_s,<)_2$ and $(I_s,<)_3$ in
the given order and defined as follows.
\begin{enumerate}
\item
$(I_s,<)_1$ consists of a countably infinite number of 
copies of $(J^-,<)$ such that the different copies are
ordered as the negative integers, i.e., $(I_s,<)_1$
can be obtained by ordering $\mathbb{Z}_{\mathit{neg}}\times J^-$
lexicographically, where $\mathbb{Z}_{\mathit{neg}}$
denotes the negative integers; schematically,
%
%
%
$(I_s,<)_1 :=\ ...\cdot (J^-,<)\cdot (J^-,<)\cdot (J^-,<)$
%
%
%
%
where $``\cdot"$ denotes concatenation.
\item
$(I_s,<)_2 := (J,<)$.
\item
$(I_s,<)_3$ consists of a countably infinite number of
copies of $(J^{+},<)$ such that the different copies
are ordered as the positive integers.
\end{enumerate}
Schematically, $(I_s,<)$ is therefore the structure

\smallskip

$...\cdot (J^-,<)\cdot (J^-,<)\cdot (J,<)
\cdot (J^+,<)\cdot (J^+,<)\cdot ...$

\smallskip

\item
Assume $\Gamma_{\varphi}$ is
admissible for $\mathcal{WO}$ but not for $\mathcal{O}_{\mathit{fin}}$.
%
%
%
Again the interval $(I_s,<)$ is the concatenation of three
parts ${(I_s,<)_1}$, ${(I_s,<)_2}$, ${(I_s,<)_3}$ in that order, but
while $(I_s,<)_2$ and $(I_s,<)_3$ are the same as above,
now $(I_s,<)_1 := (J^-,<)$. Thus, $(I_s,<)$ is 
the structure

\smallskip

$(J^-,<)\cdot (J,<)
\cdot (J^+,<)\cdot (J^+,<)\cdot ...$ 

\smallskip

%
%
%
%
\item
Assume $\Gamma_{\varphi}$ is admissible for $\mathcal{O}_{fin}$.
%
%
%
In this case we define $(I_s,<)$ to be the structure%

\smallskip

$(J^{-},<)\cdot (J,<)\cdot (J^{+},<)$.

\smallskip

%
%
%
%
\end{enumerate}
%
%
%
%
%
%
%
%
%
%
Note that since  we already associated each $3(m_{\exists} + n)$-block in
each of the structures ${(J,<), (J^-,<),(J^+,<)}$ with a labelling 
with $1$-types, we have now also defined the $1$-types over
the interval $(I_s,<)$.
Therefore we have now shown how to construct an ordered domain $(A,<)$
for $\mathfrak{A}$ and also defined a labelling of $A$ with $1$-types.
\subsection{Copying \texorpdfstring{$\mathfrak{C}$}{modC} into \texorpdfstring{$\mathfrak{A}$}{modA} }
\label{copyingCsubsection}
%
%
Due to axiom \ref{POAcourt}, the structure $\mathfrak{B}$
contains an isomorphic copy $\mathfrak{C}_{\mathfrak{B}}$ of
the structure $\mathfrak{C}$ from $\Gamma_{\varphi}$, that is, $\mathfrak{B}$
has a substructure $\mathfrak{C}_{\mathfrak{B}}'$
such that $\mathfrak{C}_{\mathfrak{B}}'\upharpoonright\tau$ is
isomorphic to $\mathfrak{C}$
and $\mathfrak{C}_{\mathfrak{B}} :=
\mathfrak{C}_{\mathfrak{B}}'\upharpoonright\tau$.
The domain $C_{\mathfrak{B}}$ of $\mathfrak{C}_{\mathfrak{B}}$
is the union of the sets $U_{\delta(c)}^{\mathfrak{B}}$ for all $c \in C$;
recall that by axiom \ref{POAsingleton}, each $U_{\delta(c)}^{\mathfrak{B}}$, for $c \in C$, is a singleton.
Let $g$ be the isomorphism
from $\mathfrak{C}_\mathfrak{B}$ to $\mathfrak{C}$.
(The isomorphism is unique since $\mathfrak{C}$ is an
ordered set.)
%
%
%
%
%
%
We shall create an isomorphic
copy of $\mathfrak{C}$ into $\mathfrak{A}$ by 
introducing tuples to the relations of $\mathfrak{A}$; no 
new points will be added to $A$.
We first define an injective mapping $h$ from
$C_{\mathfrak{B}}$ to  $A$ as follows.
Let $b\in C_{\mathfrak{B}}$, and denote $\delta(g(b))$ by $s$.
Now, if $b$ realizes a
$1$-type $\alpha \in \bbalpha_{\tau,s}\cap\bbalpha^K$,
then we recall from
Section \ref{ordersection} that $I_s\subseteq A$ is a singleton interval that realizes
the type $\alpha$. We let $h$ map $b$ to the element in $I_{s} \subseteq A$.
Otherwise $b$ realizes a $1$-type $\alpha \in \bbalpha_{\tau,s}$
such that $\alpha\not\in\bbalpha^K$.
Then, by admissibility condition \ref{AdmissibilityCourtiers}, (see Definition \ref{admidefn}), $\bbalpha_{\tau,s}^{-}$
and $\bbalpha_{\tau,s}^{+}$ are empty. Therefore, using the
notation from Section \ref{ordersection}, we
have $(I_{s},<) = (J,<)$ as $J^{-}$ and $J^{+}$ are empty.
Therefore, and since $\bbalpha_{\tau,s}$ is a singleton (by admissibility condition
\ref{AdmissibilityCourtiers}), we
observe that $(I_s,<)$ consists of a single $3(m_{\exists}+n)$-block of
elements realizing $\alpha$.
We let $h$ map $b$ to the first element in $I_{s}\subseteq A$.
Denote the set $img(h)$ by $C_{\mathfrak{A}}$.
Hence $h$ is a bijection from $C_{\mathfrak{B}}$ onto $C_{\mathfrak{A}}$
that preserves $1$-types over $\tau$.
Due to the construction of the order $<^{\mathfrak{A}}$ and 
axiom \ref{POAquasiorder}, it is easy to see that $h$ also
preserves order,  i.e.,
we have $b < b' \text{ iff } h(b) < h(b')$ for all $b,b' \in C_{\mathfrak{B}}$.
Now let $r'$ denote the highest arity of the relation symbols in $\varphi$.
Let $\{b_1,\ldots,b_j\}\subseteq
C_{\mathfrak{B}}$ be a set with $j\in \{2,...\, , r'\}$ elements.
We define
%
$$tb_{\mathfrak{A}}(h(b_1),...\, , h(b_j))
:= tb_{\mathfrak{B}\upharpoonright\tau}(b_1,...\, , b_j)$$
%
%
%
and repeat this for each subset of $C_{\mathfrak{A}}$ of size from $2$ up to $r'$.
By construction, $h$ is an isomorphism
from $\mathfrak{C}_{\mathfrak{B}}
\upharpoonright\tau$ to $\mathfrak{C}_\mathfrak{A}$.
%
%

%
%
%
%
%
%
%
%
%
\subsection{Finding witness structures}\label{witnessstructuresection}
Recalling the function $h$ from the previous section, we define
$K_{\mathfrak{A}} := \{h(k)\mid k \in K^{\mathfrak{B}}\}$ and 
$D_{\mathfrak{A}} := \{h(d)\mid d \in D^{\mathfrak{B}}\}$.
By axiom \ref{POAfreecourt}
and due to the definition of the domain $C_{\mathfrak{B}}$ (cf. Subsection \ref{copyingCsubsection}),
we have $(K^{\mathfrak{B}}\cup D^{\mathfrak{B}}) \subseteq C_{\mathfrak{B}}\subseteq B$.
Moreover, by axiom \ref{POAkingswitnesses} and how $C_{\mathfrak{B}}$ was defined, 
there is a witness structure in $\mathfrak{C}_{\mathfrak{B}}$ for
every $b \in (K^{\mathfrak{B}}\cup D^{\mathfrak{B}})
\subseteq C_{\mathfrak{B}}$ and every existential conjunct
$\varphi_{i}^{\exists}$ of $\varphi$.
As $\mathfrak{C}_{\mathfrak{A}}$ is isomorphic to
$\mathfrak{C}_{\mathfrak{B}}$,
there is a witness structure in $\mathfrak{C}_{\mathfrak{A}}$ for
every $a \in (K_{\mathfrak{A}}\cup D_{\mathfrak{A}}) \subseteq C_{\mathfrak{A}}$
and every conjunct $\varphi_{i}^{\exists}$ of $\varphi$.
In this section we show how to define, for each element 
$a \in A \setminus (K_{\mathfrak{A}} \cup D_{\mathfrak{A}})$ 
and each existential conjunct $\varphi_{i}^{\exists}$ of $\varphi$, a
witness structure in $\mathfrak{A}$.
This consists of the following steps, to be described in
detail later on.
\begin{enumerate}
\item
We first choose, for each $a\in A\setminus (K_{\mathfrak{A}}\cup D_{\mathfrak{A}})$, a
\textit{pattern element} $b_a$
of the same $1$-type (over $\tau$) from $\mathfrak{B}$.
\item
We then locate, for each pattern element $b_a$ and each
existential conjunct $\varphi_i^{\exists}$, a 
witness structure $\mathfrak{B}_{b_a,\varphi_i^{\exists}}$ in $\mathfrak{B}$.
\item
We then find, for each element $b'$ of the
live part $\bar{\mathfrak{B}}_{b_a,\varphi_i^{\exists}}$
of $\mathfrak{B}_{b_a,\varphi_i^{\exists}}$, a
corresponding $3(m_{\exists}+n)$-block of elements
from $\mathfrak{A}$. The elements of the block satisfy the same $1$-type as $b'$.
We denote the block by $\mathit{bl}(b')$.
\item
After this, we locate from each block $\mathit{bl}(b')$ an element
corresponding to $b'$. We then construct from these elements a 
live part $\bar{\mathfrak{A}}_{a,\varphi_i^{\exists}}$
of a witness structure for $a$ and $\varphi_i^{\exists}$.
\item
These live parts are then, at the very end of 
our procedure, completed to full witness structures by
locating suitable dead parts from $\mathfrak{A}$.
\end{enumerate}
Let $a \in A \setminus (K_{\mathfrak{A}} \cup D_{\mathfrak{A}})$
and let $s_a \in \{1,...\, , N\}$
denote the index of the interval $I_{s_a}$ such that $a\in I_{s_a}$.
Let $\alpha \in \bbalpha_{\tau,s_a} \setminus \bbalpha^{K}$
be the $1$-type of $a$ over $\tau$.
We next show how to select a pattern element $b_a$ for $a$.
The pattern element $b_a$ will be selected
from the set $U_{s_a}^{\mathfrak{B}}\subseteq B$.
%
%

%
\begin{enumerate}
\item
Firstly, if $a \in C_{\mathfrak{A}}$,
then we let $b_a := h^{-1}(a) \in C_{\mathfrak{B}}$,
where $h$ is the bijection from $C_{\mathfrak{B}}$ to $C_{\mathfrak{A}}$.
Otherwise we consider the following cases 2-4.
\item
Assume $\Gamma_{\varphi}$ is admissible for $\mathcal{O}$ but
not for $\mathcal{WO}$ (and thus not for $\mathcal{O}_{fin}$ either).
Then we let $b_a$ be an arbitrary
realization of $\alpha$ in $U_{s_a}^{\mathfrak{B}}$.
\item
Assume that $\Gamma_{\varphi}$ is admissible for $\mathcal{WO}$ but
not for $\mathcal{O}_{fin}$.
Then, if $\alpha\not\in\bbalpha_{\tau,s_a}^{-}$, we let $b_a$ 
be an arbitrary realization of $\alpha$ in $U_{s_a}^{\mathfrak{B}}$.
If $\alpha \in \bbalpha_{\tau,s_a}^{-}$,
we let $b_a$ be the element in $U_{s_a}^{\mathfrak{B}}$ that satisfies
$min_{\alpha}(x)$;
this is possible due to admissibility condition
\ref{AdmissibilityBot} and axiom \ref{POAmin}.
\item
Assume $\Gamma_{\varphi}$ is admissible for $\mathcal{O}_{fin}$.
Now, if we have {$\alpha \not\in \bbalpha_{\tau,s_a}^{-}\cup\,
\bbalpha_{\tau,s_a}^{+}$},
we let $b_a$ be an arbitrary realization of $\alpha$ in $U_{s_a}^{\mathfrak{B}}$.
If $\alpha \in \bbalpha_{\tau,s_a}^{-}\setminus\bbalpha_{\tau,s_a}^{+}$, 
then we let $b_a$ be the element in $U_{s_a}^{\mathfrak{B}}$ that
satisfies $min_{\alpha}(x)$,
which is possible due to the
admissibility condition \ref{AdmissibilityBot} and
axiom {\ref{POAmin}}. If $\alpha \in
\bbalpha_{\tau,s_a}^{+}\setminus\bbalpha_{\tau,s_a}^{-}$,
then we let $b_a$ be the element in $U_{s_a}^{\mathfrak{B}}$ that
satisfies $max_{\alpha}(x)$, which is possible due to
the admissibility condition \ref{AdmissibilityTop}
and axiom {\ref{POAmax}}.
Finally, if $\alpha \in \bbalpha_{\tau,s_a}^{-}\cap \bbalpha_{\tau,s_a}^{+}$, then there are the following two cases:
If $a$ is not in the last $3(m_{\exists}+n)$-block in $I_{s_a}$, then we
choose $b_a$ as in the case
$\alpha \in \bbalpha_{\tau,s_a}^{-}\setminus\bbalpha_{\tau,s_a}^{+}$.
If $a$ is in the last $3(m_{\exists}+n)$-block in $I_{s_a}$, then we choose $b_a$ as in the case
$\alpha \in \bbalpha_{\tau,s_a}^{+}\setminus\bbalpha_{\tau,s_a}^{-}$.
\end{enumerate}

We have now a pattern element $b_a$ for
each $a$ in $A\setminus (K_{\mathfrak{A}}\cup D_{\mathfrak{A}})$.
%
%
%
%
%
%
%
%
Let $a$ denote an
arbitrary  element in $A\setminus (K_{\mathfrak{A}}\cup D_{\mathfrak{A}})$ and
let $\varphi_{i}^{\exists}$ be an arbitrary existential conjunct of $\varphi$.
By axiom \ref{POAchi}, we have $\mathfrak{B} \models \varphi$,  
and thus we find a witness structure $\mathfrak{B}_{b_a,\varphi_{i}^{\exists}}$
in $\mathfrak{B}$ for the pair $(b_a,\varphi_{i}^{\exists})$.
Next we consider a number of cases based on what the
live part $\bar{\mathfrak{B}}_{b_a,\varphi_{i}^{\exists}}$ of the witness structure
$\mathfrak{B}_{b_a,\varphi_{i}^{\exists}}$ is like and how the
live part is oriented in relation to $\mathfrak{B}_{b_a,\varphi_{i}^{\exists}}$. 
In each case, we ultimately define a live part $\bar{\mathfrak{A}}_{a,\varphi_{i}^{\exists}}$
for some witness structure $\mathfrak{A}_{a,\varphi_{i}^{\exists}}$. The
dead part of the witness structure $\mathfrak{A}_{a,\varphi_{i}^{\exists}}$
will be found at a later stage of our construction. In many of the cases, the
identification of the live part $\bar{\mathfrak{A}}_{a,\varphi_{i}^{\exists}}$
requires that we first identify suitable $3(m_{\exists}+n)$-blocks $\mathit{bl}(b')$
for the elements $b'$ of $\bar{{\mathfrak{B}}}_{b_a,\varphi_{i}^{\exists}}$,
and only after finding the blocks, we identify suitable elements 
from the blocks in order to construct $\bar{\mathfrak{A}}_{a,\varphi_{i}^{\exists}}$.
%


\textbf{Case `\textit{empty live part}'} : If the
live part $\bar{\mathfrak{B}}_{b_a, \varphi_{i}^{\exists}}$
of the witness structure $\mathfrak{B}_{b_a, \varphi_{i}^{\exists}}$ is
empty, we let the live part $\bar{\mathfrak{A}}_{a, \varphi_{i}^{\exists}}$
of a witness structure for $(a,\varphi_i^{\exists})$, whose
dead part will be constructed later, be empty.
%
%
%

%
\textbf{Case `\textit{free live part}'} : Assume that $b_a$ does not belong to
the (non-empty) live part $\bar{\mathfrak{B}}_{b_a, \varphi_{i}^{\exists}}$
of the witness structure $\mathfrak{B}_{b_a, \varphi_{i}^{\exists}}$.
By axiom \ref{POArunningfreeyeah},
there is a witness structure for $(\alpha,\varphi_{i}^{\exists})$
in $\mathfrak{B}$
whose live part is in the set $D^{\mathfrak{B}}\subseteq C_{\mathfrak{B}}\subseteq B$.
Let $d_1,...\, ,d_k\in D^{\mathfrak{B}}$ be the
elements of $\bar{\mathfrak{B}}_{b_a, \varphi_{i}^{\exists}}$
(so $\bar{\mathfrak{B}}_{b_a, \varphi_{i}^{\exists}}$
contains exactly $k\geq 1$ elements).
%
%
According to axiom \ref{POArunningfreeyeah},
as $C_{\mathfrak{B}}$ and $C_{\mathfrak{A}}$ are isomorphic (via the bijection $h$),
it is clear that $tb_{\mathfrak{A}}(h(d_1),...\, ,h(d_k))
= tb_{\mathfrak{B}\upharpoonright\tau}(d_1,...\, , d_k)$.
Therefore we let $\{h(d_1),...\, , h(d_k)\}$
be the domain of the live part $\bar{\mathfrak{A}}_{a, \varphi_{i}^{\exists}}$
of a witness structure for $(a,\varphi_i^{\exists})$, whose
dead part will be constructed later;
we note that $a\not\in D_{\mathfrak{A}}$ due to our
assumption that $a\not\in K_{\mathfrak{A}}\cup D_{\mathfrak{A}},$
so $\bar{\mathfrak{A}}_{a, \varphi_{i}^{\exists}}$ is free w.r.t. $a$,
i.e., $a\not\in \bar{A}_{a, \varphi_{i}^{\exists}}$. 
%
%
%

\textbf{Case `\textit{local singleton live part}'} : Assume that $b_a$ is alone in the live part
$\bar{\mathfrak{B}}_{b_a, \varphi_{i}^{\exists}}$ of
the witness structure $\mathfrak{B}_{b_a,
\varphi_{i}^{\exists}}$, i.e., $|\bar{B}_{b_a,\varphi_{i}^{\exists}}|=1$.
We recall that $tb_{\mathfrak{A}}(a) = tb_{\mathfrak{B}\upharpoonright\tau}(b_a)$,
and we let $\{a\}$ be the
domain of the live part $\bar{\mathfrak{A}}_{a, \varphi_{i}^{\exists}}$
of a witness structure for $(a,\varphi_i^{\exists})$,
whose dead elements will be identified later.
\textbf{Case `\textit{local doubleton live part}'} : Assume that $b_a$
and some other element $b'\not=b_a$ in $B$ form
the live part $\bar{\mathfrak{B}}_{b_a, \varphi_{i}^{\exists}}$
of the witness structure $\mathfrak{B}_{b_a, \varphi_{i}^{\exists}}$.
Thus $|\bar{B}_{b_a,\varphi_{i}^{\exists}}|=2$.
Let $t_{b'}\in\{1,...\, , N\}$ be
the index such that $b'\in U_{t_{b'}}^{\mathfrak{B}}\subseteq B$.
Next we consider several subcases of the case \textit{local doubleton live part}.
%
%
%
%
%
%
%
%
%
%

%
In the following subcases 1 and 2, we assume that
$t_{b'} \neq s_{a}$; recall that $b_a\in U_{s_a}^{\mathfrak{B}}$
and $b'\in U_{t_{b'}}^{\mathfrak{B}}$.
%
%
We first note that if $t_{b'} < s_{a}$ (respectively, if $s_{a} < t_{b'}$),
then by axiom \ref{POAquasiorder}, we
have $\mathfrak{B} \models b' < b_a$ (resp., $\mathfrak{B} \models b_a < b'$).
%
%
%
%
\begin{itemize}
\item[1.]
If $b'\in C_{\mathfrak{B}}$,
%
%
%
%
then we define $tb_{\mathfrak{A}}(a,h(b'))
:= tb_{\mathfrak{B}\upharpoonright\tau}(b_a,b')$.
We note that in the special case where $a\in C_{\mathfrak{A}}$,
as we have $b'\in C_{\mathfrak{B}}$,
both elements $a$ and $h(b')$ are in $C_{\mathfrak{A}}$,
and therefore we have actually already
defined the table $tb_{\mathfrak{A}}(a,h(b'))$
when $\mathfrak{C}_{\mathfrak{B}}$ was
copied into $\mathfrak{A}$.
%
%
%
%
%
%
%
%
\item[2.]
If $b'\not\in C_{\mathfrak{B}}$,
then we select some $3(m_{\exists}+n)$-block $bl(b')$ of elements in
$I_{t_{b'}} \subseteq A$
realizing the 1-type $tp_{\mathfrak{B}\upharpoonright\tau}(b')$;
this is possible as for all $s\in\{1,...\, , N\}$,
the interval $I_s\subseteq A$ has been
constructed so that it
realizes exactly the same $1$-types over $\tau$ as the set $U_s^{\mathfrak{B}}$,
and furthermore, for the following reason: Since $b' \not\in C_{\mathfrak{B}}$, 
we have $b' \not \in K^{\mathfrak{B}}$, and
thus (by axiom \ref{POAkingscourt}) we have 
$tp_{\mathfrak{B}\upharpoonright{\tau}}(b') \not \in \bbalpha_{\tau}^{K}$,
whence it follows from the construction of the domain $A$ that the interval $I_{t_{b'}}$
contains at least one $3(m_{\exists}+n)$-block of each $1$-type realized in the interval.
With the block $\mathit{bl}(b')$ chosen, we
will later on show how to choose an element $a'\in bl(b')\subseteq A$ in
order to construct a full live part of a witness structure for $(a,\varphi_i^{\exists})$.
After that we will identify related dead elements in order to ultimately complete
the live part into a full witness structure. (Strictly speaking, rather
than seeking full definitions of witness structures, we
will always define only a table for the live part of a 
witness structure in addition to making sure that
suitable elements for the dead part can be found.)
%
%
%
%
\end{itemize}
In the following subcases 3 and 4 of the case \textit{local doubleton free-part},
we assume that $t_{b'} = s_a$, i.e., $b_a,b'\in U_{s_a}^{\mathfrak{B}}$.
It follows from axiom \ref{POAtournament}
that either $\mathfrak{B}\models b_a < b'$ or $\mathfrak{B}\models b' < b_a$
but not both.
In both subcases 3 and 4, we
locate only a $3(m_{\exists}+n)$-block $bl(b')\subseteq A$ 
of elements of $1$-type $tp_{\mathfrak{B}\upharpoonright{\tau}}(b')$.
%
%
%
%
Once again
we will only later find elements from the block $bl(b')$ in
order to identify a live part of a witness structure for $(a,\varphi_i^{\exists})$,
and after that we ultimately complete the live part to a full witness structure by
finding suitable dead elements.
Note that since $b_a$ and $b'\not= b_a$ are both
in $U_{s_a}^{\mathfrak{B}}$, the set $U_{s_a}^{\mathfrak{B}}$ is 
not a singleton and thus $U_{s_a}^{\mathfrak{B}}\cap C_{\mathfrak{B}}=\emptyset$.
Therefore, $b' \not \in K^{\mathfrak{B}}$ and by axiom \ref{POAkingscourt}, 
$tp_{\mathfrak{B}\upharpoonright{\tau}}(b') \not \in \bbalpha_{\tau}^{K}$.
Now it follows from the construction of the domain $A$ that
the interval $I_{s_a}$ contains at least one $3(m_{\exists}+n)$-block of 
1-type $tp_{\mathfrak{B}\upharpoonright{\tau}}(b')$.
\begin{itemize}
\item[3.] Assume that  $\mathfrak{B} \models b' < b_a$.
Let $\alpha'$ denote the $1$-type $tp_{\mathfrak{B}\upharpoonright\tau}(b')$ of $b'$.
If $\alpha' \not \in \bbalpha_{\tau,s_a}^{-}$, then we must have
$\alpha'\in\bbalpha_{\tau,t}$ for some $t < s_a$.
Thus, and as $tp_{\mathfrak{B}\upharpoonright{\tau}}(b') \not \in \bbalpha_{\tau}^{K}$, 
$I_t\subseteq A$ contains at least one block $bl(b')$ of elements
realizing the $1$-type $\alpha'$.
We choose the block $bl(b')$ to be the desired block to be used later.
If, on the other hand, we
have $\alpha' \in \bbalpha_{\tau,s_a}^{-}$, we proceed as follows.
%
%
\begin{enumerate}
\item[a)]
Assume $\Gamma_{\varphi}$ is admissible only for $\mathcal{O}$
and not for $\mathcal{WO}$ (and thus not for $\mathcal{O}_{fin}$ either).
Then, due to the way we have defined 
the interval $I_{s_a}\subseteq A$ and labelled its elements by $1$-types,
there exists a $3(m_{\exists}+n)$-block $bl(b')\subseteq I_{s_a}$ of
elements of type $\alpha'$
such that $bl(b')$ precedes the block in $I_{s_a}$ that contains $a$.
We appoint $bl(b')$ to be the desired block to be used later.
\item[b)]
Assume $\Gamma_{\varphi}$ is 
admissible for $\mathcal{WO}$ and not for $\mathcal{O}_{fin}$.
Assume first that $\alpha \not\in \bbalpha_{\tau,s_a}^{-}$
(where we recall that $\alpha$ is the $1$-type of $a$ and $b_a$ over $\tau$).
Since $\alpha'\in\bbalpha_{\tau,s_a}^-$ and 
$\alpha\not\in\bbalpha_{\tau,s_a}^-$, we
observe that the interval $I_{s_a}\subseteq A$ has been defined such that
there exists a $3(m_{\exists}+n)$-block $bl(b')\subseteq I_{s_a}$ of
elements of type $\alpha'$
such that $bl(b')$ precedes the block in $I_{s_a}$ that contains $a$.
We appoint $bl(b')$ to be the block to be used later.
%
%
%
%
%
%

Assume then that $\alpha \in \bbalpha_{\tau,s_a}^{-}$.
In this case, we have
chosen the pattern element $b_a$ to be 
the minimal realization of $\alpha$ in $\mathfrak{B}$.
Since $\mathfrak{B} \models b' < b_a$,
we must have $tp_{\mathfrak{B}\upharpoonright\tau}(b')\neq tp_{\mathfrak{B}
\upharpoonright\tau}(b_a)$.
Thus we must have $\alpha' = tp_{\mathfrak{B}\upharpoonright\tau}(b')
\not \in \bbalpha_{\tau,s_a}^{-}$
by the admissibility condition \ref{AdmissibilityMinusalphaone} (which states that $|\bbalpha_{\tau,s_a}^{-} | \leq 1$).
This contradicts the assumption that $\alpha'\in \bbalpha_{\tau,s_a}^{-}$, so
this case is in fact impossible and can thus be ignored.
%
%
%
\item[c.1)]
Assume $\Gamma_{\varphi}$ is admissible for $\mathcal{O}_{fin}$.
Furthermore, assume that one of the following conditions holds.
\begin{itemize}
\item[c.1.1)]$\alpha \not \in \bbalpha_{\tau,s_a}^{-}$ (but $\alpha$ may be in $\bbalpha_{\tau,s_a}^{+}$).
\item[c.1.2)]$\alpha \in \bbalpha_{\tau,s_a}^{-} \cap \bbalpha_{\tau,s_a}^{+}$
and $a$ is in the last block in $I_{s_a}$.
\end{itemize}
Now, since $\alpha'\in\bbalpha_{\tau,s_a}^-$
%
%
we observe that the interval $I_{s_a}\subseteq A$ has been defined such that
there is a $3(m_{\exists}+n)$-block $bl(b')\subseteq I_{s_a}$ of
elements of type $\alpha'$
such that $bl(b')$ precedes the block in $I_{s_a}$ that contains $a$.
We appoint $bl(b')$ to be the block to be used later.
%
%
%
%
\item[c.2)]
Now assume $\Gamma_{\varphi}$ is admissible for $\mathcal{O}_{fin}$, and
furthermore, assume that one of the following conditions holds.
\begin{itemize}
\item[c.2.1)] $\alpha \in \bbalpha_{\tau,s_a}^{-} \setminus \bbalpha_{\tau,s_a}^{+}$.
\item[c.2.2)] $\alpha \in \bbalpha_{\tau,s_a}^{-} \cap \bbalpha_{\tau,s_a}^{+}$
and $a$ is not in the
last block in $I_{s_a}$.
\end{itemize}
In these cases we have
chosen the pattern element $b_a$ to be 
the minimal realization of $\alpha$ in $\mathfrak{B}$.
Since $\mathfrak{B} \models b' < b_a$,
we must have $tp_{\mathfrak{B}\upharpoonright\tau}(b')\neq tp_{\mathfrak{B}
\upharpoonright\tau}(b_a)$.
Thus we must have $\alpha' = tp_{\mathfrak{B}\upharpoonright\tau}(b')
\not \in \bbalpha_{\tau,s_a}^{-}$
by the admissibility condition \ref{AdmissibilityMinusalphaone} (which states that $|\bbalpha_{\tau,s_a}^{-} | \leq 1$).
This contradicts the assumption that $\alpha'\in \bbalpha_{\tau,s_a}^{-}$, so
this case is in fact impossible and can thus be ignored.
\end{enumerate}
\item[4.]
Assume that $\mathfrak{B} \models b_a < b'$.
%
%
Again we let $\alpha'$
denote $tp_{\mathfrak{B}\upharpoonright\tau}(b')$.
If $\alpha' \not \in \bbalpha_{\tau,s_a}^{+}$, then
we have $\alpha'\in\bbalpha_{\tau,t}$ for some $t > s_a$.
We choose $bl(b')$ to be some block of elements
realizing the 1-type $\alpha'$ from the interval $I_{t}\subseteq A$.
If $\alpha' \in\bbalpha_{\tau,s_a}^{+}$, we
proceed as follows.
\begin{itemize}
\item[a)]
Assume that $\Gamma_{\varphi}$ is \textit{not} admissible for $\mathcal{O}_{fin}$
but \emph{is} admissible for $\mathcal{O}$
or even for $\mathcal{WO}$.
%
%
%
%
%
Then, due to the way we defined $1$-types
over the interval $I_{s_a}$, there
exists a block $bl(b')\subseteq I_{s_a}$ of type $\alpha'$
following the block that contains $a$ in $I_{s_a}$.
We appoint the block $bl(b')$ to be used later.
\item[b.1)]
Assume that $\Gamma_{\varphi}$ is admissible for $\mathcal{O}_{fin}$.
Furthermore, recall that $\alpha$ is the $1$-type of $a$ and assume that
%
%
one of the following conditions holds.
\begin{itemize}
\item[b.1.1)] $\alpha \not\in \bbalpha_{\tau,s_a}^-\cup\bbalpha_{\tau,s_a}^{+}$ 
\item[b.1.2)] $\alpha \in \bbalpha_{\tau,s_a}^{-}\setminus \bbalpha_{\tau,s_a}^{+}$
\item[b.1.3)] $\alpha \in \bbalpha_{\tau,s_a}^{-} \cap
\bbalpha_{\tau,s_a}^{+}$ and $a$ is not in the last block in $I_{s_a}$.
\end{itemize}
%
%
%
%
%
%
%
%
Now, since $\alpha'\in\bbalpha_{\tau,s_a}^+$ and due to 
admissibility condition \ref{AdmissibilityPlusalphaone} and the way we defined $1$-types
over the interval $I_{s_a}$, the last block in $I_{s_a}$ is of $1$-type $\alpha'$.
Clearly this last block comes after the block that contains $a$ in $I_{s_a}$.
We call this last block $bl(b')$ and appoint it for later use.
%
%
%
%
%
%
\item[b.2)]
Assume $\Gamma_{\varphi}$ is admissible for $\mathcal{O}_{fin}$ 
and that one of the following cases holds.
\begin{itemize}
\item[b.2.1)] $\alpha \in \bbalpha_{\tau,s_a}^{+}\setminus\bbalpha_{\tau,s_a}^-$
\item[b.2.2)] $\alpha \in \bbalpha_{\tau,s_a}^{-} \cap \bbalpha_{\tau,s_a}^{+}$
and $a$ is in the last block in $I_{s_a}$.
\end{itemize}
Then we have chosen the pattern element $b_a$ to be
the maximal realization of $\alpha$ in $B$, i.e., it satisfies $\mathit{max}_{\alpha}(x)$.
As admissibility for $\mathcal{O}_{fin}$
implies that $|\bbalpha_{\tau,s_a}^+| \leq 1$, we have $\alpha = \alpha'$.
As we have assumed that $\mathfrak{B}\models b_a<b'$,
we observe that this case is in fact impossible and can
thus be ignored. 
%
%
%
%
\end{itemize}
\end{itemize}
Now recall that when constructing the domain $A$ of $\mathfrak{A}$
using $3(m_{\exists} + n)$-blocks, we
defined the \emph{$E$-part} of a $3(m_{\exists} + n)$-block to be
the set that contains the first $(m_{\exists} + n)$ elements of the block. Similarly,
we defined the \emph{$F$-part} to be the set with the
subsequent $(m_{\exists}+n)$ elements immediately
after the $E$-part, and the $G$-part was defined to be the set
with the last $(m_{\exists}+n)$ elements.
Below, we let $E\subseteq A$ denote the union of the $E$-parts of
all the $3(m_{\exists}+n)$-blocks used in the construction of $A$.
Similarly, we let $F$ and $G$ denote the
unions of the $F$-parts and $G$-parts, respectively.
Now, in the subcases 2-4 of the case \textit{doubleton live part}, we
located a $3(m_{\exists}+n)$-block $bl(b')\subseteq A$ of
elements of type $\alpha' = tp_{\mathfrak{B}\upharpoonright\tau}(b')$.
Let $t\in\{1,...\, , N\}$ be the index of the
interval $I_t \subseteq A$ where the block $bl(b')$ is.
Next we will select an element $a'$ from $bl(b')\subseteq I_t$ 
in order to define the domain of a  live part of a
witness structure for $(a,\varphi_i^{\exists})$ in $A$;
note that in the subcase 1, such an element was already chosen.
Now, if $a \in E$, we let $a'$ be the $i$-th element
(where $i$ is the index of $\varphi_i^{\exists}$)
realizing $\alpha'$ in $F\cap \mathit{bl}(b')$.
Similarly, if $a \in F$ (respectively, if {$a \in G\cup (C_{\mathfrak{A}}
\setminus {(K_{\mathfrak{A}}\cup D_{\mathfrak{A}} ))}$\hspace{0.0mm})}, we
choose $a'$ to be the $i$-th element in $G\cap
\mathit{bl}(b')$ (resp., in $E\cap\mathit{bl}(b')$).
Then we define $tb_{\mathfrak{A}}(a,a') := tb_{\mathfrak{B}
\upharpoonright\tau}(b,b')$, thereby possibly
creating new tuples into the relations of $\mathfrak{A}$.
Now $\{a,a'\}$ is the domain of the live part of a 
witness structure for $(a,\varphi_i^{\exists})$.
Assigning 2-tables in this cyclic way prevents conflicts, 
as each pair $(a,a') \in A^2$ is considered at most once.
We then proceed to considering the case where $b_a$
and \textit{at least two other elements} in $B$ form
the live part $\bar{\mathfrak{B}}_{b_a,\varphi_{i}^{\exists}}$ 
of the witness structure $\mathfrak{B}_{b_a, \varphi_{i}^{\exists}}$.
The sets $E,F,G\subseteq A$ defined above will play a role here as well.
\textbf{Case `\textit{local large live part}'} : 
Assume indeed that the live part $\bar{\mathfrak{B}}_{b_a,\varphi_{i}^{\exists}}$
has at least three elements, i.e., $|\bar{B}_{b_a,\varphi_{i}^{\exists}}|\geq 3$.
Let $r_{1},...\, , r_{k}$ (possibly $k = 0$)
be the elements in $\bar{B}_{b_a,\varphi_{i}^{\exists}}$
that belong also to $K^{\mathfrak{B}}$,
and let $b_a, b_{1},...\, , b_{l}$ (possibly $l=0$) be
the remaining elements of $\bar{B}_{b_a,\varphi_{i}^{\exists}}$.
As $|\bar{B}_{b_a,\varphi_{i}^{\exists}}|\geq 3$, we have $k+l \geq 2$.
Now let $j\in\{1,...\, , l\}$ and identify, in an arbitrary
way, a $3(m_{\exists}+n)$-block $bl(b_j)\subseteq A$ of
elements that realize the same $1$-type as $b_j$ does.
We let $\alpha_{j}$ be
the $1$-type of $b_{j}$, i.e., $\alpha_{j} = tp_{\mathfrak{B}
\upharpoonright\tau}(b_{j})$, and we also let $t_{b_j}\in\{1,...\, ,N\}$ denote the
index of the interval where $bl(b_j)$ is.
Then, with the blocks $bl(b_j)$ chosen for each $j$,
we move on to considering the
following subcases of the case \textit{local large live part}
in order to define a live part of a
witness structure for $(a,\varphi_i^{\exists})$ in $A$.
\begin{itemize}
\item[1.] Assume $l=0$ and $a \in C_{\mathfrak{A}}$ (whence $k\geq 2$).
We let $\{a,h(r_1),...\, , \\ h(r_k)\}$ 
(where $h$ is the bijection from $C_{\mathfrak{B}}$ to $C_{\mathfrak{A}}$ we
defined above) be the domain of the desired live part.
We note that $tb_{\mathfrak{A}}(a,h(r_1),...\, ,h(r_k))$ has already been defined
when $\mathfrak{C}_{\mathfrak{B}}$ was copied into $\mathfrak{A}$.
\item[2.] Assume $l=0$ and $a \not \in C_{\mathfrak{A}}$ (whence $k\geq 2$).
Let $\{a,h(r_1),...\, , \\ h(r_k)\}$ be the domain of the desired live part and define
$$tb_{\mathfrak{A}}(a,h(r_1),...\, ,h(r_k))
:= tb_{\mathfrak{B}\upharpoonright\tau}(b_a,r_1,...\, ,r_k).$$
Note here that the mapping $h$ is
injective and $a \not \in img(h) = C_{\mathfrak{A}}$.
\item[3.] Assume $l>0$ and $a \in E$.
We will next define elements $a_1,...\, ,a_l\in A$ corresponding to $b_1,...\, ,b_l$.
We first let $a_{1}$ be the $i$-th (where $i\leq m_{\exists}$ is the
index of $\varphi_i^{\exists}$)
element in $\mathit{bl}(b_1)\cap F$.
%
%
%
%
%
Then, if $l > 1$, we define the elements $a_2,...\, ,a_l$
to be \emph{distinct} elements such that $a_{j}$ is,
%
%
for an arbitrary $p \in \{m_{\exists}+\nolinebreak 1, ...
\, , m_{\exists}+\nolinebreak n\}$, the $p$-th element in $\mathit{bl}(b_j)\cap F$.
%
%
Note that $l<n$, so it is easy to
ensure the elements $a_2,...\, ,a_l$ are distinct even if chosen from a single block.
We let $\{a,h(r_1),...\, ,\\ h(r_k),a_{1},...\, ,a_{l}\}$
be the domain of the desired live part of a witness structure, and we define
$tb_{\mathfrak{A}}(a,h(r_1),...\, ,h(r_k),a_{1},...\, , \\ a_{l}) := 
tb_{\mathfrak{B}\upharpoonright\tau}(b_a,r_1,...\, ,r_k,b_{1},...\, ,b_{l}),$
thereby possibly creating new tuples to the relations of $\mathfrak{A}$.
\item[4.] Assume $l>0$ and $a \in F$.
Then we proceed as in the previous case,
but we take the elements $a_1,...\, ,a_l$ from $G$.
Similarly, if $l>0$ and $a \in G \cup (C_{\mathfrak{A}}
\setminus (K_{\mathfrak{A}}\cup D_{\mathfrak{A}} ))$,
we take the elements $a_1,...\, ,a_l$ from $E$.
As before, we let $\{a,h(r_1),...\, ,h(r_k),a_{1},...\, , \\ a_{l}\}$
be the domain of the desired live part of a witness structure, and we then define
$tb_{\mathfrak{A}}(a,h(r_1),...\, ,h(r_k),a_{1},...\, ,a_{l}) := \\
tb_{\mathfrak{B}\upharpoonright\tau}(b_a,r_1,...\, ,r_k,b_{1},...\, ,b_{l}),$
thus again possibly creating new tuples to relations.
\end{itemize}
We have now considered several cases and defined the live part
$\bar{\mathfrak{A}}_{a,\varphi_{i}^\exists}$
of a witness structure $\mathfrak{A}_{a,\varphi_{i}^\exists}$ in
each case (or rather a table over the elements of
the live part). We next show how to
complete the definition of $\mathfrak{A}_{a,\varphi_{i}^\exists}$ by
finding a suitable dead part for it.
We have defined $\bar{\mathfrak{A}}_{a,\varphi_i^{\exists}}$ in each case so
that there is a bijection from $\bar{B}_{b_a,\varphi_i^{\exists}}\cup\{b_a\}$
onto $\bar{A}_{a,\varphi_i^{\exists}}\cup\{a\}$;
%
%
note that $b_a$ (respectively, $a$)
may or may not be part of the live part $\bar{\mathfrak{B}}_{b_a,\varphi_i^{\exists}}$
(resp., $\bar{\mathfrak{A}}_{a,\varphi_i^{\exists}}$)
depending on whether the live part is free, and it
holds that $b_a\in\bar{{B}}_{b_a,\varphi_i^{\exists}}
\Leftrightarrow a\in\bar{{A}}_{a,\varphi_i^{\exists}}$.
The task is now to extend this bijection to a map that maps injectively
from ${B}_{b_a,\varphi_i^{\exists}}$ into $A$ and
preserves $1$-types over $\tau$. This
will complete the construction of $\mathfrak{A}_{a,\varphi_i^{\exists}}$.
This is very easy to do: Note first that since $n$ is
the width of $\varphi$, we have $|{B}_{b_a,\varphi_i^{\exists}}|\leq n$.
Now recall that in $\mathfrak{A}$,
each pawn is part of some $3(m_{\exists} + n)$-block of
elements of the same $1$-type, so there
are at least $3(m_{\exists} + n)$ elements of that type in $\mathfrak{A}$.
Furthermore, the elements of $\mathfrak{B}$ 
with a $1$-type (over $\tau$) that is royal in $\mathfrak{A}$ are
all in $K^{\mathfrak{B}}\subseteq C_{\mathfrak{B}}$,
and $\mathfrak{A}$ contains the
copy $\mathfrak{C}_{\mathfrak{A}}$ of $\mathfrak{C}_{\mathfrak{B}}$ as a substructure.
Thus it is easy to extend the bijection in the required way.
\subsection{Completion procedure}\label{completionproceduresubsection}
Let $r$ be the highest arity occurring in the symbols in $\tau$
and $n$ the width of $\varphi$. Define $m:= \mathit{min}\{r,n\}$
and $k \in \{2,...\, , m\}$.
Let $S\subseteq A$ be a set with $k$-elements.
Assume that $tb_{\mathfrak{A}}(\overline{s})$
has not been defined for any $k$-tuple $\overline{s}$
enumerating the elements of $S$
when copying $\mathfrak{C}$ into $\mathfrak{A}$ and
when finding witness structures in $\mathfrak{A}$;
thus we still need to define some $k$-table for some
tuple $\overline{s}$ that enumerates the points in $S$.
We do this next.
Assume first that $k = 2$.
%
%
Assume $S = \{a_1,a_2\}$ such that $a_1 < a_2$ 
and such that $tp_{\mathfrak{A}}(a_1) = \alpha_1$ and
$ tp_{\mathfrak{A}}(a_2) = \alpha_2$.
Let $s,t \in \{1,...\, ,N\}$ be 
the indices such that $a_1 \in I_s$ and $a_2 \in I_t.$
Due to the way we constructed
the intervals of $A$ in Section \ref{ordersection}, 
we know that $\alpha_1 \in \bbalpha_{\tau,s}$ and $\alpha_2 \in \bbalpha_{\tau,t}$.
Furthermore, as $a_1 < a_2$, we know that either
$I_s$ is an interval preceding the
interval $I_t$ and thus $s < t$, or $I_s$ and $I_t$ are the same
interval and thus $s = t$.
%
%
%
%

%
%
%
If $s < t$, then by axioms \ref{POArealizations} and \ref{POAquasiorder}, we find
from $\mathfrak{B}$ a point $b_1 \in U_s^{\mathfrak{B}}$ realizing $\alpha_1$
and  a point $b_2 \in U_t^{\mathfrak{B}}$ realizing $\alpha_2$
such that $b_1 <^{\mathfrak{B}} b_2$.
We set $tb_{\mathfrak{A}}(a_1,a_2) := tb_{
\mathfrak{B}\upharpoonright\tau}(b_1,b_2)$.
Now assume that $s = t$. We consider
the two cases where $\alpha_2\not\in\bbalpha_{\tau,s}^{+}$
and $\alpha_2\in\bbalpha_{\tau,s}^{+}$.
If $\alpha_2 \not \in \bbalpha_{\tau,s}^{+}$, then
there is some $ t' \in \{1,...\, ,N\}$ such that $s< t'$ and $\alpha_2 \in \bbalpha_{\tau,t'}$.
Thus, again by axioms \ref{POArealizations} and
\ref{POAquasiorder}, we find from $\mathfrak{B}$ a
point $b_1 \in U_{s}^{\mathfrak{B}}$ realizing $\alpha_1$
and a point $b_2 \in U_{t'}^{\mathfrak{B}}$
realizing $\alpha_2$ such that $b_1 <^{\mathfrak{B}} b_2$.
We set $tb_{\mathfrak{A}}(a_1,a_2) := tb_{\mathfrak{B}\upharpoonright\tau}(b_1,b_2)$.
Assume then that $\alpha_2 \in \bbalpha_{\tau,s}^{+}$.
We consider the two subcases where $s\not\in \mathit{img}(\delta)$
and $s\in \mathit{img}(\delta)$;
recall the definition of $\delta$ from Section \ref{admissibilitytuplessection}.
%
%
%
If $s\not\in \mathit{img}(\delta)$, then by axioms
\ref{POArealizations} and \ref{POAadmissibility4},
there is in $\mathfrak{B}$ a point $b_1 \in U_s^{\mathfrak{B}}$ realizing $\alpha_1$
and a point $b_2 \in U_s^{\mathfrak{B}}$
realizing $\alpha_2$ such that $b_1 <^{\mathfrak{B}} b_2$.
Once again we set $tb_{\mathfrak{A}}(a_1,a_2) := tb_{
\mathfrak{B}\upharpoonright\tau}(b_1,b_2)$.
If $s\in \mathit{img}(\delta)$, then,
by admissibility condition \ref{AdmissibilityCourtiers},
either $I_s$ is a singleton with an
element with a royal type or $\bbalpha_{\tau,s}^-=\emptyset =\bbalpha_{\tau,s}^+$.
If $I_s$ is a singleton, then the assumption $a_1 < a_2$ fails,
so we must have $\bbalpha_{\tau,s}^-=\emptyset =\bbalpha_{\tau,s}^+$.
Thus the assumption $\alpha_2\in\bbalpha_{\tau,s}^+$ fails, and thus
this case is in fact impossible and can thus be ignored.
%
%
%
%
%
%

%
%
%
Assume then that $k > 2$.
We select distinct elements $b_1,...\, , b_k$ in $B$
such that $tp_{\mathfrak{A}}(a_i) = tp_{\mathfrak{B}\upharpoonright\tau}(b_i)$
for each $i\in\{1,...\, ,k\}$; 
this is possible because every king 
of $\mathfrak{A}$ is in $C_{\mathfrak{A}}$ and
thus there exists a corresponding point in $C_\mathfrak{B}$, and furthermore,
by axiom \ref{POApawns},
for each pawn $u$ of $\mathfrak{A}$,
there exist at least $n\geq k$ points of the $1$-type (over $\tau$) of $u$ in $\mathfrak{B}$.
Now we set $tb_{\mathfrak{A}}(a_1,...\, ,a_k) := tb_{
\mathfrak{B}\upharpoonright\tau}(b_1,...\, ,b_k).$
Finally, if the maximum arity $r$ of relations in $\tau$ is greater than $n$,
then the tables of $\mathfrak{A}$ over sets with more than $n$ elements
are defined arbitrarily. The model $\mathfrak{A}$ is now fully defined.
To finish the proof of Lemma \ref{DecidabilityLemma2},
we argue that $\mathfrak{A}\models\varphi$.
The fact that $\mathfrak{A}$ satisfies all the existential
conjuncts of $\varphi$ was established in Section \ref{witnessstructuresection}.
To see that $\mathfrak{A}$ satisfies also the universal conjuncts,
consider such a conjunct $\forall x_1...\forall x_k \psi(x_1,...\, ,x_k)$,
and let $(a_1,...\, ,a_k)$ be a tuple of elements from $A$, with possible repetitions.
We must show that $\mathfrak{A}\models\psi(a_1,...\, ,a_k)$.
Let $\{u_1,...\, ,u_{k'}\} := \mathit{live}(\psi(x_1,...\, ,x_k)[a_1,...\, ,a_k])$, and
let $V := \{a_1,...\, ,a_k\}\setminus \{u_1,...\, ,u_{k'}\}$.
The table $tb_{\mathfrak{A}}( u_1,...\, ,u_{k'})$ has
been defined either when finding witness structures or in the
above completion construction based on
some table $tb_{\mathfrak{B}\upharpoonright\tau}( b_1,...\, ,b_{k'})$ of distinct
elements. We now observe the following.
\begin{enumerate}
\item
All the kings of $\mathfrak{A}$ are in $C_{\mathfrak{A}}$
and thereby have corresponding elements in $C_{\mathfrak{B}}$
that satisfy the same $1$-type over $\tau$.
\item
For each pawn $u$ of $\mathfrak{A}$, there exist at
least $n$ elements of the
same $1$-type over $\tau$ as $u$ in $\mathfrak{B}$ (by axiom \ref{POApawns}).
\item
The set $V\cup\{u_1,...\, ,u_{k'}\} = \{a_1,...\, ,a_k\}$ 
has at most $n$ elements.
\end{enumerate}
Based on the above, it is easy to see that we can define an
injection $f$ from $\{u_1,...\, , u_{k'}\}\cup V$ into $B$
that preserves $1$-types (over $\tau$) and
satisfies $f(u_i) = b_i$ for each $i\in\{1,...\, , k'\}$.
Therefore $\mathfrak{A}\models\psi(a_1,...\, ,a_k)$
iff $\mathfrak{B}\models\psi(f(a_1),...\, , f(a_k))$.
Since $\mathfrak{B}\models\varphi$, we
have $\mathfrak{B}\models\psi(f(a_1),...\, , f(a_k))$
and therefore $\mathfrak{A}\models\psi(a_1,...\, ,a_k)$.
\end{proof}
%
%
%

%
%
%
%
%
%

\section{Proof of Theorem \ref{nexptimethmmaintext}}\label{complexitysection}
In this section we study the complexity of the algorithm 
outlined in Figure \ref{pseudocodefig} (in Section \ref{reducing})
and establish that it runs in \textsc{NExpTime} in
all cases $\mathcal{K}\in\{\mathcal{O},\mathcal{WO},\mathcal{O}_{fin}\}$.
We now fix some $\mathcal{K}\in\{\mathcal{O},\mathcal{WO},\mathcal{O}_{fin}\}$
and study only the algorithm for the class $\mathcal{K}$; below we
call the algorithm \emph{Algorithm 1}.
Let $\psi'$ be a $\mathrm{U}_1$-sentence given as
an input to Algorithm \ref{pseudocodefig}.
It follows from Proposition \ref{PreliminariesNormalForm} that
$\psi'$ can be translated in polynomial time in $|\psi'|$ to a
normal form sentence $\psi$
such that $\psi'$ is satisfiable in
some model  $\mathfrak{M}\in\mathcal{K}$ iff $\psi$ is
satisfiable in some expansion $\mathfrak{M}^*\in\mathcal{K}$ of $\mathfrak{M}$.
The formula $\psi$ is the normal form sentence of $\mathrm{U}_1$
constructed at line \ref{normalform} of Algorithm \ref{pseudocodefig}.
Let $\tau$ be the vocabulary consisting of the relation
symbols in $\psi$. We assume w.l.o.g. that $<\in\tau$.
At line 3 we guess some $\Gamma_{\psi} \in \hat{\Gamma}_{\psi}$
and check that $\Gamma_{\psi}$ is indeed an admissibility tuple
admissible for $\mathcal{K}$. The length of $\Gamma_{\psi}$ is
bounded exponentially in $|\psi|$ by Lemma
\ref{th!MostImportantProofInYourLife}, and checking admissibility of $\Gamma_{\psi}$ 
for $\mathcal{K}$ can be done in polynomial time in $|\Gamma_{\psi}|$.
%
%
%
%
%
%
%

%
%
%
At line 4 we let $\tau'$ be the vocabulary
$\tau \cup \{U_s\mid s\in\{1,...\, , N\} \} \cup \{K,D,P_{\bot}, P_{\top}\}$
and formulate the conjunction $\mathit{Ax}(\Gamma_{\psi})$ of the
pseudo-ordering axioms for $\Gamma_{\psi}$ over $\tau'$.
%
%
%
%

%
\begin{lemma}\label{sentencelimitlemma}
Consider a normal form sentence $\chi$ of $\mathrm{U}_1$ and a
related admissibility tuple.
The size of the sentence $\mathit{Ax}(\Gamma_{\chi})$ is 
exponentially bounded in $|\chi|$.
\end{lemma}
\begin{proof}
Let $N$ be the index of $\Gamma_{\chi}$
and $C$ the domain of the court structure of $\Gamma_{\chi}$.
Let $\sigma$ be the
vocabulary of $\chi$.
%
Now let $\beta$ be some axiom from the list of 16 axioms that
make $\mathit{Ax}(\Gamma_{\chi})$, see Section \ref{pseudo-orderingaxiomssect}.
The sentence $\beta$ is a normal form sentence with
some number $m_{\exists,\beta}$ of
existential conjuncts and some number $m_{\forall,\beta}$ of universal conjuncts.
Now, by inspection of the pseudo-ordering axioms,
the sum $m_{\exists,\beta}+m_{\forall,\beta}$ is bounded above by
the very generous\footnote{We shall not
seek minimal or in any sense canonical bounds.
Instead we settle with "clearly sufficient" bounds.
This applies here as well as later on.} 
bound $\mathit{const}\cdot |\chi|\cdot N^2
\cdot |\bbalpha_{\sigma}|^2\cdot | C |^{|\chi|} + \mathit{const}$
for some constant \emph{const}.
%
%
Recalling from Section \ref{admissibilitytuplessection}
that $|C|\leq 2|\chi|^4|
\bbalpha_{\sigma}|$ and $N\leq 6|\chi|^4|\bbalpha_{\sigma}|$,
we get that ${m}_{\exists,\beta}+{m}_{\forall,\beta}$ is
bounded by $\mathit{const}\cdot |\chi|
\cdot (6|\chi|^4|\bbalpha_{\sigma}|)^2
\cdot |\bbalpha_{\sigma}|^2\cdot (2 |\chi|^4|\bbalpha_{\tau}| )^{|\chi|}
+\mathit{const}$. 
Since $|\bbalpha_{\sigma}|\leq 2^{|\chi|}$, it is therefore easy to see
%
%
that this bound is exponential in $|\chi|$.
%
%
%
%
%
Therefore, to conclude our proof, it suffices to
find some bound $\mathcal{B}$ exponential in $|\chi|$ 
such that the length of each existential conjunct as well as
the length of each universal conjunct in $\mathit{Ax}(\Gamma_{\chi})$ is
bounded above by $\mathcal{B}$.
%

%
To find such a bound $\mathcal{B}$, we first investigate axiom \ref{POAcourt}.
We note that each formula $\beta_{[c_1,...\, ,c_k]}(x_1,...\, , x_k)$
in axiom \ref{POAcourt} is a $k$-table and
therefore consists of a conjunction over a
set such as---to give a possible example---the one given in
Example \ref{rubbishexample}.
The \emph{number of conjuncts} in $\beta_{[c_1,...\, ,c_k]}(x_1,...\, , x_k)$ is
therefore definitely bounded above by
the bound ${|\chi|\cdot|\chi|^{|\chi|}}$.
Thus it is easy to see 
that there exists a term $\mathcal{B}_{(11)}$ exponential in $|\chi|$
such that the length of each universal conjunct of axiom \ref{POAcourt} is 
bounded above by $\mathcal{B}_{(11)}$.
To cover the existential and universal conjuncts in
the other axioms, we investigate 
each axiom individually and easily conclude that 
there exists a term $\mathcal{B}_{(i)}$ for each axiom $i\in\{1,...\, , 16\}$
such that the length of each existential and universal conjunct in
the axiom $(i)$ is bounded above by $\mathcal{B}_i$,
and furthermore, $\mathcal{B}_i$ is exponential in $|\chi|$.
By taking the product of the terms $\mathcal{B}_{(i)}$, we
find a uniform exponential bound for the length of all
existential and universal conjuncts in $\mathit{Ax}(\Gamma_{\chi})$.
\end{proof}
At line 5 of Algorithm 1 we guess a $\tau'$-model $\mathfrak{B}$
whose domain size is exponential in $|\psi|$
(rather than exponential in $|\mathit{Ax}(\Gamma_{\psi})|$); a
sufficient bound is established below (Lemma \ref{ComplexityLemma1}), and
furthermore, it is shown that not only the domain size but even the
full description of $\mathfrak{B}$ can be bounded
exponential in $|\psi|$.
(Recall that $\mathfrak{B}$ does not have to interpret the
binary relation symbol $<$ as and order.)
%
%
%
We now begin the process of finding an
exponential upper bound (in $|\psi|$) for the size of $\mathfrak{B}$
and show that this bound is indeed sufficient.
We also establish that, indeed,
the full description of $\mathfrak{B}$ likewise has a bound
exponential in $|\psi|$.
To achieve these goals, we first analyze below the proof of
Theorem \ref{ComplexityTheorem1}; this theorem is Theorem 2 in the
article \cite{DBLP:conf/mfcs/KieronskiK14}
(and Theorem 3.4 in \cite{kieroarxiv} due to different numbering).
The original proof is given in detail in Section 3 of both
\cite{DBLP:conf/mfcs/KieronskiK14} and \cite{kieroarxiv}.
We state the theorem exactly as in
\cite{DBLP:conf/mfcs/KieronskiK14} and \cite{kieroarxiv}, and
thus note that $\mathrm{UF}_{1}^{=}$
denotes $\mathrm{U}_{1}$ in the theorem.
(Note that obviously the theorem concerns general $\mathrm{U}_1$ as
opposed to $\mathrm{U}_1$ over ordered structures.)
\begin{theorem}[\cite{DBLP:conf/mfcs/KieronskiK14}]
\label{ComplexityTheorem1}
$\mathrm{UF}_{1}^{=}$ has the finite model property. 
Moreover, every satisfiable $\mathrm{UF}_{1}^{=}$-formula $\varphi$ has 
a model whose size is bounded exponentially in $|\varphi|$. 
\end{theorem}
It follows from Theorem \ref{ComplexityTheorem1}
that $\mathit{Ax}(\Gamma_{\psi})$ has a model $\mathfrak{M}$ whose
size is exponential in $|\mathit{Ax}(\Gamma_{\psi})|$,
but since $|\mathit{Ax}(\Gamma_{\psi})|$ is exponential in $|\psi|$,
the size of the model $\mathfrak{M}$ is double exponential in $|\psi|$.
This is not the desired result.
%
%
%
To lower the bound to exponential,
we now analyze the proof of
Theorem \ref{ComplexityTheorem1} given in
Section 3 of \cite{DBLP:conf/mfcs/KieronskiK14} and \cite{kieroarxiv}.
This will result in the following lemma which
follows directly and very easily from 
\cite{DBLP:conf/mfcs/KieronskiK14, kieroarxiv} but is implicit there, i.e.,
not stated as an explicit lemma.
Recall here that $\bbalpha_{\mathfrak{A}}$ denotes the $1$-types 
realized in $\mathfrak{A}$.
\begin{lemma}\label{ComplexityCorollary1}
Let $\varphi$ be a normal form sentence of\, $\mathrm{U}_{1}$.
Let $m_{\exists}> 0$ be the
number of existential conjuncts in $\varphi$.
Let $n\geq 2$ be the width of $\varphi$ and $\sigma$ the vocabulary of $\varphi$.
If $\varphi$ is satisfiable, then it is satisfiable in some model $\mathfrak{M}$
such that $|M|\leq 8m_{\exists}^2n^2\bbalpha_{\mathfrak{M}}$
where $\bbalpha_{\mathfrak{M}}\subseteq\bbalpha_{\sigma}$.
\begin{proof}
Let $\varphi$, $n\geq 2$, $\sigma$ and $m_{\exists}\not= 0$ be as specified above.
%
%
Assume $\varphi$ is satisfiable.
The claim of the current lemma follows
directly by inspection of the relatively short argument in
Section 3 of \cite{DBLP:conf/mfcs/KieronskiK14,kieroarxiv},
but we shall anyway outline here why
there exists a model $\mathfrak{M}$ with the given limit 
$8m_{\exists}^2 n^2\bbalpha_{\mathfrak{M}}$ on domain size.
%
%
%
%
%
%
%

%
Assume $\mathfrak{A}$ is a $\sigma$-model
such that $\mathfrak{A}\models\varphi$.
The original proof constructs from the $\sigma$-model\footnote{The vocabulary
used in the original proof is
denoted by $\tau$ instead of $\sigma$. We use $\sigma$
here because $\tau$ is in 'global' use by Algorithm 1.} $\mathfrak{A}$ of $\varphi$ a
new $\sigma$-model
$\mathfrak{A}'$ whose domain $A'$ 
consists of the union of four sets $C,E,F,G$, 
where the set $C$ is constructed with the help of two sets $K$ and $D$.
Now, while it is stated in \cite{DBLP:conf/mfcs/KieronskiK14}
that $|K|\leq (n-1)|\bbalpha_{\sigma}|$
and $|D|\leq (n-1)m_{\exists}|\bbalpha_{\sigma}|$, it is
straightforward to observe that in fact
$|K|\leq (n-1)|\bbalpha_{\mathfrak{A}}|$
and $|D|\leq (n-1)m_{\exists}|\bbalpha_{\mathfrak{A}}|$.
(Note that we use $\bbalpha_{\mathfrak{A}}$
instead of $\bbalpha_{\mathfrak{A}'}$ here.)
It is also easily seen that $|C|\leq n |K \cup D| m_{\exists}$,
and thus we can calculate, using the above bounds for $K$ and $D$, that
\begin{align*}
C\leq\, &n|K\cup D|m_{\exists} 
\leq\,n( (n-1)|\bblet{\alpha}_{\mathfrak{A}}| + (n-1)m_{\exists}|
\bblet{\alpha}_{\mathfrak{A}}|)m_{\exists} \\
\leq\,& (n^2|\bblet{\alpha}_{\mathfrak{A}}| + 
n^2m_{\exists}|\bblet{\alpha}_{\mathfrak{A}}|)m_{\exists}
%
\leq\,2n^2m_{\exists}^2|\bblet{\alpha}_{\mathfrak{A}}|.
\end{align*}
We then consider the sets $E,F,G$.
The article \cite{DBLP:conf/mfcs/KieronskiK14} gives a
bound $(n + m_{\exists})|\bbalpha_{\sigma}|$ for each of 
these sets, but it is immediate 
that in fact $(n + m_{\exists})|\bbalpha_{\mathfrak{A}}|$ suffices.
Putting all the above together, we calculate
\begin{align*}
|C \cup E \cup F \cup E| 
&\leq\,2n^2m_{\exists}^2|\bblet{\alpha}_{\mathfrak{A}}|
+3(n+m_{\exists})|\bblet{\alpha}_{\mathfrak{A}}| 
\\ &\leq \, 8n^2m_{\exists}^2|\bblet{\alpha}_{\mathfrak{A}}|.
\end{align*}
%
%
%
%
%
%
%
%
%
%
%
%
%
%
%
%
%
%
%
It is also immediate that $\bbalpha_{
\mathfrak{A}'}\subseteq\bbalpha_{\mathfrak{A}}$, so
the domain of $\mathfrak{A'}$,
i.e., the set $C \cup E \cup F \cup E$, is
bounded above by $8n^2m_{\exists}^2|\bblet{\alpha}_{\mathfrak{A}'}|$. 
\end{proof}
\end{lemma}
%

%
%
%
%

%
\begin{lemma}\label{ComplexityLemma1}
%
%
%
%
%
%
Let $\Gamma_{\varphi}\in \hat{\Gamma}_{\varphi}$ be some
tuple admissible for $\mathcal{K}\in\{\mathcal{O},
\mathcal{WO},\\ \mathcal{O}_{fin}\}$ such
that $\mathit{Ax}(\Gamma_{\varphi})$ is satisfiable.
%
%
%
%
%
%
Then $\mathit{Ax}(\Gamma_{\varphi})$ has a model $\mathfrak{A}$ 
whose size is bounded exponentially in $|\varphi|$. 
Moreover, even the length of the description of $\mathfrak{A}$ is
bounded exponentially in $|\varphi|$.

\begin{proof}
Let $\sigma$ be the vocabulary of $\varphi$.
Let $n$ be the width of $\varphi$ and $m_{\exists}$ the
number of existential conjuncts in $\varphi$.
Let $N$ be the index of $\Gamma_{\varphi}$
and $\sigma' := \sigma \cup \{U_{s}\mid 1\leq s \leq N\}
\cup \{K, D, P_{\bot}, P_{\top}\}$
the vocabulary of $\mathit{Ax}(\Gamma_{\varphi})$.
Let $C$ be the domain of the court structure in $\Gamma_{\varphi}$.
Assume $\mathfrak{M}\models\mathit{Ax}(\Gamma_{\varphi})$.
Recalling from Section \ref{admissibilitytuplessection}
that $N\leq 2|\varphi|^4|\bbalpha_{\sigma}|$
and thus clearly $N\leq 2|\varphi|^4\cdot 2^{|\varphi|}$,
%
%
%
%
%
we have

\smallskip

${|\sigma'|} = |\sigma| + |\{U_{s}\mid 1\leq s \leq N\}|
+ |\{K, D, P_{\bot}, P_{\top}\}| = {|\sigma| + N + 4}
\leq {|\varphi|+ 2|\varphi|^4\cdot 2^{|\varphi|} + 4}$

\smallskip

Thus $|\bbalpha_{\sigma'}|$ is bounded by $2^{{|\varphi|+ 2|\varphi|^4
\cdot 2^{|\varphi|} + 4}}$.
This is double exponential in $|\varphi|$.
However, the upper bound for $|\bbalpha_{\mathfrak{M}}|$
(i.e., the number of $1$-types over $\sigma'$ realized in $\mathfrak{M}$) is
exponentially bounded in $|\varphi|$ for the following reason.
Since the predicates $U_s$, where $s \in \{1,...\, , N\}$,
partition the domain $M$,
each element in $M$ satisfies exactly one
of the predicates $U_s$.
%
%
%
%
%
%
%
%
%
%
%
Therefore, letting $\sigma'' := \sigma'\setminus \{U_s\, |\, 1\leq s\leq N\}$,
we have $|\bblet{\alpha}_{\mathfrak{M}}|
\leq N|\bblet{\alpha}_{\sigma''}|$.
%
%
On the other hand, $|\bblet{\alpha}_{\sigma''}|
\leq 2^{|\sigma| + 4}
\leq 2^{|\varphi| + 4}$.
%
%
%
Combining these, we obtain that $|\bblet{\alpha}_{\mathfrak{M}}|
\leq  N\cdot 2^{|\varphi| + 4}$.
%
%
Recalling (from a few lines 
above) that $N\leq 2 |\varphi|^{4}\cdot 2^{|\varphi|}$, we
get $|\bblet{\alpha}_{\mathfrak{M}}|
\leq  2 |\varphi|^{4}\cdot 2^{|\varphi|}\cdot 2^{|\varphi| + 4}
= 2 |\varphi|^{4}\cdot 2^{2|\varphi| + 4}$.
This is exponential in $|\varphi|$.
As $\mathit{Ax}(\Gamma_{\varphi})$ is satisfiable,
it follows from Lemma \ref{ComplexityCorollary1}
that $\mathfrak{A}\models\mathit{Ax}(\Gamma_{\varphi})$ for
some $\sigma'$-structure $\mathfrak{A}$ whose size is
bounded by $\\ 8\hat{m}_{\exists}{\hat{n}}^2|\bbalpha_{\mathfrak{A}}|$,
where $\hat{m}_{\exists}$ is the number of
existential conjuncts in $\mathit{Ax}(\Gamma_{\varphi})$ and $\hat{n}$ is 
the width of $\mathit{Ax}(\Gamma_{\varphi})$.
On the other hand, by the result from the previous paragraph, we
have $|\bbalpha_{\mathfrak{A}}| \leq 2 |\varphi|^{4}
\cdot 2^{2|\varphi| + 4}$.
Therefore, to show that the domain of $\mathfrak{A}$ is 
bounded exponentially in $|\varphi|$, it
suffices to show that $\hat{m}_{\exists}$ and $\hat{n}$ are
exponentially bounded in $|\varphi|$.
This follows immediately by Lemma \ref{sentencelimitlemma}.
We then show that even the length of the description of $\mathfrak{A}$ is,
likewise, exponentially bounded in $|\varphi|$.
For describing models, we use the straightforward convention from
Chapter 6 of \cite{DBLP:books/sp/Libkin04}, according to which
the unique description of $\mathfrak{A}$ with some
ordering of $\sigma'$ is of the length
%
%
%
%
$|A|+1 +\sum_{i=1}^{|\sigma'|}\, |A|^{ar(R_i)}$
%
%
%
%
where $ar(R_i)$ is the arity of $R_i \in \sigma'$.
Since $|A|$ is exponential in $|\varphi|$
and $\mathit{ar}(R_i)\leq |\varphi|$, each term $|A|^{\mathit{ar}(R_i)}$ is
likewise exponentially bounded in $|\varphi|$.
Furthermore, at the beginning of the 
current proof we calculated that $|\sigma'|\leq
{|\varphi|+ 2|\varphi|^4\cdot 2^{|\varphi|} + 4}$.
Thus we conclude that the description of $\mathfrak{A}$
exponentially bounded in $|\varphi|$.
\end{proof}
\end{lemma}
Once we have guessed the exponentially
bounded model $\mathfrak{B}$ at line 5 of Algorithm 1,
the remaining part of the algorithm is devoted for
checking that $\mathfrak{B}\models\mathit{Ax}(\Gamma_{\psi})$.
At lines 6-11 we scan each $b\in B$ and each existential
conjunct of $\mathit{Ax}(\Gamma_{\psi})$.
%
%
Then at lines 12-16 we check the universal conjuncts by
checking all tuples of length at most $n'$ in $B$, where $n'$ is
the width of $\mathit{Ax}(\Gamma_{\psi})$. Noting that $n'\leq n + 1$,
where $n$ is the width of $\psi$, the procedure at lines 5-16 can be carried out in
exponential time in $|\psi|$.
We have now proved the following theorem, which is a
restatement of Theorem \ref{nexptimethmmaintext}.
(Recall that the lower bound is
obtained because $\mathrm{FO}^2$ is \textsc{NExpTime}-complete for all the 
classes $\mathcal{K}\in\{\mathcal{O},\mathcal{WO},\mathcal{O}_{fin}\}$
\cite{DBLP:journals/jsyml/Otto01}.)
\begin{theorem}\label{nexptimethmappendix} \emph{(Restatement of Theorem
\ref{nexptimethmmaintext}):}
\newline Let $\mathcal{K}\in\{\mathcal{O}, \mathcal{WO},\mathcal{O}_{fin}\}$.
The satisfiability problem for $\mathrm{U}_1$ over $\mathcal{K}$
is \textsc{NExpTime}-complete.
%
%
%
%
%
%
%
%
%
%
%
%
%
%
%
%
%

%
%
%
%
%
%
\end{theorem}

\section{Proof of Theorem  \ref{U1twosucc}} \label{appedixtiling}

Before giving the proof, we introduce some definitions and lemmas used
in the proof. 

A \textit{domino system} $\mathcal{D}$ is a structure $(D, H_{do}, V_{do})$, where
$D$ is a finite set (of dominoes) and $H_{do},V_{do} \subseteq D\times D$.
We say that a mapping $\tau: \mathbb{N}\times \mathbb{N} \rightarrow D$ 
is a $\mathcal{D}$-\textit{tiling} of $\mathbb{N} \times \mathbb{N}$,
if for every $i,j \in \mathbb{N}$, it holds that 
$(\tau(i,j),\tau(i+1,j)) \in H_{do}$ and
$(\tau(i,j),\tau(i,j+1)) \in V_{do}$.
The \textit{tiling problem} asks, given a domino system $\mathcal{D}$ as an input, whether 
there exists a $\mathcal{D}$-tiling of $\mathbb{N}\times \mathbb{N}$.
It is well known that the tiling problem  is undecidable.

Let $\mathfrak{G}_{\mathbb{N}} = (\mathbb{N \times N}, H, V)$ be the \emph{standard grid},  
where $H = \{ \big( (i,j), \\ (i+1,j) \big) \mid i,j \in \mathbb{N} \}$ 
and $V = \{ \big( (i,j), (i,j+1) \big) \mid i,j \in \mathbb{N} \}$ are binary relations. 

Let $\mathfrak{A} = (A,H,V)$ and $\mathfrak{B} = (B,H,V)$ be $\{H,V\}$-structures,
where $H$ and $V$ are binary relations. 
The structure $\mathfrak{A}$ is \textit{homomorphically embeddable} into $\mathfrak{B}$, 
if there is a homomorphism $h: A \rightarrow B$ 
defined in the usual  way.
\begin{definition} \label{gridlikenessdef}
A structure 
$\mathfrak{G} = (G,H,V)$ is called \textit{grid-like}, 
if there  exists a homomorphism 
from $\mathfrak{G}_{\mathbb{N}}$ to $\mathfrak{G}$, i.e., $\mathfrak{G}_{\mathbb{N}}$ 
is homomorphically embeddable into $\mathfrak{G}$. 
\end{definition}
Let $\mathfrak{G}$ be a $\{H,V\}$-structure with two binary relations $H$ and $V$.
We say that $H$ is \textit{complete over} $V$, 
if $\mathfrak{G}$ satisfies the formula $\forall xyzt (\,(\,Hxy \wedge Vxt \wedge Vyz\,) \rightarrow Htz\,)$. 

The following lemma is from \cite{DBLP:journals/jsyml/Otto01}. 
Note that $\mathrm{FO^2}$ is contained in $\mathrm{U}_{1}$. 
\begin{lemma}[\cite{DBLP:journals/jsyml/Otto01}]\label{gridlikelemma1}
Let $\mathfrak{G} = (G,H,V)$ be a structure satisfying 
the $\mathrm{FO^2}$-axiom $\forall x (\, \exists y  Hxy \wedge \exists y  Vxy\,)$.
If H is complete over V, then $\mathfrak{G}$ is grid-like. 
\end{lemma}

Let $\mathcal{D}$ be a domino system, and 
let $(P_d)_{d \in D}$ be a set of unary relation symbols.
Assume that there is a $\mathcal{D}$-tiling of $\mathbb{N}\times \mathbb{N}$.
The correctness of the $\mathcal{D}$-tiling can be expressed by the $\mathrm{FO^2}$-sentence
$\varphi_{\mathcal{D}} := \forall x ( \bigvee_{d} P_{d}x \wedge \bigwedge_{d \neq d'} \neg( P_{d}x \wedge P_{d'}x )\, )
\wedge \forall x y (\, Hxy \rightarrow \\ \bigvee_{(d,d')\in H_{do}} (\, P_{d}x \wedge P_{d'}y\,) \,)
\wedge \forall x y (\, Vxy \rightarrow \bigvee_{(d,d')\in V_{do}} (\, P_{d}x \wedge P_{d'}y\,) \,)$.

\begin{lemma} \label{dominoreduction}
Let $\mathcal{D}$ be a domino system, and let $\mathcal{G}$ be a class 
of grid-like structures such that $\mathfrak{G}_{\mathbb{N}} \in \mathcal{G}$.
Then there exists a $\mathcal{D}$-\textit{tiling} of $\mathbb{N} \times \mathbb{N}$ iff 
there exists $\mathfrak{G} \in \mathcal{G}$ that can be expanded to $\mathfrak{G}' = (G,H,V, (P_d)_{d\in D})$
such that $\mathfrak{G}' \models \varphi_{ \mathcal{D} }$.
\begin{proof}
Assume first that there exists a $\mathcal{D}$-\textit{tiling} of $\mathbb{N} \times \mathbb{N}$.
Then, as $\mathfrak{G}_{\mathbb{N}} \in \mathcal{G}$,  we expand $\mathfrak{G}_{\mathbb{N}}$
to $\mathfrak{G}'_{\mathbb{N} } = (\mathbb{N}\times \mathbb{N}, H,V, (P_d)_{d\in D})$
in the obvious way, whence $\mathfrak{G}'_{\mathbb{N} } \models \varphi_{\mathcal{D}}$.

Assume then that there exists $\mathfrak{G} \in \mathcal{G}$ 
that can be expanded to $\mathfrak{G}' = (G,H,V, (P_d)_{d\in D})$
such that $\mathfrak{G}' \models \varphi_{ \mathcal{D} }$.
As $\mathfrak{G}$ is grid-like, it follows from Definition \ref{gridlikenessdef} that 
there is a homomorphism $h: \mathfrak{G}_{\mathbb{N}} \rightarrow \mathfrak{G}$.
We define $\tau: \mathbb{N}\times \mathbb{N} \rightarrow D$ such that $\tau(i,j) = d$, if $h(i,j) \in P_{d}$ for some $d \in D$.
Now the mapping $\tau$ is a $\mathcal{D}$-\textit{tiling} of $\mathbb{N} \times \mathbb{N}$.
\end{proof}
\end{lemma}

\begin{proof}[Proof of Theorem  \ref{U1twosucc}]
Let $\tau = \{H,V\}$.
Recall that the standard grid 
$\mathfrak{G}_{\mathbb{N}}$ is a $\tau$-structure.
Let $\tau' = \tau \cup \{<_1,<_2,N\}$,
where $<_1$ and $<_2$
are binary symbols and $N$ is a 4-ary symbol.
Let us first informally outline the proof.  
First the standard grid $\mathfrak{G}_{\mathbb{N}}$ is expanded
to $\tau'$-structure $\mathfrak{G'}_{\mathbb{N}}$.  
Expanding $\mathfrak{G}_{\mathbb{N}}$
to $\mathfrak{G'}_{\mathbb{N}}$ amounts to
describing how the new symbols $<_1$, $<_2$, and $N$ are 
interpreted in $\mathfrak{G'}_{\mathbb{N}}$.
A fragment of the intended structure can be seen in \fn \ref{picture1}.
Then we axiomatize some important properties of $\mathfrak{G'}_{\mathbb{N}}$
such that the structures that interpret $<_1$ and $<_2$ as linear orders and satisfy the axioms, resemble 
$\mathfrak{G'}_{\mathbb{N}}$ closely enough.
Now, let $\mathcal{G}$ be the class of $\tau$-reducts of  
$\tau'$-structures that interpret $<_1$ and $<_2$ as linear orders and satisfy the axioms. 
In particular, $\mathfrak{G}_{\mathbb{N}}$ is in $\mathcal{G}$.
We show that every structure in $\mathcal{G}$ satisfies 
the local criterion that $H$ is complete over $V$. 
It will then follow from Lemma \ref{gridlikelemma1} 
that every structure in $\mathcal{G}$ is grid-like. 
Then the undecidability of the general satisfiability problem for $\mathrm{U}_{1}[<_{1},<_{2}]$ 
follows from Lemma \ref{dominoreduction}.

We now go to the details of the proof.
We define the $\tau'$-expansion $\mathfrak{G'}_{\mathbb{N}}$  of
$\mathfrak{G}_{\mathbb{N}}$ as follows.
The linear order $<_1$ follows a lexicographical order 
such that for all $(i,j), (i',j') \in \mathbb{N}^2$,  
we have $(i,j) <_{1} (i',j')$ if and only if $j<j'$ or ($j=j'$ and $i<i'$). 
In the linear order $<_2$, the roles of $i$ and $j$ are swapped, 
i.e., for all $(i,j), (i',j') \in \mathbb{N}^2$,  
we have $(i,j) <_{2} (i',j')$ if and only if $i<i'$ or ($i=i'$ and $j<j'$). 

The symbol $N$ is defined as follows.
For all points $a,b,c,d$ in $\mathbb{N}^2$,
we have $Nabdc$ if and only if $Hab$, $Hcd$, $Vac$, and $Vbd$; see \fn \ref{picture1}. 

Next we define a few auxiliary formulae.
For $i\in\{1,2\}$, let $x \leq_{i} y := x = y \vee x <_{i} y$.
Define also $\sigma_{i}(x,y,z) := x
<_{i}y \wedge (z\leq_{i} x \vee y \leq_{i} z)$.
We are now ready to give the desired axioms defining a class of $\tau'$-structures.
Let $\eta$ be the conjunction of the following sentences.
\begin{description}
  \item[$\eta_{G}$]$ = \forall x (\, \exists y Hxy \wedge \exists y Vxy\, ).$

  \item[$\eta_{H}$]$ = \forall x y z \big( \,Hxy \rightarrow
  \sigma_1(x,y,z) \,\big).$
  Together with the previous axiom, this axiom forces $H$ to be a 
kind of an
"induced successor relation" of the linear order $<_1$. 
  It is worth noting that $H$ is subject to the uniformity condition of $\mathrm{U}_1$, i.e.,
  $H$ cannot be used freely in quantifier-free $U_1[<_1,<_2]$-formulae,
but the order symbols $<_1,<_2$ can.
 
  \item[$\eta_{V}$]$ = \forall x y z \big(\, Vxy \rightarrow
  \sigma_2(x,y,z) \,\big).$ 
  This is analogous to $\eta_{H}$.

  \item[$\eta_{N\exists}$] $ = \forall x \exists y z t (\, Nxyzt \,).$
  This axiom states that each point is a first coordinate in some 4-tuple in $N$. We call the
 4-tuples in $N$ \emph{quasi-squares}.

  \item[$\eta_{N\forall}$]$ = \forall x y z t u 
  \big( \,Nxyzt   \rightarrow (\sigma_1(x,y,u) \wedge \sigma_2(x,t,u) 
  \wedge \sigma_2(y,z,u) \wedge \sigma_1(t,z,u )) \,\big).$
  The points of the quasi-squares are connected via 
  the induced successors of $<_1$ and $<_2$; 
  see the green curves representing tuples in $N$, Figure \ref{picture1}.
\end{description}

Thus we have $\eta := \eta_{G} \wedge \eta_{H} \wedge \eta_{V} \wedge 
\eta_{N_\exists} \wedge \eta_{N_\forall}$. 
It is readily checked that the expansion $\mathfrak{G'}_\mathbb{N}$ of 
the standard grid $\mathfrak{G}_\mathbb{N}$ satisfies the sentence $\eta$. 
Let $\mathcal{G} = \{\mathfrak{G'}\upharpoonright \tau \mid 
\mathfrak{G'} \text{ is } \tau' \text{-model s.t. $<_{1}^\mathfrak{G'}$ and }$ 
\\ $<_{2}^\mathfrak{G'} \text{ are linear orders and }  
\mathfrak{G'} \models \eta \}$.
Next we need to show that every structure $\mathfrak{G} 
\in \mathcal{G}$ is grid-like.
This can be done by applying Lemma \ref{gridlikelemma1}:
as every structure $\mathfrak{G} \in \mathcal{G}$ satisfies $\eta_{G}$, 
it suffices to show that for every structure $\mathfrak{G} \in \mathcal{G}$, $H$ is complete over $V$.
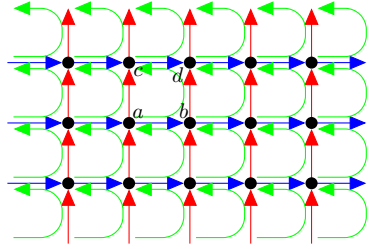
\begin{figure}[h!] 
\centering
\begin{tikzpicture}[scale=0.8, every node/.style={scale=0.8}] 

\foreach \i in {0,1,2} {
  \foreach \j in {-1,0,1,2,3,4} {
  \draw[triangle 45-,blue] (\j+0.9,\i)--(\j,\i);
  }
}

\foreach \i in {-1,0,1,2} {
  \foreach \j in {0,1,2,3,4} {
  \draw[-triangle 45,red] (\j,\i)--(\j,\i+0.9);
  }
}

\foreach \i in {2,1,0} {
  \foreach \j in {0,1,2,3,4} {
    \draw (\j,\i) node [circle, fill, inner sep=2pt] {};
  } 
}

\foreach \i in {2,1,0,-1} {
  \foreach \j in {-1,0,1,2,3,4} {
    \draw[triangle 45-,green, rounded corners=8pt] 
     (\j+0.1,\i+0.9) -- (\j+0.9,\i+0.9) -- (\j+0.9,\i+0.1) -- (\j+0.1,\i+0.1);
  }
}

\node at (1.15,1.15) {$a$};
\node at (1.9,1.2) {$b$};
\node at (1.15,1.85) {$c$};
\node at (1.8,1.8) {$d$};

\end{tikzpicture}
\caption{A finite fragment of the intended structure. 
The blue arrows represent the H-relations 
and the red ones the V-relations. 
The green curves represent the $N$-relations,
e.g. $Nabdc$.}
\label{picture1}
\end{figure}

To show that $H$ is complete over $V$ in every structure in $\mathcal{G}$, 
let $\mathfrak{G'}$ be a $\tau'$-structure interpreting 
$<_1$ and $<_2$ as linear orders and satisfying $\eta$. 
For convenience, for $i \in \{1,2\}$, let $\beta_{i}(x,y) := \forall z\, (\sigma_{i}(x,y,z))$.
%
%

Let $a \in G'$. 
From $\eta_G$, we get points $b,c,d \in G'$ 
such that $Hab \wedge Vac \wedge Vbd$.
As  $Hab \wedge Vac \wedge Vbd$, 
we conclude that $\beta_{1}(a,b) \wedge \beta_{2}(a,c) \wedge \beta_{2}(b,d)$ 
from $\eta_H$ and $\eta_V$.
From $\eta_{N\exists}$, we get $Nab'd'c'$ for some $b',c',d' \in G'$. 
As $Nab'd'c'$, we conclude that
$\beta_{1}(a,b') \wedge \beta_{2}(a,c') \wedge  \beta_{2}(b',d') \wedge \beta_{1}(c',d')$
from $\eta_{N\forall}$. The following claim is clear.

\textbf{Claim.} If $\beta_{1}(a,b) \wedge \beta_{1}(a,b')$, then $b = b'$.

As $\beta_{1}(a,b) \wedge \beta_{1}(a,b')$,
it follows from the claim that $b=b'$. 
We then conclude similarly that $c=c'$ and $d=d'$ (recalling that $b=b'$). 
From $\eta_G$, we get a point $d'' \in G'$ such that $Hcd''$
and then conclude that $\beta_{1}(c,d'')$ from $\eta_H$.
Furthermore, as $\beta_{1}(c',d')$, $c=c'$ and $d=d'$, we
have $\beta_{1}(c,d) \wedge \beta_{1}(c,d'')$.
Now, analogously to the claim, we have $d=d''$.
Therefore, as $Hcd''$, we have $Hcd$.

Let $\mathfrak{G} := {\mathfrak{G'}\upharpoonright \tau}$.
Thus for $\mathfrak{G} \in \mathcal{G}$,
it holds that $H$ is complete over $V$.
Now it follows from Lemma \ref{gridlikelemma1} 
that $\mathfrak{G}$ is grid-like.

As $\mathfrak{G}'_{\mathbb{N}} \models \eta$, 
the standard grid $\mathfrak{G}_{\mathbb{N}}$ is also in $\mathcal{G}$.
It now follows from Lemma \ref{dominoreduction} that 
the (general) satisfiability problem 
for $\mathrm{U}_1[<_1,<_2]$ over structures with linear orders $<_1$
and $<_2$ is undecidable. 
\end{proof}

\end{document}